\documentclass[9pt,a4paper]{amsart}
\usepackage{verbatim}

\usepackage[toc]{appendix}
\usepackage[T1]{fontenc}

\usepackage{graphicx}
\usepackage{enumerate}
\usepackage{amsmath,amsfonts,amssymb}
\usepackage{color}
\def\loc{\operatorname{loc}}
\usepackage{cite}
\usepackage{ latexsym }
\definecolor{citation}{rgb}{0.11,0.67,0.84}
\definecolor{formula}{rgb}{0.1,0.2,0.6}
\definecolor{url}{rgb}{0.11,0.67,0.84}
\usepackage{pgf,tikz}
\usepackage{mathrsfs}
\usepackage{fancyhdr}
\usepackage{dutchcal}

\usepackage{dsfont}

\newcommand{\reqnomode}{\tagsleft@false}

\usepackage{hyperref}

\def\dx{\,{\rm d}x}

\def\dtt{\,{\rm d}t}

\def\ds{\,{\rm d}s}

\def\dy{\,{\rm d}y}

\def \d{\,{\rm d}}

\def\dist{\,{\rm dist}}

\def\supp{\,{\rm supp}}

\allowdisplaybreaks
\makeatletter
\DeclareRobustCommand*{\bfseries}{%
  \not@math@alphabet\bfseries\mathbf
  \fontseries\bfdefault\selectfont
  \boldmath
}

\makeatother

\newlength{\defbaselineskip}
\newcommand{\setlinespacing}[1]
           {\setlength{\baselineskip}{#1 \defbaselineskip}}
\newcommand{\mint}{\mathop{\int\hskip -1,05em -\, \!\!\!}\nolimits}

\newtheorem{theorem}{Theorem}
\newtheorem{corollary}{Corollary}[section]
\newtheorem{definition}{Definition}
\newtheorem{remark}{Remark}[section]
\newtheorem{lemma}{Lemma}[section]
\newtheorem{proposition}{Proposition}[section]
\numberwithin{equation}{section}
\allowdisplaybreaks[4]

\newcommand{\kk}{\kappa}

\def\er{\mathbb R}

\newcommand{\ti}[1]{\tilde{#1}}
\newcommand{\mf}[1]{\mathfrak{#1}}

\newcommand\eps\varepsilon

\def\eqn#1$$#2$${\begin{equation}\label#1#2\end{equation}}

\newcommand{\be}{\begin{equation}}
\newcommand{\ee}{\end{equation}}

\newcommand{\rr}{\varrho}

\newcommand{\snr}[1]{\lvert #1\rvert}

\newcommand{\nr}[1]{\lVert #1 \rVert}

\newcommand{\uu}{\mathfrak{u}}

\newcommand{\RN}{\mathbb{R}^{N}}

\newcommand{\N}{\mathbb{N}}

\def\name[#1, #2]{#1 #2}

\title[Singular multiple integrals and nonlinear potentials]{Singular multiple integrals and nonlinear potentials}

\author[De Filippis]{Cristiana De Filippis}  \address{Cristiana De Filippis\\Dipartimento SMFI, Universit\'a di Parma\\ Parco Area delle Scienze 53/A, 43124 Parma, Italy} \email{\url{cristiana.defilippis@unipr.it}}
\author[Stroffolini]{Bianca Stroffolini}  \address{Bianca Stroffolini\\ Dipartimento di Ingegneria Elettrica e delle Tecnologie dell'Informazione, University of Napoli "Federico II", Via Claudio, 80125 Napoli, Italy} \email{\url{bstroffo@unina.it}}

\begin{document}

\subjclass[2020]{35B65, 31C45 \vspace{1mm}} 

\keywords{Quasiconvexity, $(p,q)$-growth, Nonlinear potential theory, Singular variational integrals.\vspace{1mm}}


\maketitle

\begin{abstract}
We prove sharp partial regularity criteria of nonlinear potential theoretic nature for the Lebesgue-Serrin-Marcellini extension of nonhomogeneous singular multiple integrals featuring $(p,q)$-growth conditions. 
 \end{abstract}
%

\setlinespacing{1.00}
\section{Introduction}\label{si}
We provide optimal partial regularity criteria for relaxed minimizers of nonhomogeneous, singular multiple integrals of the form
\begin{eqnarray}\label{fun}
W^{1,p}(\Omega,\mathbb{R}^{N})\ni w\mapsto \mathcal{F}(w;\Omega):=\int_{\Omega}\left[F(Dw)-f\cdot w\right] \ \dx,
\end{eqnarray}
i.e., local minimizers of the Lebesgue-Serrin-Marcellini extension of $\mathcal{F}(\cdot)$:
\begin{flalign}\label{exfun}
\bar{\mathcal{F}}(w;\Omega):=\inf\left\{\liminf_{j\to \infty}\mathcal{F}(w_{j};\Omega)\colon \{w_{j}\}_{j\in \N}\subset W^{1,q}_{\loc}(\Omega,\mathbb{R}^{N})\colon w_{j}\rightharpoonup w \ \mbox{in} \ W^{1,p}(\Omega,\mathbb{R}^{N})\right\},
\end{flalign}
using tools from Nonlinear Potential Theory, thus completing the analysis started in \cite{deqc} for degenerate functionals. More precisely, we prove almost everywhere gradient continuity for local minimizers of \eqref{exfun} under sharp assumptions on the external datum $f$. Here, $\Omega\subset \mathbb{R}^{n}$ is an open, bounded set with Lipschitz boundary, $n\ge 2$, and $F\colon \mathbb{R}^{N\times n}\to \mathbb{R}$ is a strictly quasiconvex integrand, verifying so-called $(p,q)$-growth conditions according to Marcellini's terminology \cite{ma2}:
\eqn{pqpq}
$$
\snr{z}^{p}\lesssim F(z)\lesssim 1+\snr{z}^{q},\qquad\qquad 1<p\le q;
$$
the singular behavior of $F(\cdot)$ around zero being encoded in the requirement $p\in (1,2)$. Let us recall that $F(\cdot)$ is quasiconvex when
\begin{flalign}\label{qc}
\mint_{B_{1}(0)}F(z+D\varphi) \  \dx \ge F(z)\quad \mbox{holds for all} \ \ z\in \mathbb{R}^{N\times n}, \ \ \varphi\in C^{\infty}_{\rm c}(B_{1}(0),\mathbb{R}^{N}),
\end{flalign}
therefore the three main aspects of \eqref{fun}-\eqref{exfun} we are interested in are the presence of a nontrivial forcing term $f$, the $(p,q)$-growth conditions in \eqref{pqpq} and the quasiconvexity \eqref{qc} of the integrand $F(\cdot)$. Let us briefly discuss some classical and recent results on these ingredients as each of them is currently object of intense investigation. The problem of determining the best conditions to impose on $f$ in order to prove gradient continuity for minima is classical and received a considerable attention in the past decades. To better understand this issue, let us introduce the Lorentz space $L(n,1)$, defined by
$$
w\in L(n,1) \ \Longleftrightarrow \ \nr{w}_{L(n,1)}:=\int_{0}^{\infty}\snr{\{x\in \mathbb{R}^{n}\colon \snr{w(x)}>t\}}^{1/n} \ \dtt<\infty.
$$
A related deep result of Stein \cite{stein} states that 
\eqn{xx.4}
$$
w\in W^{1,n}\ \ \mbox{and} \ \ Dw\in L(n,1) \ \Longrightarrow \ w \ \ \mbox{is continuous},
$$
so \eqref{xx.4} and the immersions $L^{n+\varepsilon}\hookrightarrow L(n,1) \hookrightarrow L^{n}$ for all $\varepsilon>0$, lead to the borderline characterization of $L(n,1)$ as the limiting space with respect to the Sobolev embedding theorem. A linear PDE interpretation of Stein's theorem relying on the combination of \eqref{xx.4} with standard Calder\'on-Zygmund theory prescribes that
$$
-\Delta u=f\in L(n,1) \ \Longrightarrow \ Du \ \ \mbox{is continuous},
$$
which turns out to be sharp, in the light of Cianchi's counterexample \cite{CiGA}. Surprisingly enough, the same conclusion holds in a way more general setting than the linear one. It is indeed true for uniformly elliptic operators \cite{akm,ba2,cm,cm1,dz,dumi2,kumi1,kumig,kumi2,np1,tn}; systems of differential forms \cite{sil}; fully nonlinear elliptic equations \cite{bm,dkm}, general nonuniformly elliptic functionals \cite{bemi,demi1}; and it also holds at the level of partial regularity for systems of the $p$-Laplacian type without Uhlenbeck's structure \cite{by,kumi}. 
The key point consists in the possibility of gaining local control on the gradient of solutions via the truncated Riesz potential of $f$, that is
\begin{eqnarray}\label{i.0}
\mathbf{I}^{f}_{1}(x_{0},\rr):=\int_{0}^{\rr}\left(\mint_{B_{\sigma}(x_{0})}\snr{f} \ \dx\right) \ \d\sigma\lesssim \int_{\mathbb{R}^{n}}\frac{\snr{f(y)}}{\snr{x-y}^{n-1}} \ \dy,
\end{eqnarray}
which is a standard aspect of linear equations and a remarkable feature of nonlinear ones, cf. Kuusi \& Mingione's seminal works \cite{kumig,kumi2}. On the other hand, in \cite{bemi,demi1} the gradient of minima is dominated via a nonlinear Wolff type potential, first introduced by Havin \& Maz'ya \cite{hm1}, defined as:
\begin{eqnarray*}
\mathbf{I}^{f}_{1,m}(x_{0},\rr):=\int_{0}^{\rr}\left(\sigma^{m}\mint_{B_{\sigma}(x_{0})}\snr{f}^{m} \ \dx\right)^{1/m} \ \frac{\d\sigma}{\sigma},\qquad m>1,
\end{eqnarray*}
sharing the same homogeneity - and therefore analogous mapping properties on function spaces - as the linear potential in \eqref{i.0}. 
All the aforementioned results crucially rely on the strong ellipticity of the operators involved, while in \eqref{fun} the integrand $F(\cdot)$ is only quasiconvex. This notion was first introduced by Morrey \cite{mo} and it turns out to be a natural condition in the multidimensional Calculus of Variations. Indeed, under polynomial growth conditions on the integrand $F(\cdot)$, quasiconvexity is a necessary and sufficient condition for sequential weak lower semicontinuity in $W^{1,p}$, \cite{af1,bamu,bfm,foma,ma3,mo}. A peculiar characteristic of quasiconvexity is that it is a purely nonlocal concept \cite{k,mo} in the sense that there is no condition involving only $F(\cdot)$ and a finite number of its derivatives, which is equivalent to quasiconvexity. Moreover, minimizers and critical points of quasiconvex functionals have a very different behavior. Precisely, a classical result of Evans \cite{ev} states that minima are regular outside a negligible "singular" set, while M\"uller \& \v{S}ver\'ak \cite{musv} proved that critical points, i.e. solutions to the associated Euler-Lagrange system, may be everywhere discontinuous. This is coherent with the theory of elliptic systems: well-known counterexamples \cite{ms,svya} show that solutions might develop singularities, therefore in the genuine vectorial setting the best one could hope for is partial regularity. The matter of almost everywhere regularity for minimizers of quasiconvex integrals was first treated by Evans \cite{ev} in the case of quadratic functionals, and, after that, it received lots of attention over the years. Subsequently, partial regularity for multiple integrals with standard $p$-growth was obtained in \cite{af,cfm,kt} exploiting Evans' blow up method, while in \cite{dumi} a unified approach to the partial regularity for degenerate or singular quasiconvex integrals was proposed via the $p$-harmonic approximation and in \cite{km} was derived an upper bound on the Hausdorff dimension of the singular set of minima of quasiconvex functionals. We refer to \cite{bs,bddms,deqc,dlsv,dgk,gme,gm1,gk,i1,i,k1,li,ts,ts1,ts2} and references therein for a non-exhaustive list of remarkable contributions in more general settings.
The other main feature of the class of integrands considered in this paper is their $(p,q)$-growth conditions. This nomenclature was introduced by Marcellini in the fundamental papers \cite{ma2,ma5} within the framework of nonlinear elasticity. In fact, a basic model describing the behavior of compressible materials subject to deformations is given by
\eqn{hh}
$$
W^{1,p}(\Omega,\mathbb{R}^{n})\ni w\mapsto \mathcal{H}(w;\Omega):=\int_{\Omega}\left[\snr{Dw}^{p}+\sqrt{1+\snr{\det(Dw)}^{2}}-f\cdot w\right] \ \dx,
$$
for some $f\in W^{1,p}(\Omega,\mathbb{R}^{N})^{*}$, see \cite{bb,bamu,ma2,ma5}. A natural phenomenon in compressible elasticity is cavitation, i.e. the possible formation of cavities (holes) in elastic bodies after stretch, corresponding to the development of singularities in equilibrium solutions (minima) of $\mathcal{H}(\cdot)$. Functional $\mathcal{H}(\cdot)$ is quasiconvex in the sense of \eqref{qc}, \cite[Chapter 5]{giu}, however in general it is not $W^{1,p}$-quasiconvex\footnote{I.e.: \eqref{qc} holds for all $\varphi\in W^{1,p}_{0}(B_{1}(0),\mathbb{R}^{N})$.} unless $p\ge n$, \cite[Theorem 4.1]{bamu}, while its Lebesgue-Serrin-Marcellini extension $\bar{\mathcal{H}}(\cdot)$ is $W^{1,p}$-quasiconvex provided that $p>n-1$, \cite[Lemma 7.6]{ts1}. This means that the approach by relaxation based on the extension of the ambient space proposed in \cite{ma2,ma5} fits the analysis of cavitation better than the pointwise one of \cite{bb,bamu}, as it allows dealing with discontinuous maps thus describing the possible formations of cavities. We also point out that the integrand $H(z):=\snr{z}^{p}+\sqrt{1+\snr{\det(z)}^{2}}$ in \eqref{hh} verifies
$$
\snr{z}^{p}\le H(z)\lesssim 1+\snr{z}^{n},
$$
that is \eqref{pqpq} with $q=n$. This was the first main reason behind the investigation of variational integrals with $(p,q)$-growth: starting with \cite{ma2} for questions of semicontinuity and \cite{ma4,ma1} about regularity, since then such class of functionals received lots of attention - with no pretence of completeness we mention the everywhere regularity results in \cite{bemi,BS1,BS,BF,bdms,demi1,demi3,ELM,hs,ko,ko1}, the partial regularity proven in \cite{cdk,px,demi2,dsv,dumi1,s} for general systems and manifold constrained problems with special structure and refer to \cite{masu2} for a reasonable survey. The aforementioned results hold for strictly convex variational integrals. In the quasiconvex setting partial regularity has been obtained by Schmidt \cite{ts,ts2} for homogeneous - $f\equiv 0$ in \eqref{fun} - functionals with $(p,q)$-growth and for their Lebesgue-Serrin-Marcellini extension \cite{ts1}, while \cite{deqc} contains sharp partial regularity criteria in terms of a nontrivial forcing term $f$ for relaxed minimizers of degenerate integrals of the form \eqref{fun}. The standard notion of relaxed local minimizers \cite{ts1} reads as:
\begin{definition}\label{d1}
Let $p\in (1,\infty)$. A function $u\in W^{1,p}(\Omega,\mathbb{R}^{N})$ is a local minimizer of \eqref{exfun} on $\Omega$ with $f\in W^{1,p}(\Omega,\mathbb{R}^{N})^{*}$ if and only if every $x_{0}\in \Omega$ admits a neighborhood $B\Subset \Omega$ so that $\bar{\mathcal{F}}(u;B)<\infty$ and $\bar{\mathcal{F}}(u;B)\le \bar{\mathcal{F}}(w;B)$ for all $w\in W^{1,p}(B,\mathbb{R}^{N})$ so that $\supp(u-w)\Subset B$. 
\end{definition}
Such definition can be immediately adapted to local minimizers of functional \eqref{fun}. Let us point out that when considering \eqref{fun}-\eqref{exfun} we will assume with no loss of generality that $f$ is defined on the whole $\mathbb{R}^{n}$, which is always possible if we set $f\equiv 0$ in $\mathbb{R}^{n}\setminus \Omega$. For this reason, when stating that $f$ belongs to a certain function space, we shall often avoid to specify the underlying domain. Further details about the notation employed can be found in Section \ref{presec} below. The main result of our paper is the following 
\begin{theorem}\label{t.v.0}
Under assumptions \eqref{assf}-\eqref{sqc}, \eqref{assf.0} and \eqref{f}, let $u\in W^{1,p}(\Omega,\mathbb{R}^{N})$ be a local minimizer of \eqref{exfun}. Suppose that
\eqn{1.2}
$$
\lim_{\rr\to 0}\mathbf{I}^{f}_{1,m}(x,\rr)=0\qquad \mbox{locally uniformly in} \ \ x\in \Omega.
$$
Then there exists an open "regular" set $\Omega_{u}\subset \Omega$ of full $n$-dimensional Lebesgue measure with $\snr{\Omega\setminus \Omega_{u}}=0$ such that $V_{p}(Du)$ and $Du$ are continuous on $\Omega_{u}$. In particular, the regular set $\Omega_{u}$ can be characterized as
\begin{flalign*}
&\Omega_{u}:=\left\{\frac{}{}x_{0}\in \Omega\colon \exists M\equiv M(x_{0})\in (0,\infty), \ \ti{\varepsilon}\equiv \ti{\varepsilon}(\textnormal{\texttt{data}},M), \ \ti{\rr}\equiv \ti{\rr}(\textnormal{\texttt{data}},M,f(\cdot))\in (0,\min\{d_{x_{0}},1\})\frac{}{}\right.\nonumber \\
&\qquad \qquad \left.\frac{}{}\mbox{such that} \ \ \snr{(V_{p}(Du))_{B_{\rr}(x_{0})}}<M \ \mbox{and} \ \mf{F}(u;B_{\rr}(x_{0}))<\ti{\varepsilon} \ \mbox{for some} \ \rr\in (0,\ti{\rr}]\frac{}{}\right\}.
\end{flalign*}
\end{theorem}
Theorem \ref{t.v.0} comes as a consequence of a fine connection established between the Lebesgue points of $Du$ and $V_{p}(Du)$ and the pointwise behavior of the Wolff potential $\mathbf{I}^{f}_{1,m}(\cdot)$.
\begin{theorem}\label{t.v}
Under assumptions \eqref{assf}-\eqref{sqc}, \eqref{assf.0} and \eqref{f}, let $u\in W^{1,p}(\Omega,\mathbb{R}^{N})$ be a local minimizer of \eqref{exfun}, $x_{0}\in \Omega$ be a point such that 
\eqn{1.3}
$$\mathbf{I}^{f}_{1,m}(x_{0},1)<\infty$$ and $M\equiv M(x_{0})$ be a positive constant. There exist $\ti{\varepsilon}\equiv \ti{\varepsilon}(\textnormal{\texttt{data}},M)\in (0,1)$ and $\ti{\rr}\equiv \ti{\rr}(\textnormal{\texttt{data}},M,f(\cdot))\in (0,\min\{1,d_{x_{0}}\})$ such that if
\begin{eqnarray}\label{1.4}
\left\{
\begin{array}{c}
\displaystyle
\ \snr{(V_{p}(Du))_{B_{\rr}(x_{0})}}<M\\[8pt]\displaystyle
\ \mf{F}(u;B_{\rr}(x_{0}))+\mathbf{I}^{f}_{1,m}(x_{0},\rr)^{\frac{p}{2(p-1)}}+\mathbf{I}^{f}_{1,m}(x_{0},\rr)^{\frac{q}{2(p-1)}}+M^{(2-p)/p}\mathbf{I}^{f}_{1,m}(x_{0},\rr)<\ti{\varepsilon},
\end{array}
\right.
\end{eqnarray}
for some $\rr\in (0,\ti{\rr}]$, then
\begin{flalign}\label{limv}
\lim_{s\to 0}(V_{p}(Du))_{B_{s}(x_{0})}=V_{p}(Du(x_{0})),\qquad \qquad \lim_{s\to 0}(Du)_{B_{s}(x_{0})}=Du(x_{0})
\end{flalign}
and
\begin{flalign}\label{1.7}
\left\{
\begin{array}{c}
\displaystyle
\ \snr{V_{p}(Du(x_{0}))-(V_{p}(Du))_{B_{\sigma}(x_{0})}}\le c\mf{N}(x_{0};\sigma)\\[8pt]\displaystyle
\ \snr{Du(x_{0})-(Du)_{B_{\sigma}(x_{0})}}\le c\mf{N}(x_{0};\sigma)^{2/p}+c\snr{(Du)_{B_{\sigma}(x_{0})}}^{(2-p)/2}\mf{N}(x_{0};\sigma) ,
\end{array}
\right.
\end{flalign}
for all $\sigma\in (0,\rr]$, where $c\equiv c(\textnormal{\texttt{data}},M)$ and
\begin{eqnarray*}
\mf{N}(x_{0};\sigma)&\approx& \mf{F}(u;B_{\sigma}(x_{0}))+\textbf{I}^{f}_{1,m}(x_{0},\sigma)^{\frac{p}{2(p-1)}}+\textbf{I}^{f}_{1,m}(x_{0},\sigma)^{\frac{q}{2(p-1)}}\nonumber \\
&&+\snr{(V_{p}(Du))_{B_{\sigma}(x_{0})}}^{(2-p)/p}\mathbf{I}^{f}_{1,m}(x_{0},\sigma),
\end{eqnarray*}
up to constants depending on $(\textnormal{\texttt{data}},M)$. In particular, $x_{0}\in \Omega$ satisfying \eqref{1.3} is a Lebesgue point of $V_{p}(Du)$ and of $Du$ if and only if it verifies \eqref{1.4}.
\end{theorem}
Conditions \eqref{1.2} or \eqref{1.3} can be guaranteed once prescribed the membership of $f$ to a proper function space, as stated in the following optimal function space criterion.
\begin{theorem}\label{t.v.1}
Under assumptions \eqref{assf}-\eqref{sqc}, \eqref{assf.0} and \eqref{f}, let $u\in W^{1,p}(\Omega,\mathbb{R}^{N})$ be a local minimizer of \eqref{exfun}. There exists an open set $\Omega_{u}\subset \Omega$ of full $n$-dimensional Lebesgue measure such that $f\in L(n,1)$ yields that $Du$, $V_{p}(Du)$ are continuous on $\Omega_{u}$, while if $f\in L^{d}$ for some $d>n$, then $Du,V_{p}(Du)\in C^{0,\ti{\alpha}}_{\loc}(\Omega_{u},\mathbb{R}^{N\times n})$ with $\ti{\alpha}\equiv \ti{\alpha}(n,N,p,d)$.
\end{theorem}
We stress that Theorems \ref{t.v.0}-\ref{t.v.1} hold for general strictly $W^{1,p}$-quasiconvex functionals\footnote{I.e.: \eqref{sqc} below holds for all $\varphi\in W^{1,p}_{0}(B,\mathbb{R}^{N})$.} as in \eqref{fun}, or for functionals that coincide with their Lebesgue-Serrin-Marcellini extension. Moreover, our results also apply to relaxed local minimizer of the functional $\mathcal{H}(\cdot)$ in \eqref{hh} with the choice $n=q=2$ and $p\in (8/5,2)$, cf. \cite[Section 2.4]{deqc}; we refer to \cite{ts2,sv1} for further discussions and more examples. We finally remark that the nonlinear potential theory for singular nonhomogeneous equations or systems of the $p$-Laplacian type is a very recent achievement. In fact, after Duzaar \& Mingione's breakthrough \cite{dumi2} on pointwise potential estimates with $p\in \left(2-1/n,\infty\right)$, lots of efforts have been devoted to the extension of such result to all $p\in (1,2)$: in \cite{np1} Nguyen \& Phuc decreased the lower bound on $p$ switching from $p>2-1/n$ to $p>\frac{3n-2}{2n-1}$; later on Dong \& Zhu \cite{dz} and Nguyen \& Phuc \cite{tn} (singular equations with measure data) and Byun \& Youn \cite{by} (general subquadratic systems) eventually covered the full range $p\in (1,2)$. In this respect, our paper fits such line of research as we provide pointwise bounds on the gradient oscillation of minima of \eqref{exfun} that hold almost everywhere subject to the validity of a smallness condition on the excess functional that naturally involves also the potential $\mathbf{I}^{f}_{1,m}(\cdot)$. The combined effect of the singular character of the integrand $F(\cdot)$ that unavoidably burdens potentials, and of its $(p,q)$-growth leading to nonhomogeneous estimates, forces us to develop some delicate iteration schemes, relying on refined approximation methods \cite{dumi,ts2} and on potential theoretic arguments \cite{deqc,kumi,kumi2}, that simultaneously allow a uniform control on the size of the excess functional and of the gradient average during iterations and preserves the rough regularity information available on $f$, i.e. the mere finiteness of the related Wolff type potential.

\section{Preliminaries}\label{presec}
In this section we record the notation employed throughout the paper, describe the structural assumptions governing the ingredients appearing in \eqref{fun} and collect certain basic results that will be helpful later on.
\subsection{Notation}\label{notation} In this paper, $\Omega\subset \er^n$ will always be an open, bounded domain with Lipschitz-regular boundary, and $n \geq 2$. We denote by $c$ a general constant larger than one depending on the main parameters governing the problem. We will still denote by $c$ distinct occurrences of constant $c$ from line to line. Specific occurrences will be marked with symbols $c_*,  \tilde c$ or the like. Significant dependencies on certain parameters will be outlined by putting them in parentheses, i.e. $c\equiv c(n,p)$ means that $c$ depends on $n$ and $p$. By $ B_r(x_0):= \{x \in \er^n  :   |x-x_0|< r\}$ we indicate the open ball with center in $x_0$ and radius $r>0$; we shall avoid denoting the center when this is clear from the context, i.e., $B \equiv B_r \equiv B_r(x_0)$; this happens in particular with concentric balls. For $x_{0}\in \Omega$, it is $d_{x_{0}}:=\dist(x_{0},\partial \Omega)$ and with $z_{1},z_{2}\in \mathbb{R}^{N\times n}$, $s\ge 0$ we set $\mathcal{D}_{s}(z_{1},z_{2}):=(s^{2}+\snr{z_{1}}^{2}+\snr{z_{2}}^{2})$. Given a measurable set $ B \subset \er^{n}$ with bounded positive Lebesgue measure $\snr{B}\in (0,\infty)$, and a measurable map $g \colon  B \to \er^{k}$, $k\geq 1$, we set $$
   (g)_{B} \equiv \mint_{ B}  g(x) \dx  :=  \frac{1}{\snr{B}}\int_{B}  g(x) \dx.
$$
A useful feature of the average is its almost minimality, i.e.:
\begin{eqnarray}\label{minav}
\left(\mint_{B}\snr{g-(g)_{B}}^{t} \ \dx\right)^{1/t}\le 2\left(\mint_{B}\snr{g-z}^{t} \dx\right)^{1/t}\qquad \mbox{for all} \ \ z\in \mathbb{R}^{k}, \ \ t\ge 1.
\end{eqnarray}
For $t\ge 1$, $s\ge 0$, $q\ge p> 1$, we shorten:
\begin{eqnarray}\label{ik}
\mf{I}_{t}(g;\mathcal{B}):=\left(\mint_{\mathcal{B}}\snr{g(x)}^{t} \dx\right)^{\frac{1}{t}},\qquad\qquad   \mf{K}(s):=s+s^{q/p}
\end{eqnarray}
and define 
\begin{flalign*}
\mathds{1}_{\{q>p\}}:=\begin{cases}
\ 1\quad &\mbox{if} \ \ q>p\\
\ 0\quad &\mbox{if} \ \ q=p,
\end{cases}\qquad \qquad \mathds{1}_{\{q\ge 2\}}:=\begin{cases}
\ 1\quad &\mbox{if} \ \ q\ge 2\\
\ 0\quad &\mbox{if} \ \ 1<q<2.
\end{cases}
\end{flalign*}
Finally, if $t>1$ is any number, its conjugate will be denoted by $t':=t/(t-1)$ and its Sobolev esponent as $t^{*}:=nt/(n-t)$ when $t<n$ or any number larger than one for $t\ge n$. To streamline the notation, we gather together the main parameters governing our problem in the shorthand $\textnormal{\texttt{data}}:=\left(n,N,\lambda,\Lambda,p,q,\omega(\cdot),F(\cdot),m\right)$, we refer to Section \ref{assec} for more details on such quantities.

\subsection{Tools for $p$-Laplacian type problems} The vector field $V_{s,p}\colon \er^{N\times n} \to  \er^{N\times n}$, defined as
\begin{flalign*}
V_{s,p}(z):= (s^{2}+|z|^{2})^{(p-2)/4}z, \qquad p\in (1,\infty)\ \ \mbox{and}\ \ s\ge 0
\end{flalign*}
for all $z \in \er^{N\times n}$, which encodes the scaling features of the $p$-Laplacian operator, is a useful tool for handling $p$-Laplacean type problems. If $s=0$, we simply write $V_{s,p}(\cdot)\equiv V_{p}(\cdot)$. Let us premise that although most of the properties of the vector field $V_{s,p}(\cdot)$ that we are going to list below hold for all $p\in (1,\infty)$, from now on, we shall always assume that $p\in (1,2)$. It is well-known that
\begin{flalign}\label{Vm}
\left\{
\begin{array}{c}
\displaystyle 
\ \frac{\snr{V_{s,p}(z_{1}+z_{2})}^{2}}{\snr{z_{1}+z_{2}}}\lesssim \frac{\snr{V_{s,p}(z_{1})}^{2}}{\snr{z_{1}}}+\frac{\snr{V_{s,p}(z_{2})}^{2}}{\snr{z_{2}}} \\[6pt]\displaystyle
\ \snr{V_{s,p}(z_{1})-V_{s,p}(z_{2})}\approx (s^{2}+\snr{z_{1}}^{2}+\snr{z_{2}}^{2})^{(p-2)/4}\snr{z_{1}-z_{2}}\\[6pt]\displaystyle
\ \snr{V_{\snr{z_{1}},p}(z_{2}-z_{1})}\approx \snr{V_{p}(z_{1})-V_{p}(z_{2})}\\[6pt]\displaystyle
\ \frac{\snr{V_{s,p}(z_{1})}^{2}\snr{z_{2}}}{\snr{z_{1}}}\lesssim \snr{V_{s,p}(z_{1})}^{2}+\snr{V_{s,p}(z_{2})}^{2}\\[6pt]\displaystyle
\ \snr{V_{s,p}(kz)}\lesssim \max\{k,k^{p/2}\}\snr{V_{s,p}(z)}\\[6pt]\displaystyle
\ \snr{V_{s,p}(z)}\approx \min\{\snr{z},\snr{z}^{p/2}\}\\[6pt]\displaystyle
\ \snr{V_{s,p}(z_{1}+z_{2})}\lesssim \snr{V_{s,p}(z_{1})}+\snr{V_{s,p}(z_{2})},
\end{array}
\right.
\end{flalign}
for all $k>0$ cf. \cite{cfm,dumi,ts2} - of course to avoid trivialities, above $\snr{z_{1}+z_{2}}$, $\snr{z_{1}}$ are supposed to be positive - and that whenever $t>-1$, $s\in [0,1]$ and $z_{1},z_{2}\in \mathbb{R}^{N\times n}$ verify $s+\snr{z_{1}}+\snr{z_{2}}>0$, then
\begin{flalign}\label{l6}
\int_{0}^{1}\left[s^2+\snr{z_{1}+y(z_{2}-z_{1})}^{2}\right]^{\frac{t}{2}} \ \dy\approx (s^2+\snr{z_{1}}^{2}+\snr{z_{2}}^{2})^{\frac{t}{2}}.
\end{flalign}
As useful consequences of $\eqref{Vm}_{2}$ we have
\eqn{Vm.1}
$$
\snr{z_{1}-z_{2}}^{p}\lesssim \snr{V_{s,p}(z_{1})-V_{s,p}(z_{2})}^{2} +  \snr{V_{s,p}(z_{1})-V_{s,p}(z_{2})}^p(\snr{z_1}+s)^{p(2-p)/2}.
$$
It is also worth recalling a Poincar\'e-type inequality involving the vector field $V_{s,p}(\cdot)$: with $p\in (1,2)$, $B_{\rr}(x_{0})\Subset \Omega$ and $w\in W^{1,p}(B_{\rr}(x_{0}),\mathbb{R}^{N})$ it is
\begin{eqnarray}\label{v.poi}
\mint_{B_{\rr}(x_{0})}\left| \ V_{s,p}\left(\frac{w-(w)_{B_{\rr}(x_{0})}}{\rr}\right) \ \right|^{p^{\#}} \ \dx\le c\left(\mint_{B_{\rr}(x_{0})}\snr{V_{s,p}(Dw)}^{2} \ \dx\right)^{p^{\#}/2},
\end{eqnarray}
with $p^{\#}:=2n/(n-p)$ and $c\equiv c(n,N,p)$, see \cite[Lemma 8]{dumi}. For $B_{\rr}(x_{0})\Subset \Omega$, $w\in W^{1,p}(B_{\rr}(x_{0}),\mathbb{R}^{N})$ and $z_{0}\in \mathbb{R}^{N}$, we define the excess functional by
$$
\mf{F}(w,z_{0};B_{\rr}(x_{0})):=\left(\mint_{B_{\rr}(x_{0})}\snr{V_{p}(Dw)-z_{0}}^{2} \ \dx\right)^{1/2}
$$
and further introduce the auxiliary integral
$$
\ti{\mf{F}}(w,z_{0};B_{\rr}(x_{0})):=\left(\mint_{B_{\rr}(x_{0})}\snr{V_{\snr{z_{0}},p}(Dw-z_{0})}^{2} \ \dx\right)^{1/2}.
$$
If $z_{0}=(V_{p}(Dw))_{B_{\rr}(x_{0})}$ we shall abbreviate $\mf{F}(w,(V_{p}(Dw))_{B_{\rr}(x_{0})};B_{\rr}(x_{0}))\equiv \mf{F}(w;B_{\rr}(x_{0}))$. Let us point out that combining \eqref{minav} with \cite[(2.6)]{gm} it holds that
\eqn{equiv.11}
$$
\mathfrak{F}(w;B_{\rr}(x_{0}))\approx \mathfrak{F}(w,V_{p}((Dw)_{B_{\rr}(x_{0})});B_{\rr}(x_{0})),
$$
while if $z_{0}=(V_{p})^{-1}((V_{p}(Dw))_{B_{\rr}(x_{0})})$ - recall that $V_{p}(\cdot)$ is an isomorphism of $\mathbb{R}^{N\times n}$ - via $\eqref{Vm}_{3}$ we have
\begin{eqnarray}\label{equiv.22}
\mf{F}(w;B_{\rr}(x_{0}))\approx \ti{\mf{F}}(w,z_{0};B_{\rr}(x_{0})).
\end{eqnarray}
In all the above displays, the constants implicit in "$\approx$, $\lesssim$" depend on $(n,N,p,t)$ - the dependency on $t$ accounts for the exponent appearing in \eqref{l6}. To the scopes of this paper, it is fundamental to record some well-known scaling features of the excess functional.
\begin{lemma}
Let $1<p<2$ be a number, $B_{\rr}(x_{0})\subset \mathbb{R}^{n}$ be a ball, and $w\in W^{1,p}(B_{\rr}(x_{0}),\mathbb{R}^{N})$ be any function. With $\nu\in (0,1)$ it holds that
\begin{flalign}\label{tri.1}
\left\{
\begin{array}{c}
\displaystyle
\ \mf{F}(w;B_{\nu\rr}(x_{0}))\le \frac{2}{\nu^{n/2}}\mf{F}(w;B_{\rr}(x_{0}))\\[8pt]\displaystyle
\ \snr{(V_{p}(Dw))_{B_{\nu\rr}(x_{0})}}\le \frac{1}{\nu^{n/2}}\mf{F}(w;B_{\rr}(x_{0}))+\snr{(V_{p}(Dw))_{B_{\rr}(x_{0})}}\\[8pt]\displaystyle
\ \snr{(V_{p}(Dw))_{B_{\nu\rr}(x_{0})}-(V_{p}(Dw))_{B_{\rr}(x_{0})}}\le \frac{1}{\nu^{n/2}}\mf{F}(w;B_{\rr}(x_{0})).
\end{array}
\right.
\end{flalign}
Moreover, if for $\sigma\le \rr$ there is $\kk\in \N\cup \{0\}$ satisfying $\nu^{\kk+1}\rr<\sigma\le \nu^{\kk}\rr$, then
\begin{flalign}\label{ls.42.1}
\left\{
\begin{array}{c}
\displaystyle
\ \mf{F}(w;B_{\nu^{\kk+1}\rr}(x_{0}))\le \frac{2}{\nu^{n/2}}\mf{F}(w;B_{\sigma}(x_{0}))^\le \frac{2^{2}}{\nu^{n}}\mf{F}(w;B_{\nu^{\kk}\rr}(x_{0}))\\[8pt]\displaystyle
\ \snr{(V_{p}(Dw))_{B_{\nu^{\kk+1}\rr}(x_{0})}}\le \snr{(V_{p}(Dw))_{B_{\sigma}(x_{0})}}+\frac{1}{\nu^{n/2}}\mf{F}(w;B_{\sigma}(x_{0}))\le \snr{(V_{p}(Dw))_{B_{\nu^{\kk}\rr}(x_{0})}}+\frac{2^{2}}{\nu^{n}}\mf{F}(w;B_{\nu^{\kk}\rr}(x_{0})),
\end{array}
\right.
\end{flalign}
and, whenever $c_{*}\ge 1$ is an absolute constant it is
\begin{flalign}\label{ls.43.1}
 &\snr{(V_{p}(Dw))_{B_{\nu^{\kk+1}\rr}(x_{0})}}+c_{*}\mf{F}(w;B_{\nu^{\kk+1}\rr}(x_{0}))\le \frac{2^{2}}{\nu^{n/2}}\left[\snr{(V_{p}(Dw))_{B_{\sigma}(x_{0})}}+c_{*}\mf{F}(w;B_{\sigma}(x_{0}))\right]\nonumber \\
 &\qquad \qquad \qquad \le\frac{2^{4}}{\nu^{n}}\left[\snr{(V_{p}(Dw))_{B_{\nu^{\kk}\rr}(x_{0})}}+c_{*}\mf{F}(w;B_{\nu^{\kk}\rr}(x_{0}))\right].
\end{flalign}
\end{lemma}
We conclude this section with a classical iteration lemma, \cite[Lemma 6.1]{giu}.
\begin{lemma}\label{l5}
Let $h\colon [\rr_{0},\rr_{1}]\to \mathbb{R}$ be a non-negative and bounded function, and let $\theta \in (0,1)$, $A,B,\gamma_{1},\gamma_{2}\ge 0$ be numbers. Assume that $h(t)\le \theta h(s)+A(s-t)^{-\gamma_{1}}+B(s-t)^{-\gamma_{2}}$ holds for all $\rr_{0}\le t<s\le \rr_{1}$. Then the following inequality holds $h(\rr_{0})\le c(\theta,\gamma_{1},\gamma_{2})[A(\rr_{1}-\rr_{0})^{-\gamma_{1}}+B(\rr_{1}-\rr_{0})^{-\gamma_{2}}].$
\end{lemma}

\subsection{Structural assumptions}\label{assec} We assume that $F\colon \mathbb{R}^{N\times n}\to \mathbb{R}$ is an integrand verifying:
\begin{flalign}\label{assf}
\begin{cases}
\ F\in C^{1}_{\loc}(\mathbb{R}^{N\times n})\cap C^{2}_{\loc}(\mathbb{R}^{N\times n}\setminus \{0\})\\
\ \Lambda^{-1}\snr{z}^{p}\le F(z)\le \Lambda \left(1+\snr{z}^{q}\right)
\end{cases}
\end{flalign}
for all $z\in \mathbb{R}^{N\times n}$, with $\Lambda\ge 1$ being a positive, absolute constant and exponents $(p,q)$ satisfying:
\begin{eqnarray}\label{pq}
1<p<2\qquad \mbox{and}\qquad \frac{q}{p}<1+\frac{1}{2n}.
\end{eqnarray}
It is fundamental that $F(\cdot)$ is strictly degenerate quasiconvex, in the sense that whenever $B\Subset \Omega$ is a ball it holds that 
\begin{flalign}\label{sqc}
\int_{B}\left[F(z+D\varphi)-F(z)\right] \dx\ge \lambda\int_{B}(\snr{z}^{2}+\snr{D\varphi}^{2})^{\frac{p-2}{2}}\snr{D\varphi}^{2} \dx\qquad \mbox{for all} \ \ z\in \mathbb{R}^{N\times n}, \  \ \varphi\in C^{\infty}_{c}(B,\mathbb{R}^{N}),
\end{flalign}
where $\lambda$ is a positive, absolute constant. As a consequence, for all $z\in \mathbb{R}^{N\times n}\setminus \{0\}$, $\xi\in \mathbb{R}^{N}$, $\zeta\in \mathbb{R}^{n}$ it holds that
\begin{eqnarray}\label{sqc.1}
\partial^{2}F(z)\langle\xi\otimes \zeta,\xi\otimes \zeta\rangle\ge 2\lambda\snr{z}^{p-2}\snr{\xi}^{2}\snr{\zeta}^{2},
\end{eqnarray}
see \cite[Chapter 5]{giu} or \cite[Lemma 7.14]{ts1}. Moreover, as a minimal requirement on the second derivatives of $F(\cdot)$, we need to prescribe their behavior near the origin. Precisely, we need that
\begin{flalign}\label{assf.0}
\frac{\snr{\partial^{2}F(z)-\partial^{2}(\snr{z}^{p})}}{\snr{z}^{p-2}}\to 0\qquad \mbox{uniformly as} \ \ \snr{z}\to 0.
\end{flalign}
The by-product of \eqref{assf.0} is summarized in the following lemma, that collects results from \cite{af,ma3,ts2}.
\begin{lemma}
Let $F\colon \mathbb{R}^{N\times n}\to \mathbb{R}$ be an integrand verifying \eqref{assf}, $\eqref{pq}_{1}$ and \eqref{assf.0}. Then,
\begin{itemize}
    \item there exists a positive constant $c\equiv c(n,N,\Lambda,q)$ such that
    \begin{eqnarray}\label{assf.0.0}
    \snr{\partial F(z)}\le c(1+\snr{z}^{q-1})\qquad \mbox{for all} \ \ z\in \mathbb{R}^{N\times n};
    \end{eqnarray}
    \item whenever $z_{0}\in \mathbb{R}^{N\times n}$ verifies $\snr{z_{0}}\le L+1$ for some positive constant $L$, it is
\begin{eqnarray}\label{assf.0.1}
\left\{
\begin{array}{c}
\displaystyle 
\ \snr{F(z_{0}+z)-F(z_{0})-\langle\partial F(z_{0}),z\rangle}\le c\left(\snr{V_{\snr{z_{0}},p}(z)}^{2}+\snr{V_{\snr{z_{0}},q}(z)}^{2}\right) \\[8pt]\displaystyle
\ \snr{\partial F(z_{0}+z)-\partial F(z_{0})}\le c\snr{z}^{-1}\left(\snr{V_{\snr{z_{0}},p}(z)}^{2}+\snr{V_{\snr{z_{0}},q}(z)}^{2}\right) ,
\end{array}
\right.
\end{eqnarray}
for all $z\in \mathbb{R}^{N\times n}$, with $c\equiv c(n,N,\Lambda,p,q,F(\cdot),L)$;
\item there exists a concave modulus of continuity $\omega\colon (0,\infty)\to (0,\infty)$ with $\lim_{s\to 0}\omega(s)=0$ such that
\begin{eqnarray}\label{assf.0.2}
\snr{z}\le \omega(s) \ \Longrightarrow \ \snr{\partial F(z)-\partial F(0)-\snr{z}^{p-2}z}\le s\snr{z}^{p-1};
\end{eqnarray}
\item whenever $L>0$ is a positive constant and $z\left(\in \mathbb{R}^{N\times n}\setminus \{0\}\right)\cap \left\{\snr{z}\le L\right\}$ it is:
\begin{eqnarray}\label{assf.0.4}
\snr{\partial^{2}F(z)}\le c\snr{z}^{p-2},
\end{eqnarray}
for $c\equiv c(n,N,p,F(\cdot),L)$;
\item for all positive constants $L$ and vectors $z_{1},z_{2}\in \mathbb{R}^{N\times n}$ so that $0<\snr{z_{1}}\le L$, $0<\snr{z_{1}}\le 2L$ it holds that
\begin{eqnarray}\label{assf.0.3}
\snr{\partial^{2}F(z_{1})-\partial^{2}F(z_{2})}\le \left(\frac{\snr{z_{1}}^{2}+\snr{z_{2}}^{2}}{\snr{z_{1}}^{2}\snr{z_{2}}^{2}}\right)^{(2-p)/2}\mu_{L}\left(\frac{\snr{z_{1}-z_{2}}^{2}}{\snr{z_{1}}^{2}+\snr{z_{2}}^{2}}\right),
\end{eqnarray}
where $\mu_{L}\colon (0,\infty)\to (0,\infty)$ is a nondecreasing modulus of continuity with $s\mapsto \mu_{L}(s)^{2}$ concave, depending on $F(\cdot)$ and on $L$.
\end{itemize}
\end{lemma}
Finally, the forcing term $f\colon \Omega\to \mathbb{R}^{N}$ displayed in \eqref{fun} is such that
\begin{eqnarray}\label{f}
f\in L^{m}(\Omega,\mathbb{R}^{N})\quad \mbox{with} \ \ n>m>(p^{*})'>1
\end{eqnarray}
which, together with $\eqref{pq}_{1}$ yields:
\begin{eqnarray}\label{f.0}
1<m'<p^{*}<\infty\qquad \mbox{and} \qquad f\in W^{1,p}(\Omega,\mathbb{R}^{N})^{*}.
\end{eqnarray}
Assumption \eqref{f} should be interpreted as a minimal integrability requirement on the forcing term $f$, in the sense that it must at least belong to some intermediate Lebesgue space between $L^{(p^{*})'}$ and $L^{n}$. The motivation behind this choice is twofold: the lower bound $m>(p^{*})'$ assures that $f\in (W^{1,p})^{*}$, so that the linear functional $w\mapsto \int f\cdot w \ \dx$ is continuous on $W^{1,p}$; the upper bound $m<n$ reminds that all in all the forthcoming estimates $f$ should appear raised to a power strictly less than $n$ - this will eventually contribute to the construction of the Wolff potential $\mathbf{I}^{f}_{1,m}(\cdot)$, which is well-behaved with respect to the embedding in Lorentz spaces exactly when $m<n$, cf. \cite[Section 2.3]{kumi1}.

\subsection{Harmonic approximation lemmas}
This section is devoted to a quick overview of the main features of $\mathcal{A}$-harmonic maps and of $p$-harmonic maps. Let $\mathcal{A}$ be a constant bilinear form on $\mathbb{R}^{N\times n}$, elliptic in the sense of Legendre-Hadarmard i.e., satisfying
\begin{eqnarray}\label{con.0}
\snr{\mathcal{A}}\le H\qquad \mbox{and}\qquad \mathcal{A}\langle \xi\otimes \zeta,\xi\otimes \zeta\rangle\ge H^{-1}\snr{\xi}^{2}\snr{\zeta}^{2},
\end{eqnarray}
for all $\zeta\in \mathbb{R}^{n}$, $\xi\in \mathbb{R}^{N}$, with $H\ge 1$ being an absolute constant. An $\mathcal{A}$-harmonic map on an open set $\Omega\subset \mathbb{R}^{n}$ is a function $h\in W^{1,2}(\Omega,\mathbb{R}^{N})$ such that
$$
\int_{\Omega}\mathcal{A}\langle Dh,D\varphi\rangle \ \dx=0\qquad \mbox{for all} \ \ \varphi\in C^{\infty}_{c}(\Omega,\mathbb{R}^{N}).
$$
In \cite{cfm,dgk} we find that $\mathcal{A}$-harmonic maps have good regularity features; in fact for $B_{\rr}(x_{0})\Subset \Omega$ with $\rr\in (0,1]$ it holds that
\begin{eqnarray}\label{y.0}
\nr{Dh}_{L^{\infty}(B_{\rr/2}(x_{0}))}+\rr\nr{D^{2}h}_{L^{\infty}(B_{\rr/2}(x_{0}))}\le c\mint_{B_{\rr}(x_{0})}\snr{Dh} \ \dx,
\end{eqnarray}
for $c\equiv c(n,N,H,d)$. We record an $\mathcal{A}$-harmonic approximation result from \cite[Lemma 4]{dumi}, see also \cite[Lemma 6]{dgk}.
\begin{lemma}\label{ahar}
Let $\mathcal{A}$ be a bilinear form on $\mathbb{R}^{N\times n}$ verifying \eqref{con.0}, $B_{\rr}(x_{0})\Subset \Omega$ be a ball and $p>1$ be a number. For any $\varepsilon>0$ there exists $\delta\equiv \delta(n,N,H,p,\varepsilon)\in (0,1]$ such that if $v\in W^{1,p}(B_{\rr}(x_{0}),\mathbb{R}^{N})$ with $\mf{I}_{2}(V_{1,p}(Dv);B_{\rr}(x_{0}))\le \sigma\le 1$ is approximately $\mathcal{A}$-harmonic in the sense that
$$
\left| \ \mint_{B_{\rr}(x_{0})}\mathcal{A}\langle Dv,D\varphi\rangle \ \dx \ \right|\le \sigma\delta\nr{D\varphi}_{L^{\infty}(B_{\rr}(x_{0}))}\qquad \mbox{for all} \ \ \varphi\in C^{\infty}_{c}(B_{\rr}(x_{0}),\mathbb{R}^{N}),
$$
then there exists an $\mathcal{A}$-harmonic map $h\in W^{1,p}(B_{\rr}(x_{0}),\mathbb{R}^{N})$ such that 
\begin{eqnarray*}
\mint_{B_{\rr}(x_{0})}\snr{V_{1,p}(Dh)}^{2} \ \dx\le c \qquad \mbox{and}\qquad \mint_{B_{\rr}(x_{0})}\left| \ V_{1,p}\left(\frac{v-\sigma h}{\rr}\right)\ \right|^{2} \ \dx \le c\sigma^{2}\varepsilon,
\end{eqnarray*}
for $c\equiv c(n,N,p)$.
\end{lemma}
We further recall the definition of $p$-harmonic map, i.e. a function $h\in W^{1,p}(\Omega,\mathbb{R}^{N})$ satisfying
\begin{eqnarray*}
\int_{\Omega}\langle \snr{Dh}^{p-2}Dh,D\varphi\rangle \ \dx=0\qquad \mbox{for all} \ \ \varphi\in C^{\infty}_{c}(\Omega,\mathbb{R}^{N}).
\end{eqnarray*}
According to the regularity theory contained in \cite{uh,ur}, whenever $B_{\rr}(x_{0})\Subset B_{r}(x_{0})\Subset \Omega$ are concentric balls, it is
\begin{flalign}\label{y.1}
\nr{Dh}_{L^{\infty}(B_{\rr/2}(x_{0}))}\le c'\left(\mint_{B_{\rr}(x_{0})}\snr{Dh}^{p} \ \dx\right)^{1/p}\qquad \mbox{and}\qquad \mf{F}(h;B_{\rr}(x_{0}))\le c''\left(\frac{\rr}{r}\right)^{\alpha}\mf{F}(h;B_{r}(x_{0})),
\end{flalign}
with $c',c''\equiv c',c''(n,N,p)$ and $\alpha\equiv \alpha(n,N,p)\in (0,1)$. As a "singular" variant of Lemma \ref{ahar}, we have the following $p$-harmonic approximation lemma from \cite[Lemma 1]{dumi1}.
\begin{lemma}\label{p.phar}
Let $p\in (1,\infty)$ be a number and $B_{\rr}(x_{0})\Subset \Omega$ be any ball. For all $\varepsilon>0$, there exists $\delta\equiv \delta(n,N,p,\varepsilon)\in (0,1]$ such that if $v\in W^{1,p}(B_{\rr}(x_{0}),\mathbb{R}^{N})$ with $\mf{I}_{p}(Dv;B_{\rr}(x_{0}))\le 1$ is approximately $p$-harmonic in the sense that
\eqn{p.phar1}
$$
\left| \ \mint_{B_{\rr}(x_{0})}\langle \snr{Dv}^{p-2}Dv,D\varphi\rangle \ \dx \ \right|\le \delta\nr{D\varphi}_{L^{\infty}(B_{\rr}(x_{0}))}\qquad \mbox{for all} \ \ \varphi\in C^{\infty}_{c}(B_{\rr}(x_{0}),\mathbb{R}^{N}),
$$
then there exits a $p$-harmonic map $h\in W^{1,p}(B_{\rr}(x_{0}),\mathbb{R}^{N})$ such that
\begin{eqnarray*}
\mf{I}_{p}(Dh;B_{\rr}(x_{0}))\le 1\qquad \mbox{and}\qquad  \mint_{B_{\rr}(x_{0})}\left| \ \frac{v-h}{\rr} \ \right|^{p} \ \dx \le c\varepsilon^{p},
\end{eqnarray*}
with $c\equiv c(n,N,p)$.
\end{lemma}
\subsection{On the Lebesgue-Serrin-Marcellini extension} Let $\mathbb{B}_{\Omega}$ be the family of all open subsets of $\Omega$ and $B\in \mathbb{B}_{\Omega}$. For $1<p\le q<\infty$, an integrand $F\in C(\mathbb{R}^{N\times n})$, and maps $f\in W^{1,p}(\Omega,\mathbb{R}^{N})^{*}$ and $w\in W^{1,p}(\Omega,\mathbb{R}^{N})$, the Lebesgue-Serrin-Marcellini extension of a functional of type \eqref{fun}, i.e.:
$$
\mathcal{F}(w;B):=\int_{B}[F(Dw)-f\cdot w] \ \dx =:\mathcal{F}_{0}(w;B)-\int_{B}f\cdot w \ \dx,
$$
is defined as
\begin{eqnarray*}
\bar{\mathcal{F}}(w;B):=\inf_{\{w_{j}\}_{j\in \N}\in \mathcal{C}(w;B)}\liminf_{j\to \infty}\mathcal{F}(w_{j};B),
\end{eqnarray*}
with
\begin{eqnarray*}
\mathcal{C}(w;B):=\left\{\{w_{j}\}_{j\in \N}\subset W^{1,q}_{\loc}(B,\mathbb{R}^{N})\cap W^{1,p}(B,\mathbb{R}^{N})\colon w_{j}\rightharpoonup w \ \mbox{weakly in} \ W^{1,p}(B,\mathbb{R}^{N})\right\}.
\end{eqnarray*}
By density of smooth maps in $W^{1,p}(B,\mathbb{R}^{N})$ it is $\mathcal{C}(w;B)\not =\{\emptyset\}$ and since in particular $f\in W^{1,p}(B,\mathbb{R}^{N})^{*}$ we can rewrite
\begin{eqnarray}\label{ls.7}
\bar{\mathcal{F}}(w;B)=\bar{\mathcal{F}}_{0}(w;B)-\int_{B}f\cdot w \ \dx,
\end{eqnarray}
therefore while describing the relevant features of relaxation we shall refer to the "bulk" component of $\bar{\mathcal{F}}(\cdot)$, i.e. $\bar{\mathcal{F}}_{0}(\cdot)$. A first crucial observation is that $\bar{\mathcal{F}}_{0}(\cdot)$ cannot be represented as an integral. In fact, a deep result from \cite{bfm,foma} states that with $F(z)\lesssim (1+\snr{z}^{q})$ and $1<p\le q<np/(n-1)$, each $w\in W^{1,p}(\Omega,\mathbb{R}^{N})$ such that $\bar{\mathcal{F}}_{0}(w;\Omega)<\infty$ uniquely determines a finite outer Radon measure $\mu_{w}$ verifying
\eqn{measure}
$$
\bar{\mathcal{F}_{0}}(w;\cdot)=\left.\mu_{w}\right|_{\mathbb{B}_{\Omega}}\qquad \mbox{and}\qquad \frac{\d\mu_{w}}{\d\mathcal{L}^{n}}=QF(Dw).
$$
Here $QF(\cdot)$ denotes the quasiconvex envelope of $F(\cdot)$, see \cite[Section 5.3]{giu}. Moreover, a $W^{1,p}$-coercivity condition like $\snr{z}^{p}\lesssim F(z)$ assures sequential lower semicontinuity\footnote{Recall that we always work under the assumption that $\Omega$ is an open, bounded domain with Lipschitz boundary.} of $\bar{\mathcal{F}}_{0}(\cdot;\Omega)$ with respect to the weak topology of $W^{1,p}(\Omega,\mathbb{R}^{N})$, cf. \cite{foma,ts1}. In \cite{bamu} it is proven that $W^{1,p}$-quasiconvexity is necessary for this semicontinuity property. However, the results of \cite{bamu} hold for integral functionals, while, in the light of \eqref{measure}, $\bar{\mathcal{F}}_{0}(\cdot)$ cannot be represented as an integral. Despite the measure representation \cite{foma}, the arguments developed in \cite{bamu} can be adapted to prove that $\bar{\mathcal{F}}_{0}(\cdot)$ features the proper notion of $W^{1,p}$-quasiconvexity, i.e.: $\bar{\mathcal{F}}_{0}(\ell+\varphi;B)\ge \bar{\mathcal{F}}_{0}(\ell;B)$ holds for all $B\in \mathbb{B}_{\Omega}$, $\varphi\in W^{1,p}(B,\mathbb{R}^{N})$ with $\supp(\varphi)\Subset B$ and any affine function $\ell(x):=v_{0}+\langle z_{0},x-x_{0}\rangle$, cf. \cite{ts1}. Other remarkable properties of $\bar{\mathcal{F}}_{0}(\cdot)$ such as additivity and extremality conditions can be found in \cite{ts1,ts2}. Now, if $F(\cdot)$ is a continuous and $W^{1,p}$-coercive integrand and $f\in W^{1,p}(\Omega,\mathbb{R}^{N})^{*}$, the weak sequential lower semicontinuity of $\bar{\mathcal{F}}_{0}(\cdot)$ in $W^{1,p}(\Omega,\mathbb{R}^{N})$ and direct methods assure that once fixed a boundary datum $u_{0}\in W^{1,p}(\Omega,\mathbb{R}^{N})$ such that $\bar{\mathcal{F}}_{0}(u_{0};\Omega)<\infty$ - recall \eqref{ls.7} - there exists a local minimizer $u\in u_{0}+W^{1,p}_{0}(\Omega,\mathbb{R}^{N})$ of \eqref{exfun} in the sense of Definition \ref{d1}. If in addition $F\in C^{1}_{\loc}(\mathbb{R}^{N\times n})$ with \eqref{qc}, \eqref{assf}$_{2}$ and \eqref{f} in force and exponents $(p,q)$ satisfying $1<p\le q<\min\left\{np/(n-1),p+1\right\}$, then any local minimizer $u\in W^{1,p}(\Omega,\mathbb{R}^{N})$ of \eqref{exfun} verifies by minimality the integral identity
\begin{eqnarray}\label{el}
0=\int_{\Omega}\left[\langle \partial F(Du),D\varphi\rangle-f\cdot \varphi\right] \ \dx\qquad \mbox{for all} \ \ \varphi\in C^{\infty}_{c}(\Omega,\mathbb{R}^{N}),
\end{eqnarray}
see \cite[Section 2.7]{deqc} and \cite[Section 7.1]{ts1}.

\section{Caccioppoli inequality}
We start by recording a variation obtained in \cite[Lemmas 6.3-6.5]{ts} and \cite[Lemmas 4.4 and 4.6]{ts1} of the extension result from \cite{foma}, which will be crucial for constructing comparison maps for minima of \eqref{exfun}.
\begin{lemma}\label{exlem}
Let $0<\tau_{1}<\tau_{2}$ be two numbers and $B_{\tau_{2}}\Subset \Omega$ be a ball. There exists a bounded, linear smoothing operator $\mf{T}_{\tau_{1},\tau_{2}}\colon W^{1,1}(\Omega,\mathbb{R}^{N})\to W^{1,1}(\Omega,\mathbb{R}^{N})$ defined as
\begin{eqnarray*}
W^{1,1}(\Omega,\mathbb{R}^{N})\ni w\mapsto \mf{T}_{\tau_{1},\tau_{2}}[w](x):=\mint_{B_{1}(0)}w(x+\vartheta(x)y) \ \dy,
\end{eqnarray*}
where it is $\vartheta(x):=\frac{1}{2}\max\left\{\min\left\{\snr{x}-\tau_{1},\tau_{2}-\snr{x}\right\},0\right\}$. If $w\in W^{1,p}(\Omega,\RN)$ for some $p\ge 1$, the map $\mf{T}_{\tau_{1},\tau_{2}}[w]$ has the following properties:
\begin{itemize}
    \item[(\emph{i.})] $\mf{T}_{\tau_{1},\tau_{2}}[w]\in W^{1,p}(\Omega,\mathbb{R}^{N})$;
    \item[(\emph{ii.})] $w=\mf{T}_{\tau_{1},\tau_{2}}[w]$ almost everywhere on $(\Omega\setminus B_{\tau_{2}})\cup B_{\tau_{1}}$;
    \item[(\emph{iii.})] $\mf{T}_{\tau_{1},\tau_{2}}[w]\in w+W^{1,p}_{0}(B_{\tau_{2}}\setminus \bar{B}_{\tau_{1}},\mathbb{R}^{N})$;
    \item[(\emph{iv.})] $\snr{D\mf{T}_{\tau_{1},\tau_{2}}[w]}\le c(n)\mf{T}_{\tau_{1},\tau_{2}}[\snr{Dw}]$ almost everywhere in $\Omega$.
\end{itemize}
Furthermore,
\begin{flalign}\label{ex.0}
\left\{
\begin{array}{c}
\displaystyle 
\ \nr{\mf{T}_{\tau_{1},\tau_{2}}[w]}_{L^{p}(B_{\tau_{2}}\setminus B_{\tau_{1}})}\le c\nr{w}_{L^{p}(B_{\tau_{2}}\setminus B_{\tau_{1}})} \\[8pt]\displaystyle
\ \nr{D\mf{T}_{\tau_{1},\tau_{2}}[w]}_{L^{p}(B_{\tau_{2}}\setminus B_{\tau_{1}})}\le c\nr{Dw}_{L^{p}(B_{\tau_{2}}\setminus B_{\tau_{1}})}\\[8pt]\displaystyle
\nr{D\mf{T}_{\tau_{1},\tau_{2}}[w]}_{L^{p}(B_{\varsigma}\setminus B_{\tau_{1}})}\le c\nr{Dw}_{L^{p}(B_{2\varsigma-\tau_{1}}\setminus B_{\tau_{1}})}\quad \mbox{for} \ \ \tau_{1}\le \varsigma\le (\tau_{1}+\tau_{2})/2\\[8pt]\displaystyle
\nr{D\mf{T}_{\tau_{1},\tau_{2}}[w]}_{L^{p}(B_{\tau_{2}}\setminus B_{\varsigma})}\le c\nr{Dw}_{L^{p}(B_{\tau_{2}}\setminus B_{2\varsigma-\tau_{2}})}\quad \mbox{for} \ \ (\tau_{1}+\tau_{2})/2\le \varsigma\le \tau_{2},
\end{array}
\right.
\end{flalign}
for $c\equiv c(n,p)$. Finally, let $\mathcal{N}\subset \mathbb{R}$ be a set with zero Lebesgue measure. There are 
\begin{flalign}\label{ex.2}
\ti{\tau}_{1}\in \left(\tau_{1},\frac{2\tau_{1}+\tau_{2}}{3}\right)\setminus \mathcal{N},\ \ti{\tau}_{2}\in \left(\frac{\tau_{1}+2\tau_{2}}{3},\tau_{2}\right)\setminus \mathcal{N} \quad \mbox{verifying} \ \ (\tau_{2}-\tau_{1})\approx (\ti{\tau}_{2}-\ti{\tau}_{1})
\end{flalign}
up to absolute constants, such that
if $1\le p\le 2$, $s\ge 0$ and $\frac{2}{p}\le d<\frac{2n}{n-1}$, it is
\begin{eqnarray}\label{ex.1.2}
\nr{V_{s,p}(D\mf{T}_{\ti{\tau}_{1},\ti{\tau}_{2}}[w])}_{L^{d}(B_{\ti{\tau}_{2}}\setminus B_{\ti{\tau}_{1}})}\le \frac{c}{(\tau_{2}-\tau_{1})^{n\left(\frac{1}{2}-\frac{1}{d}\right)}}\nr{V_{s,p}(Dw)}_{L^{2}(B_{\tau_{2}}\setminus B_{\tau_{1}})},
\end{eqnarray}
for $c\equiv c(n,p,d)$. Clearly, operator $\mf{T}_{\ti{\tau}_{1},\ti{\tau}_{2}}$ satisfies properties \emph{(}i.\emph{)}-\emph{(}iv.\emph{)} and \eqref{ex.0} with $\ti{\tau}_{1}$, $\ti{\tau}_{2}$ substituting $\tau_{1}$, $\tau_{2}$.
\end{lemma}
In the next lemma we derive a preliminary version of Caccioppoli inequality that will be eventually adjusted depending on the singular/nonsingular behavior of $\bar{\mathcal{F}}(\cdot)$. 

\begin{lemma}\label{l.0}
Assume \eqref{assf}-\eqref{sqc}, \eqref{assf.0} and \eqref{f}, let $u\in W^{1,p}(\Omega,\mathbb{R}^{N})$ be a local minimizer of \eqref{exfun}, $B_{\rr}(x_{0})\Subset \Omega$ be any ball, $\rr/2\le \tau_{1}<\tau_{2}\le \rr$ be parameters and $\ell(x):=v_{0}+\langle z_{0},x-x_{0}\rangle$ be an affine function with $z_{0}\in \left\{z\in \mathbb{R}^{N\times n}\colon\snr{z}\le \left(80000(M+1)\right)^{2/p} \right\}$ for some positive constant $M$, and $v_{0}\in \mathbb{R}^{N}$. Then
\begin{eqnarray}\label{cacc.ns}
\ti{\mf{F}}(u,z_{0};B_{\rr/2}(x_{0}))^{2}&\le&c\mf{K}\left(\mint_{B_{\rr}(x_{0})}\left| \ V_{\snr{z_{0}},p}\left(\frac{u-\ell}{\rr}\right) \ \right|^{2} \ \dx\right)+c\mathds{1}_{\{q>p\}}\ti{\mf{F}}(u,z_{0};B_{\rr}(x_{0}))^{2q/p}\nonumber \\
&&+c\left[\left(\rr^{m}\mint_{B_{\rr}(x_{0})}\snr{f}^{m} \ \dx\right)^{\frac{p}{m(p-1)}}+\snr{z_{0}}^{2-p}\left(\rr^{m}\mint_{B_{\rr}(x_{0})}\snr{f}^{m} \ \dx\right)^{2/m}\right]
\end{eqnarray}
holds with $c\equiv c(\textnormal{\texttt{data}},M)$ and $\mf{K}(\cdot)$ being defined in \eqref{ik}.
\end{lemma}
\begin{proof}
With $\tau_{1},\tau_{2}$ as in the statement, $\ti{\tau}_{1},\ti{\tau}_{2}$ as in \eqref{ex.2}\footnote{The negligible set $\mathcal{N}$ can be defined as in \cite[Lemma 3.1]{deqc} or in \cite[Lemma 7.13]{ts1}.}, 
let $\eta\in C^{1}_{c}(B_{\ti{\tau}_{2}}(x_{0}))$ be a cut-off function satisfying
$$
\mathds{1}_{B_{\ti{\tau}_{1}}(x_{0})}\le \eta\le \mathds{1}_{B_{\ti{\tau}_{2}}(x_{0})},\qquad  \snr{D\eta}\lesssim \frac{1}{(\ti{\tau}_{2}-\ti{\tau}_{1})},
$$
set $S(x_{0}):=B_{\tau_{2}}(x_{0})\setminus B_{\tau_{1}}(x_{0})$, $\ti{S}(x_{0}):=B_{\ti{\tau}_{2}}(x_{0})\setminus B_{\ti{\tau}_{1}}(x_{0})$, $\mf{u}(x):=u(x)-\ell(x)$ and introduce the comparison maps
$$
\varphi_{1}(x):=\mf{T}_{\ti{\tau}_{1},\ti{\tau}_{2}}[(1-\eta)\mf{u}](x),\qquad \varphi_{2}(x):=\mf{u}(x)-\varphi_{1}(x).
$$
By Lemma \ref{exlem} (\emph{ii.})-(\emph{iii.}) it is
\begin{flalign}\label{0.0.0}
\varphi_{1}\equiv 0 \ \  \mbox{on} \ \  B_{\ti{\tau}_{1}}(x_{0}),\quad\varphi_{2}\in W^{1,p}_{0}(B_{\ti{\tau}_{2}}(x_{0}),\mathbb{R}^{N}),\quad\varphi_{2}\equiv \uu \ \ \mbox{on} \ \ B_{\ti{\tau}_{1}}(x_{0}), \quad D\uu=D\varphi_{1}+D\varphi_{2}.
\end{flalign}
Moreover, by \eqref{assf}-\eqref{sqc}, \eqref{f} and Lemma \ref{exlem},
we see that the construction developed in \cite[Lemma 3.1]{deqc}, \cite[Lemma 7.13]{ts1} applies to our setting as well and renders:
\begin{eqnarray*}
c\int_{B_{\ti{\tau}_{2}}(x_{0})}\snr{V_{\snr{z_{0}},p}(D\varphi_{2})}^{2} \ \dx&\le&\int_{\ti{S}(x_{0})}\left[F(Du-D\varphi_{1})-F(Du)\right] \ \dx\nonumber \\
&&+\int_{\ti{S}(x_{0})}\left[F(z_{0}+D\varphi_{1})-F(z_{0})\right] \ \dx\nonumber \\
&&+\int_{B_{\ti{\tau}_{2}}(x_{0})}f\cdot \varphi_{2} \ \dx=:\left[\mbox{(I)}+\mbox{(II)}+\mbox{(III)}\right],
\end{eqnarray*}
with $c\equiv c(n,N,p,q,\lambda,\Lambda)$. Before proceeding further, let us notice that 
\begin{flalign}\label{x.0}
\mathcal{D}_{\snr{z_{0}}}(z_{1},z_{2})^{(q-p)/2}\lesssim \left(\mathcal{D}_{\snr{z_{0}}}(z_{1},z_{2})^{(p-2)/2}\mathcal{D}_{0}(z_{1},z_{2})\right)^{(q-p)/p}+\snr{z_{0}}^{q-p},
\end{flalign}
for all $z_{0},z_{1},z_{2}\in \mathbb{R}^{N\times n}$, where $\mathcal{D}_{\snr{z_{0}}}(\cdot)$ has been defined in Section \ref{notation} and the constants implicit in "$\approx$, $\lesssim$" depend on $(n,N,p,q)$, cf. \cite[page 256]{ts2}. We then rearrange:
\begin{eqnarray*}
\mbox{(I)}+\mbox{(II)}&=&\int_{\ti{S}(x_{0})}\left\langle\left(\int_{0}^{1}\left[\partial F(z_{0})-\partial F(z_{0}+D\uu-s D\varphi_{1})\right] \ \ds\right), D\varphi_{1}\right\rangle \ \dx\nonumber \\
&&+\int_{\ti{S}(x_{0})}\left\langle\left(\int_{0}^{1}\left[\partial F(z_{0}+s D\varphi_{1})-\partial F(z_{0})\right] \ \ds\right), D\varphi_{1}\right\rangle\  \dx=:\mbox{(I')}+\mbox{(II')},
\end{eqnarray*}
and estimate (keep in mind the upper bound on $\snr{z_{0}}$),
\begin{eqnarray*}
\snr{\mbox{(I')}}&\stackrel{\eqref{assf.0.1}_{2}}{\le}&c\int_{\ti{S}(x_{0})}\int_{0}^{1}\left[\frac{\snr{V_{\snr{z_{0}},p}(D\mf{u}-sD\varphi_{1})}^{2}}{\snr{D\mf{u}-sD\varphi_{1}}}+\frac{\snr{V_{\snr{z_{0}},q}(D\mf{u}-sD\varphi_{1})}^{2}}{\snr{D\mf{u}-sD\varphi_{1}}}\right] \ \ds\snr{D\varphi_{1}} \ \dx\nonumber \\
&\stackrel{\eqref{Vm}_{1,5},\eqref{l6}}{\le}&c\int_{\ti{S}(x_{0})}\left(\frac{\snr{V_{\snr{z_{0}},p}(D\mf{u})}^{2}\snr{D\varphi_{1}}}{\snr{D\mf{u}}}+\snr{V_{\snr{z_{0}},p}(D\varphi_{1})}^{2}\right) \ \dx\nonumber \\
&&+c\int_{\ti{S}(x_{0})}\left(\int_{0}^{1}(\snr{z_{0}}^{2}+\snr{D\mf{u}-sD\varphi_{1}}^{2})^{(q-2)/2} \ \ds\right)(\snr{D\mf{u}}+\snr{D\varphi_{1}})\snr{D\varphi_{1}} \ \dx\nonumber \\
&\stackrel{\eqref{Vm}_{4},\eqref{l6}}{\le}&c\int_{\ti{S}(x_{0})}\snr{V_{\snr{z_{0}},p}(D\mf{u})}^{2}+\snr{V_{\snr{z_{0}},p}(D\varphi_{1})}^{2} \ \dx\nonumber \\
&&+c\int_{\ti{S}(x_{0})}(\snr{z_{0}}^{2}+\snr{D\mf{u}}^{2}+\snr{D\varphi_{1}}^{2})^{(q-2)/2}(\snr{D\mf{u}}+\snr{D\varphi_{1}})\snr{D\varphi_{1}} \ \dx\nonumber \\
&\stackrel{\eqref{x.0}}{\le}&c\int_{\ti{S}(x_{0})}\snr{V_{\snr{z_{0}},p}(D\mf{u})}^{2}+\snr{V_{\snr{z_{0}},p}(D\varphi_{1})}^{2} \ \dx\nonumber \\
&&+c\int_{\ti{S}(x_{0})}\snr{z_{0}}^{q-p}\mathcal{D}_{\snr{z_{0}}}(D\mf{u},D\varphi_{1})^{(p-2)/2}(\snr{D\mf{u}}+\snr{D\varphi_{1}})\snr{D\varphi_{1}} \ \dx\nonumber \\
&&+c\int_{\ti{S}(x_{0})}\mathcal{D}_{\snr{z_{0}}}(D\mf{u},D\varphi_{1})^{\frac{q(p-2)}{2p}}\mathcal{D}_{0}(D\mf{u},D\varphi_{1})^{(q-p)/p}(\snr{D\mf{u}}+\snr{D\varphi_{1}})\snr{D\varphi_{1}} \ \dx\nonumber \\
&\stackrel{\eqref{pq}_{1}}{\le}&c\int_{\ti{S}(x_{0})}(1+\snr{z_{0}}^{q-p})\left(\snr{V_{\snr{z_{0}},p}(D\mf{u})}^{2}+\snr{V_{\snr{z_{0}},p}(D\varphi_{1})}^{2} \right)\ \dx\nonumber \\
&&+c\int_{\ti{S}(x_{0})}\mathcal{D}_{\snr{z_{0}}}(D\mf{u},D\varphi_{1})^{\frac{q(p-2)}{2p}}\left(\snr{D\mf{u}}^{\frac{2q}{p}-1}\snr{D\varphi_{1}}+\snr{D\varphi_{1}}^{2q/p}\right) \ \dx\nonumber \\
&\le&c\int_{\ti{S}(x_{0})}(1+\snr{z_{0}}^{q-p})\left(\snr{V_{\snr{z_{0}},p}(D\mf{u})}^{2}+\snr{V_{\snr{z_{0}},p}(D\varphi_{1})}^{2} \right)\ \dx\nonumber \\
&&+c\int_{\ti{S}(x_{0})}\snr{V_{\snr{z_{0}},p}(D\mf{u})}^{\frac{2q}{p}-1}\snr{V_{\snr{z_{0}},p}(D\varphi_{1})}+\snr{V_{\snr{z_{0}},p}(D\varphi_{1})}^{2q/p} \ \dx\nonumber \\
&\stackrel{\eqref{ex.1.2}}{\le}&c(1+\snr{z_{0}}^{q-p})\int_{\ti{S}(x_{0})}\snr{V_{\snr{z_{0}},p}(D\mf{u})}^{2}+\left|\ V_{\snr{z_{0}},p}\left(\frac{\mf{u}}{\tau_{2}-\tau_{1}}\right)\ \right|^{2} \ \dx\nonumber \\
&&+c\left(\int_{\ti{S}(x_{0})}\snr{V_{\snr{z_{0}},p}(D\mf{u})}^{2} \ \dx\right)^{(2q-p)/2p}\left(\int_{\ti{S}(x_{0})}\snr{V_{\snr{z_{0}},p}(D\varphi_{1})}^{\frac{2p}{3p-2q}} \ \dx\right)^{\frac{3p-2q}{2p}}\nonumber \\
&&+\frac{c\mathds{1}_{\{q>p\}}}{(\tau_{2}-\tau_{1})^{n\left(\frac{q}{p}-1\right)}}\left(\int_{S(x_{0})}\snr{V_{\snr{z_{0}},p}(D\mf{u})}^{2}+\left|\ V_{\snr{z_{0}},p}\left(\frac{\mf{u}}{\tau_{2}-\tau_{1}}\right)\ \right|^{2} \ \dx\right)^{q/p}\nonumber \\
&\stackrel{\eqref{pq}}{\le}&c(1+\snr{z_{0}}^{q-p})\int_{\ti{S}(x_{0})}\snr{V_{\snr{z_{0}},p}(D\mf{u})}^{2}+\left|\ V_{\snr{z_{0}},p}\left(\frac{\mf{u}}{\tau_{2}-\tau_{1}}\right)\ \right|^{2} \ \dx\nonumber \\
&&+\frac{c\mathds{1}_{\{q>p\}}}{(\tau_{2}-\tau_{1})^{n\left(\frac{q}{p}-1\right)}}\left(\int_{S(x_{0})}\snr{V_{\snr{z_{0}},p}(D\mf{u})}^{2}+\left|\ V_{\snr{z_{0}},p}\left(\frac{\mf{u}}{\tau_{2}-\tau_{1}}\right)\ \right|^{2} \ \dx\right)^{q/p},
\end{eqnarray*}
where we also used \eqref{assf.0.1}$_{2}$ with $L\equiv 80000(M+1)$, and exploited that by $\eqref{pq}_{2}$ it is $\max\left\{2q/p,2p/(3p-2q)\right\}<2n/(n-1)$, $\min\left\{2q/p,2p/(3p-2q)\right\}\ge 2/p$ and $c\equiv c(n,N,\lambda,\Lambda,p,q,M)$. In a totally similar way we bound:
\begin{eqnarray*}
\snr{\mbox{(II')}}&\stackrel{\eqref{assf.0.1}_{2},\eqref{Vm}_{5}}{\le}&c\int_{\ti{S}(x_{0})}\snr{V_{\snr{z_{0}},p}(D\varphi_{1})}^{2}+(\snr{z_{0}}^{2}+\snr{D\varphi_{1}}^{2})^{(q-2)/2}\snr{D\varphi_{1}}^{2} \ \dx\nonumber \\
&\stackrel{\eqref{x.0},\eqref{ex.1.2}}{\le}&c(1+\snr{z_{0}}^{q-p})\int_{\ti{S}(x_{0})}\snr{V_{\snr{z_{0}},p}(D\mf{u})}^{2}+\left|\ V_{\snr{z_{0}},p}\left(\frac{\mf{u}}{\tau_{2}-\tau_{1}}\right)\ \right|^{2} \ \dx\nonumber \\
&&+c\int_{\ti{S}(x_{0})}\snr{V_{\snr{z_{0}},p}(D\varphi_{1})}^{2q/p} \ \dx\nonumber \\
&\stackrel{\eqref{pq},\eqref{ex.1.2}}{\le}&c(1+\snr{z_{0}}^{q-p})\int_{\ti{S}(x_{0})}\snr{V_{\snr{z_{0}},p}(D\mf{u})}^{2}+\left|\ V_{\snr{z_{0}},p}\left(\frac{\mf{u}}{\tau_{2}-\tau_{1}}\right)\ \right|^{2} \ \dx\nonumber \\
&&+\frac{c\mathds{1}_{\{q>p\}}}{(\tau_{2}-\tau_{1})^{n\left(\frac{q}{p}-1\right)}}\left(\int_{S(x_{0})}\snr{V_{\snr{z_{0}},p}(D\mf{u})}^{2}+\left|\ V_{\snr{z_{0}},p}\left(\frac{\mf{u}}{\tau_{2}-\tau_{1}}\right)\ \right|^{2} \ \dx\right)^{q/p}
\end{eqnarray*}
for $c\equiv c(n,N,\lambda,\Lambda,p,q,M)$. Concerning term $\mbox{(III)}$, we use \eqref{f}, \eqref{f.0}, \eqref{0.0.0}$_{2}$, Sobolev-Poincar\'e inequality, Young inequality and \eqref{Vm.1} with $s=\snr{z_{0}}$, $z_{1}=0$, $z_{2}=D\varphi_{2}$ to estimate
\begin{eqnarray*}
\snr{\mbox{(III)}}&\le&c\snr{B_{\ti{\tau}_{2}}(x_{0})}\left(\ti{\tau}_{2}^{m}\mint_{B_{\ti{\tau}_{2}}(x_{0})}\snr{f}^{m} \ \dx\right)^{1/m}\left(\mint_{B_{\ti{\tau}_{2}}(x_{0})}\snr{D\varphi_{2}}^{p} \ \dx\right)^{1/p}\nonumber \\
&\le&c\snr{B_{\ti{\tau}_{2}}(x_{0})}\left(\ti{\tau}_{2}^{m}\mint_{B_{\ti{\tau}_{2}}(x_{0})}\snr{f}^{m} \ \dx\right)^{1/m}\left(\mint_{B_{\ti{\tau}_{2}}(x_{0})}\snr{V_{\snr{z_{0}},p}(D\varphi_{2})}^{2}\ \dx\right)^{1/p}\nonumber \\
&&+\snr{B_{\ti{\tau}_{2}}(x_{0})}\left(\ti{\tau}_{2}^{m}\mint_{B_{\ti{\tau}_{2}}(x_{0})}\snr{f}^{m} \ \dx\right)^{1/m}\left(\mint_{B_{\ti{\tau}_{2}}(x_{0})}\snr{V_{\snr{z_{0}},p}(D\varphi_{2})}^{p}\snr{z_{0}}^{p(2-p)/2}\ \dx\right)^{1/p}\nonumber \\
&\le&\frac{1}{4}\int_{B_{\ti{\tau}_{2}}(x_{0})}\snr{V_{\snr{z_{0}},p}(D\varphi_{2})}^{2} \ \dx+c\snr{B_{\ti{\tau}_{2}}(x_{0})}\left(\ti{\tau}_{2}^{m}\mint_{B_{\ti{\tau}_{2}}(x_{0})}\snr{f}^{m} \ \dx\right)^{\frac{p}{m(p-1)}}\nonumber \\
&&+c\snr{B_{\ti{\tau}_{2}}(x_{0})}\snr{z_{0}}^{2-p}\left(\ti{\tau}_{2}^{m}\mint_{B_{\ti{\tau}_{2}}(x_{0})}\snr{f}^{m} \ \dx\right)^{2/m},
\end{eqnarray*}
with $c\equiv c(n,N,p,m)$. Merging the content of the previous displays, reabsorbing terms and recalling \eqref{0.0.0}$_{3}$, \eqref{ex.2} and the upper bound imposed on the size of $\snr{z_{0}}$, we obtain 
\begin{flalign*}
\int_{B_{\tau_{1}}(x_{0})}&\snr{V_{\snr{z_{0}},p}(D\mf{u})}^{2} \ \dx\le c\int_{B_{\tau_{2}}(x_{0})\setminus B_{\tau_{1}}(x_{0})}\snr{V_{\snr{z_{0}},p}(D\mf{u})}^{2}+\left|\ V_{\snr{z_{0}},p}\left(\frac{\mf{u}}{\tau_{2}-\tau_{1}}\right)\ \right|^{2} \ \dx\nonumber \\
&\qquad +\frac{c\mathds{1}_{\{q>p\}}}{(\tau_{2}-\tau_{1})^{n\left(\frac{q}{p}-1\right)}}\left(\int_{B_{\tau_{2}}(x_{0})\setminus B_{\tau_{1}}(x_{0})}\snr{V_{\snr{z_{0}},p}(D\mf{u})}^{2}+\left|\ V_{\snr{z_{0}},p}\left(\frac{\mf{u}}{\tau_{2}-\tau_{1}}\right)\ \right|^{2} \ \dx\right)^{q/p}\nonumber \\
&\qquad +c\snr{B_{\tau_{2}}(x_{0})}\left[\left(\tau_{2}^{m}\mint_{B_{\tau_{2}}(x_{0})}\snr{f}^{m} \ \dx\right)^{\frac{p}{m(p-1)}}+\snr{z_{0}}^{2-p}\left(\tau_{2}^{m}\mint_{B_{\tau_{2}}(x_{0})}\snr{f}^{m} \ \dx\right)^{2/m}\right],
\end{flalign*}
for $c\equiv c(\textnormal{\texttt{data}},M)$. 
We then sum to both sides of the above inequality the quantity $c\int_{B_{\tau_{1}}(x_{0})}\snr{V_{\snr{z_{0}},p}(D\mf{u})}^{2} \ \dx$ and use Lemma \ref{l5} to conclude with \eqref{cacc.ns}.
\end{proof}
\section{The nonsingular regime}
Let us prove the approximate $\mathcal{A}$-harmonic character of minima of \eqref{exfun} within the nonsingular scenario.
\begin{lemma}\label{aahar}
Under assumptions \eqref{assf}-\eqref{sqc}, \eqref{assf.0} and \eqref{f}, let $u\in W^{1,p}(\Omega,\mathbb{R}^{N})$ be a local minimizer of \eqref{exfun}, $B_{\rr}(x_{0})\Subset \Omega$ be a ball and $z_{0}\in \left(\mathbb{R}^{N\times n}\setminus \{0\}\right)\cap \left\{\snr{z}\le \left(80000(M+1)\right)^{2/p }\right\}$ for some positive constant $M$ be any matrix such that $\ti{\mf{F}}(u,z_{0};B_{\rr}(x_{0}))>0$. Then
\begin{flalign*}
&\left| \ \mint_{B_{\rr}(x_{0})}\frac{\partial^{2}F(z_{0})}{\snr{z_{0}}^{p-2}}\left\langle\frac{\snr{z_{0}}^{(p-2)/2}(Du-z_{0})}{\ti{\mf{F}}(u,z_{0};B_{\rr}(x_{0}))},D\varphi\right\rangle \ \dx \ \right| \le \frac{c\snr{z_{0}}^{(2-p)/2}\nr{D\varphi}_{L^{\infty}(B_{\rr}(x_{0}))}}{\ti{\mf{F}}(u,z_{0};B_{\rr}(x_{0}))}\left(\rr^{m}\mint_{B_{\rr}(x_{0})}\snr{f}^{m} \ \dx\right)^{1/m}\nonumber \\
&\qquad \qquad +c\left[\mu_{M}\left(\frac{\ti{\mf{F}}(u,z_{0};B_{\rr}(x_{0}))^{2}}{\snr{z_{0}}^{p}}\right)+\left(\frac{\ti{\mf{F}}(u,z_{0};B_{\rr}(x_{0}))^{2}}{\snr{z_{0}}^{p}}\right)^{(2\kk-1)/2}\right]\nr{D\varphi}_{L^{\infty}(B_{\rr}(x_{0}))},
\end{flalign*}
where $\kk:=(q-1)/p$ if $q\ge 2$ and $\kk:=1/p$ when $1<q<2$, and $c\equiv c(\textnormal{\texttt{data}},M)$.
\end{lemma}
\begin{proof}
Let $\varphi\in C^{\infty}_{c}(B_{\rr}(x_{0}),\mathbb{R}^{N})$ be a map and, for the ease of reading, let us shorten $B_{\rr}(x_{0})\equiv B_{\rr}$, $\ti{\mf{F}}(u,z_{0},B_{\rr}(x_{0}))\equiv \ti{\mf{F}}_{0}(u)$, $\ell(x):=v_{0}+\langle z_{0},x-x_{0}\rangle$ with $v_{0}\in \mathbb{R}^{N}$, $\mf{u}:=u-\ell$ and $\nr{D\varphi}_{L^{\infty}(B_{\rr}(x_{0}))}\equiv \nr{D\varphi}_{\infty}$. Set 
$$
B^{-}:=B_{\rr}\cap \left\{\snr{D\mf{u}}\le \snr{z_{0}}\right\},\qquad \quad B^{+}:=B_{\rr}\cap \left\{\snr{D\mf{u}}>\snr{z_{0}}\right\}
$$
and bound
\begin{eqnarray*}
\mathcal{I}&:=&\left| \ \mint_{B_{\rr}}\partial^{2}F(z_{0})\langle Du-z_{0},D\varphi\rangle  \ \dx\ \right|\nonumber \\
&\stackrel{\eqref{el}}{=}&\left| \ \mint_{B_{\rr}}\partial^{2}F(z_{0})\langle Du-z_{0},D\varphi\rangle-\langle\partial F(Du)-\partial F(z_{0}),D\varphi\rangle\ \dx+\mint_{B_{\rr}(x_{0})}f\cdot \varphi \ \dx \ \right|\nonumber \\
&\le&\mint_{B_{\rr}}\snr{\partial^{2}F(z_{0})\langle Du-z_{0},D\varphi\rangle-\langle\partial F(Du)-\partial F(z_{0}),D\varphi\rangle}\ \dx+\mint_{B_{\rr}(x_{0})}\snr{f\cdot \varphi} \ \dx\nonumber \\
&=&\frac{1}{\snr{B_{\rr}}}\int_{B^{-}}\snr{\partial^{2}F(z_{0})\langle Du-z_{0},D\varphi\rangle-\langle\partial F(Du)-\partial F(z_{0}),D\varphi\rangle}\ \dx\nonumber \\
&&+\frac{1}{\snr{B_{\rr}}}\int_{B^{+}}\snr{\partial^{2}F(z_{0})\langle Du-z_{0},D\varphi\rangle-\langle\partial F(Du)-\partial F(z_{0}),D\varphi\rangle}\ \dx\nonumber \\
&&+\mint_{B_{\rr}(x_{0})}\snr{f\cdot \varphi} \ \dx=:\mathcal{I}_{1} +\mathcal{I}_{2} +\mathcal{I}_{3},
\end{eqnarray*}
where we also used that $\int_{B_{\rr}}\langle\partial F(z_{0}),D\varphi\rangle \ \dx =0$. We then estimate
\begin{eqnarray*}
\mathcal{I}_{1}&\le&\frac{\nr{D\varphi}_{\infty}}{\snr{B_{\rr}}}\int_{B^{-}}\left(\int_{0}^{1}\snr{\partial^{2}F(z_{0})-\partial^{2}F(z_{0}+sD\mf{u})} \ \ds\right)\snr{D\mf{u}} \ \dx\nonumber \\
&\stackrel{\eqref{assf.0.3}}{\le}&\frac{c\nr{D\varphi}_{\infty}}{\snr{B_{\rr}}}\int_{B^{-}}\mu_{M}\left(\frac{\snr{V_{\snr{z_{0}}}(D\mf{u})}^{2}}{\snr{z_{0}}^{p}}\right)\left(\int_{0}^{1}\snr{z_{0}+sD\mf{u}}^{p-2} \ \ds\right)\snr{D\mf{u}} \ \dx\nonumber \\
&\stackrel{\eqref{l6}}{\le}&c\snr{z_{0}}^{(p-2)/2}\nr{D\varphi}_{\infty}\mint_{B_{\rr}}\mu_{M}\left(\frac{\snr{V_{\snr{z_{0}}}(D\mf{u})}^{2}}{\snr{z_{0}}^{p}}\right)\snr{V_{\snr{z_{0}},p}(D\mf{u})} \ \dx\nonumber \\
&\le&c\snr{z_{0}}^{(p-2)/2}\ti{\mf{F}}_{0}(u)\nr{D\varphi}_{\infty}\left(\mint_{B_{\rr}}\mu_{M}\left(\frac{\snr{V_{\snr{z_{0}},p}(D\mf{u})}^{2}}{\snr{z_{0}}^{p}} \ \right)^{2} \ \dx\right)^{1/2}\nonumber \\
&\le&c\snr{z_{0}}^{(p-2)/2}\ti{\mf{F}}_{0}(u)\mu_{M}\left(\frac{\ti{\mf{F}}_{0}(u)^{2}}{\snr{z_{0}}^{p}}\right)\nr{D\varphi}_{\infty},
\end{eqnarray*}
for $c\equiv c(n,N,p,F(\cdot),M)$. We remark that the convergence of the singular integral in the previous display can be justified as in \cite[Section 4]{dumi}. Moreover, when applying \eqref{assf.0.3} above we chose $L\equiv 80000(M+1)$ and consequently denote $\mu_{L}(\cdot)$ as $\mu_{M}(\cdot)$. Concerning term $\mathcal{I}_{2}$, notice that
\begin{eqnarray}\label{xxxx}
\snr{z}\stackrel{\eqref{pq}_{1}}{\le} 2\snr{V_{\snr{z_{0}},p}(z)}^{2/p}\qquad \mbox{for all} \ \ z\in \mathbb{R}^{N\times n}\cap\left\{\snr{z}\ge \snr{z_{0}}\right\},
\end{eqnarray}
so we have
\begin{eqnarray*}
\mathcal{I}_{2}&\stackrel{\eqref{assf.0.1}_{2},\eqref{assf.0.4}}{\le}&\frac{c\nr{D\varphi}_{\infty}}{\snr{B_{\rr}}}\int_{B^{+}}\left(\snr{z_{0}}^{p-2}+\frac{\snr{V_{\snr{z_{0}},p}(D\mf{u})}^{2}}{\snr{D\mf{u}}^{2}}+\frac{\snr{V_{\snr{z_{0}},q}(D\mf{u})}^{2}}{\snr{D\mf{u}}^{2}}\right)\snr{D\mf{u}} \ \dx\nonumber \\
&\stackrel{\eqref{xxxx}}{\le}&\frac{c\nr{D\varphi}_{\infty}}{\snr{B_{\rr}}}\int_{B^{+}}\left(1+\left(1-\mathds{1}_{\{q\ge 2\}}\right)\snr{z_{0}}^{q-p}\right)\snr{z_{0}}^{p-2}\snr{V_{\snr{z_{0}},p}(D\mf{u})}^{2/p} \ \dx\nonumber \\
&&+\frac{c\nr{D\varphi}_{\infty}}{\snr{B_{\rr}}}\int_{B^{+}}\mathds{1}_{\{q\ge 2\}}\snr{V_{\snr{z_{0}},p}(D\mf{u})}^{2(q-1)/p} \ \dx\nonumber \\
&\le&\frac{c\nr{D\varphi}_{\infty}(1+\snr{z_{0}}^{q-p})}{\snr{B_{\rr}}}\int_{B^{+}}\snr{z_{0}}^{p-2}\snr{V_{\snr{z_{0}},p}(D\mf{u})}^{2/p} \ \dx\nonumber \\
&&+\frac{c\nr{D\varphi}_{\infty}\snr{z_{0}}^{q-p}}{\snr{B_{\rr}}}\int_{B^{+}}\frac{\mathds{1}_{\{q\ge 2\}}\snr{V_{\snr{z_{0}},p}(D\mf{u})}^{2(q-1)/p}}{\snr{z_{0}}^{q-p}} \ \dx\nonumber \\
&\stackrel{\eqref{pq}}{\le}&c\nr{D\varphi}_{\infty}\snr{z_{0}}^{p-p\kk-1}\ti{\mf{F}}_{0}(u)^{2\kk},
\end{eqnarray*}
where we used that $\snr{z_{0}}\lesssim (M+1)$. In the above display, $\kk$ is defined as in the statement and $c\equiv c(\textnormal{\texttt{data}},M)$. Trivially, we also get
\begin{eqnarray*}
\mathcal{I}_{3}\le 4\nr{D\varphi}_{\infty}\left(\rr^{m}\mint_{B_{\rr}}\snr{f}^{m} \ \dx\right)^{1/m}.
\end{eqnarray*}
Merging the content of the previous displays, dividing both sides of the resulting inequality by $\snr{z_{0}}^{(p-2)/2}\ti{\mf{F}}_{0}(u)$ and recalling that by \eqref{pq}$_{1}$ it is $2\kk> 1$ we obtain \eqref{aahar} and the proof is complete.
\end{proof}
Now we are ready to prove a one-scale decay result valid in the nonsingular case. To this end, a fundamental observation is that under proper smallness conditions, local minimizers of \eqref{exfun} are approximately $\mathcal{A}$-harmonic in the sense of \eqref{con.0} for a suitable choice of the bilinear form $\mathcal{A}$.
\begin{proposition}\label{p1}
Under hypotheses \eqref{assf}-\eqref{sqc}, \eqref{assf.0} and \eqref{f} and let $u\in W^{1,p}(\Omega,\mathbb{R}^{N})$ be a local minimizer of \eqref{exfun} satisfying
\begin{eqnarray}\label{ns.1}
\snr{(V_{p}(Du))_{B_{\rr}(x_{0})}}\le 40000(M+1)
\end{eqnarray}
for some $M> 0$ on a ball $B_{\rr}(x_{0})\Subset \Omega$. Then, for any $\beta\in (0,1)$ there are $\tau\equiv \tau,(\textnormal{\texttt{data}},M,\beta)\in (0,1/16)$ and $\varepsilon_{0},\varepsilon_{1}\equiv \varepsilon_{0},\varepsilon_{1}(\textnormal{\texttt{data}},M,\beta)\in (0,1]$ such that if the smallness conditions
\begin{eqnarray}\label{ns.0}
\mf{F}(u;B_{\rr}(x_{0}))<\varepsilon_{0}\snr{(V_{p}(Du))_{B_{\rr}(x_{0})}}
\end{eqnarray}
and
\begin{eqnarray}\label{ns.2}
\left(\rr^{m}\mint_{B_{\rr}(x_{0})}\snr{f}^{m} \ \dx\right)^{1/m}\le \varepsilon_{1}\mf{F}(u;B_{\rr}(x_{0}))\snr{(V_{p}(Du))_{B_{\rr}(x_{0})}}^{(p-2)/p}
\end{eqnarray}
are verified on $B_{\rr}(x_{0})$, it holds that
\begin{eqnarray}\label{nsns}
\mf{F}(u;B_{\tau\rr}(x_{0}))\le \tau^{\beta}\mf{F}(u;B_{\rr}(x_{0})).
\end{eqnarray}
\end{proposition}
\begin{proof}
For the transparency of the exposition, let us introduce some abbreviations. As all balls considered here will be concentric to $B_{\rr}(x_{0})$, we shall omit denoting the center, given any ball $B_{\varsigma}(x_{0})\subseteq B_{\rr}(x_{0})$ we shorten $(V_{p}(Du))_{B_{\varsigma}(x_{0})}\equiv (V_{p}(Du))_{\varsigma}$ and for all $\varphi\in C^{\infty}_{c}(B_{\rr},\mathbb{R}^{N})$ we denote $\nr{D\varphi}_{L^{\infty}(B_{\rr})}\equiv \nr{D\varphi}_{\infty}$. In the light of \eqref{ns.0}, there is no loss of generality in assuming that 
\begin{eqnarray}\label{ns.3}
\snr{(V_{p}(Du))_{\rr}}>0\qquad \mbox{and}\qquad \mf{F}(u;B_{\rr})>0,
\end{eqnarray}
otherwise \eqref{nsns} would be trivially true because of $\eqref{tri.1}_{1}$. Being $V_{p}(\cdot)$ an isomorphism of $\mathbb{R}^{N\times n}$, we can find $\bar{z}\in \mathbb{R}^{N\times n}\setminus \{0\}$ such that $V_{p}(\bar{z})=(V_{p}(Du))_{\rr}$ and
\begin{eqnarray}\label{ns.4}
\snr{\bar{z}}\stackrel{\eqref{ns.3}_{1}}{>}0,\qquad \ti{\mf{F}}(u,\bar{z};B_{\rr})\stackrel{\eqref{equiv.22}}{\approx} \mf{F}(u;B_{\rr})\stackrel{\eqref{ns.3}_{2}}{>}0, \qquad \snr{\bar{z}}=\snr{(V_{p}(Du))_{\rr}}^{2/p}.
\end{eqnarray}
We then define $u_{0}(x):=\snr{\bar{z}}^{-1}\left(u(x)-(u)_{\rr}-\langle \bar{z},x-x_{0}\rangle\right)$ and, motivated by \eqref{ns.4}, we apply Lemma \ref{aahar} to get
\begin{eqnarray*}
\left| \ \mint_{B_{\rr}}\frac{\partial^{2}F(\bar{z})}{\snr{\bar{z}}^{p-2}}\langle Du_{0},D\varphi\rangle \ \dx \ \right|&\le&\frac{c\ti{\mf{F}}(u,\bar{z};B_{\rr})}{\snr{\bar{z}}^{p/2}}\left[\mu_{M}\left(\frac{\ti{\mf{F}}(u,\bar{z};B_{\rr})^{2}}{\snr{\bar{z}}^{p}}\right)+\left(\frac{\ti{\mf{F}}(u,\bar{z};B_{\rr})^{2}}{\snr{\bar{z}}^{p}}\right)^{\frac{2\kk-1}{2}}\right]\nr{D\varphi}_{\infty}\nonumber \\
&&+c\snr{\bar{z}}^{1-p}\left(\rr^{m}\mint_{B_{\rr}}\snr{f}^{m} \ \dx\right)^{1/m}\nr{D\varphi}_{\infty}\nonumber \\
&\stackrel{\eqref{ns.4}_{2,3}}{\le}&\frac{c\mf{F}(u;B_{\rr})}{\snr{(V_{p}(Du))_{\rr}}}\left[\mu_{M}\left(\frac{\mf{F}(u;B_{\rr})^{2}}{\snr{(V_{p}(Du))_{\rr}}^{2}}\right)+\left(\frac{\mf{F}(u;B_{\rr})^{2}}{\snr{(V_{p}(Du))_{\rr}}^{2}}\right)^{\frac{2\kk-1}{2}}\right]\nr{D\varphi}_{\infty}\nonumber \\
&&+c\snr{(V_{p}(Du))_{\rr}}^{2(1-p)/p}\left(\rr^{m}\mint_{B_{\rr}}\snr{f}^{m} \ \dx\right)^{1/m}\nonumber\\
&\stackrel{\eqref{ns.0},\eqref{ns.2}}{\le}&\ti{c}\varepsilon_{0}\left[\mu_{M}(\varepsilon_{0}^{2})+\varepsilon_{0}^{2\kk-1}+\varepsilon_{1}\right]\nr{D\varphi}_{\infty},
\end{eqnarray*}
with $\ti{c}\equiv \ti{c}(\textnormal{\texttt{data}},M)$. Moreover, it holds that
\begin{eqnarray*}
\left(\mint_{B_{\rr}}\snr{V_{1,p}(Du_{0})}^{2} \ \dx\right)^{1/2}=\frac{\ti{\mf{F}}(u,\bar{z};B_{\rr})}{\snr{\bar{z}}^{p/2}}\stackrel{\eqref{ns.4}_{2,3}}{\le}\frac{c_{*}\mf{F}(u;B_{\rr})}{\snr{(V_{p}(Du))_{\rr}}}\stackrel{\eqref{ns.0}}{\le}c_{*}\varepsilon_{0},
\end{eqnarray*}
where $c_{*}\equiv c_{*}(n,N,p)$ is the constant from the upper bound in \eqref{equiv.22}. Now notice that by \eqref{ns.1}, \eqref{ns.4}$_{3}$, \eqref{assf.0.4} with $L=(40000(M+1))^{2/p}$ and \eqref{sqc.1} we see that the bilinear form $\mathcal{A}:=\partial^{2}F(\bar{z})\snr{\bar{z}}^{2-p}$ satisfies \eqref{con.0} for some $H\equiv H(n,N,\lambda,p,F(\cdot),M)\ge 1$. We then set $\sigma:=c_{*}\mf{F}(u;B_{\rr})/\snr{(V_{p}(Du))_{\rr}}$, let $\varepsilon\in (0,1]$ be any number to be determined later on and, recalling that by \eqref{pq}$_{1}$ it is $2\kk>1$, we assume the following smallness conditions:
\begin{eqnarray}\label{ns.5}
\max\{\ti{c},c_{*}\}\varepsilon_{0}<\frac{\delta}{2^{10}}\qquad \mbox{and}\qquad \mu_{M}(\varepsilon_{0}^{2})+\varepsilon_{0}^{2\kk-1}+\varepsilon_{1}\le \frac{1}{2^{10}},
\end{eqnarray}
where $\delta\equiv \delta(n,N,p,\varepsilon)\in (0,1]$ is the small parameter given by Lemma \ref{ahar}, further restrictions on the size of the various parameters appearing in \eqref{ns.5} will be imposed later on. The choice in \eqref{ns.5} requires in particular that $\varepsilon_{1}\in (0,1/3]$, fixes dependency $\varepsilon_{0}\equiv \varepsilon_{0}(\textnormal{\texttt{data}},M,\varepsilon)$ and ultimately gives that
\begin{eqnarray*}
\mint_{B_{\rr}}\snr{V_{1,p}(Du_{0})}^{2} \ \dx \le \sigma^{2}\le 1\quad \mbox{and}\quad \left| \ \mint_{B_{\rr}}\mathcal{A}\langle Du_{0},D\varphi\rangle \ \dx \ \right|\le \delta\nr{D\varphi}_{\infty} \ \ \mbox{for all} \ \ \varphi\in C^{\infty}_{c}(B_{\rr},\mathbb{R}^{N}),
\end{eqnarray*}
so Lemma \ref{ahar} applies: there exists a $\mathcal{A}$-harmonic map $h\in W^{1,p}(B_{\rr},\mathbb{R}^{N})$ such that
\begin{eqnarray}\label{ns.6}
\mint_{B_{\rr}(x_{0})}\snr{V_{1,p}(Dh)}^{2} \ \dx\le c \qquad \mbox{and}\qquad \mint_{B_{\rr}(x_{0})}\left| \ V_{1,p}\left(\frac{u_{0}-\sigma h}{\rr}\right)\ \right|^{2} \ \dx \le c\sigma^{2}\varepsilon,
\end{eqnarray}
for $c\equiv c(n,N,p)$. Let $\tau\in (0,2^{-10})$ be a small number whose size will be determined later on and estimate by \eqref{Vm}$_{5,7}$, $\eqref{ns.6}$, \eqref{y.0}, \eqref{xxxx} with $\snr{z_{0}}=1$ and the mean value theorem:
\begin{flalign*}
&\mint_{B_{2\tau\rr}}\left| \ V_{1,p}\left(\frac{u_{0}-\sigma(h(x_{0})+\langle Dh(x_{0}),x-x_{0}\rangle)}{2\tau\rr}\right) \ \right|^{2} \ \dx\nonumber \\
&\qquad \qquad\le c\mint_{B_{2\tau\rr}}\left| V_{1,p}\left(\frac{\sigma(h-h(x_{0})-\langle Dh(x_{0}),x-x_{0}\rangle)}{2\tau\rr}\right)\ \ \right|^{2} \ \dx+c\mint_{B_{2\tau\rr}}\left| \ V_{1,p}\left(\frac{u_{0}-\sigma h}{2\tau\rr}\right) \ \right|^{2} \ \dx\nonumber \\
&\qquad \qquad \le \frac{c\varepsilon\sigma^{2}}{\tau^{n+2}}+c\sigma^{2}\mint_{B_{2\tau\rr}}\left| \frac{h-h(x_{0})-\langle Dh(x_{0}),x-x_{0}\rangle}{2\tau\rr}\right|^{2} \ \dx\le\frac{c\varepsilon\sigma^{2}}{\tau^{n+2}}+c\sigma^{2}\tau^{2}\rr^{2}\nr{D^{2}h}_{L^{\infty}(B_{4\tau\rr})}^{2}\nonumber \\
&\qquad \qquad \le \frac{c\varepsilon\sigma^{2}}{\tau^{n+2}}+c\sigma^{2}\tau^{2}\mf{I}_{1}(Dh;B_{\rr})^{2}\le  \frac{c\varepsilon\sigma^{2}}{\tau^{n+2}}+c\sigma^{2}\tau^{2},
\end{flalign*}
with $c\equiv c(n,N,p)$. In the above display, we fix $\varepsilon:=\tau^{n+4}$ thus getting
\begin{eqnarray}\label{ns.7}
\mint_{B_{2\tau\rr}}\left| \ V_{1,p}\left(\frac{u_{0}-\sigma(h(x_{0})+\langle Dh(x_{0}),x-x_{0}\rangle)}{2\tau\rr}\right) \ \right|^{2} \ \dx\le c\sigma^{2}\tau^{2},
\end{eqnarray}
for $c\equiv c(n,N,p)$, which yields that
\begin{eqnarray}\label{ns.8}
\mint_{B_{2\tau\rr}}\snr{V_{\snr{\bar{z}},p}(\mathcal{S}(x))}^{2} \ \dx&:=&\mint_{B_{2\tau\rr}}\left| \ V_{\snr{\bar{z}},p}\left(\frac{u-(u)_{\rr}-\langle \bar{z},x-x_{0}\rangle-\sigma\snr{\bar{z}}(h(x_{0})+\langle Dh(x_{0}),x-x_{0}\rangle)}{2\tau\rr}\right) \ \right|^{2} \ \dx\nonumber \\
&=&\snr{\bar{z}}^{p}\mint_{B_{2\tau\rr}}\left| \ V_{1,p}\left(\frac{u_{0}-\sigma(h(x_{0})-\langle Dh(x_{0}),x-x_{0}\rangle)}{2\tau\rr}\right) \ \right|^{2} \ \dx\nonumber \\
&\stackrel{\eqref{ns.7}}{\le}&c\tau^{2}\snr{\bar{z}}^{p}\left(\frac{\mf{F}(u;B_{\rr})}{\snr{(V_{p}(Du))_{\rr}}}\right)^{2}\stackrel{\eqref{ns.4}_{3}}{\le}c\tau^{2}\mf{F}(u;B_{\rr})^{2},
\end{eqnarray}
with $c\equiv c(n,N,p)$. Now, notice that by \eqref{y.0} and $\eqref{ns.6}_{1}$ it is $\snr{Dh(x_{0})}\le c(n,N,p) $, so recalling \eqref{ns.0} and reducing further (with respect to \eqref{ns.5}) the size of $\varepsilon_{0}$ in such a way that $c\sigma \le c\varepsilon_{0} \le \min\{2^{-10},\tau^{n}\}$ we obtain
\eqn{ns.9}
$$\snr{\bar{z}}(1-c\sigma)\le \snr{\bar{z}+\sigma\snr{\bar{z}}Dh(x_{0})}\le \snr{\bar{z}}(1+c\sigma) \ \Longrightarrow \ \frac{1}{2}\snr{\bar{z}}\le \snr{\bar{z}+\sigma\snr{\bar{z}}Dh(x_{0})}\le \frac{3}{2}\snr{\bar{z}}.$$
We can then estimate
\begin{flalign*}
&\left|\  \snr{V_{\snr{\bar{z}},p}(\mathcal{S}(x))}^{2}-\snr{V_{\snr{\bar{z}+\sigma\snr{\bar{z}}Dh(x_{0})},p}(\mathcal{S}(x))}^{2} \ \right|\nonumber \\
&\qquad \qquad \le c\snr{\mathcal{S}(x)}^{2}\sup_{t\in [\snr{\bar{z}}(1-c\sigma),\snr{\bar{z}}(1+c\sigma)]}\left(t^{2}+\snr{\mathcal{S}(x)}^{2}\right)^{(p-3)/2}\left|\snr{\bar{z}}-\snr{\bar{z}+\sigma\snr{\bar{z}}Dh(x_{0})}\right|\nonumber \\
&\qquad \qquad \le c\sigma\snr{\bar{z}}\snr{Dh(x_{0})}\snr{\mathcal{S}(x)}^{2}\left(\snr{\bar{z}}(1-c\sigma)^{2}+\snr{\mathcal{S}(x)}^{2}\right)^{(p-3)/2}\nonumber \\
&\qquad \qquad\le\frac{c\sigma\snr{\bar{z}}\snr{V_{\snr{\bar{z}},p}(\mathcal{S}(x))}^{2}}{(1-c\sigma)^{3-p}(\snr{\mathcal{S}(x)}^{2}+\snr{\bar{z}}^{2})^{1/2}}\le c\snr{V_{\snr{\bar{z}},p}(\mathcal{S}(x))}^{2},
\end{flalign*}
for $c\equiv c(n,N,p)$. This and \eqref{ns.8} imply that
\begin{eqnarray}\label{ns.11}
\mint_{B_{2\tau\rr}}\snr{V_{\snr{\bar{z}+\sigma\snr{\bar{z}}Dh(x_{0})},p}(\mathcal{S}(x))}^{2} \ \dx\le c\mint_{B_{2\tau\rr}}\snr{V_{\snr{\bar{z}},p}(\mathcal{S}(x))}^{2} \ \dx\le c\tau^{2}\mf{F}(u;B_{\rr})^{2},
\end{eqnarray}
with $c\equiv c(n,N,p)$. Next, notice that
\begin{eqnarray}\label{ns.10}
\ti{\mf{F}}(u,\bar{z}+\sigma\snr{\bar{z}}Dh(x_{0});B_{2\tau\rr})^{2}&\stackrel{\eqref{ns.9}}{\le}&c\mint_{B_{2\tau\rr}}\snr{V_{\snr{\bar{z}},p}(Du-\bar{z}-\sigma\snr{\bar{z}}Dh(x_{0}))}^{2} \ \dx\nonumber \\
&\stackrel{\eqref{Vm}_{7}}{\le}&c\mint_{B_{2\tau\rr}}\snr{V_{\snr{\bar{z}},p}(Du-\bar{z})}^{2} \ \dx+c\mint_{B_{2\tau\rr}}\snr{V_{\snr{\bar{z}},p}(\sigma\bar{z}Dh(x_{0}))}^{2} \ \dx\nonumber \\
&\stackrel{\eqref{ns.4}_{2}}{\le}&c\tau^{-n}\mf{F}(u;B_{\rr})^{2}+c\sigma^{2}\snr{\bar{z}}^{p}\stackrel{\eqref{ns.4}_{3}}{\le} c\tau^{-n}\mf{F}(u;B_{\rr})^{2},
\end{eqnarray}
for $c\equiv c(n,N,p)$. At this stage, keeping in mind \eqref{ns.9} we apply Caccioppoli inequality \eqref{cacc.ns} to bound
\begin{eqnarray*}
\ti{\mf{F}}(u,\bar{z}+\sigma\snr{\bar{z}}Dh(x_{0});B_{\tau\rr})^{2}&\le&c\mf{K}\left(\mint_{B_{2\tau\rr}}\left| \ V_{\snr{\bar{z}+\sigma\snr{\bar{z}}Dh(x_{0})},p}(\mathcal{S}(x)) \ \right|^{2} \ \dx\right)\nonumber \\
&&+c\mathds{1}_{\{q>p\}}\ti{\mf{F}}(u,\bar{z}+\sigma\snr{\bar{z}}Dh(x_{0});B_{2\tau\rr})^{2q/p}+c\left((\tau\rr)^{m}\mint_{B_{2\tau\rr}}\snr{f}^{m} \ \dx\right)^{\frac{p}{m(p-1)}}\nonumber \\
&&+c\snr{\bar{z}+\sigma\snr{\bar{z}}Dh(x_{0})}^{2-p}\left((\tau\rr)^{m}\mint_{B_{2\tau\rr}}\snr{f}^{m} \ \dx\right)^{2/m}\nonumber \\
&=&c\left(\mbox{(I)}+\mbox{(II)}+\mbox{(III)}\right),
\end{eqnarray*}
with $c\equiv c(n,N,p,q,M)$. We continue estimating:
\begin{eqnarray*}
\mbox{(I)}&\stackrel{\eqref{ns.11}}{\le}&c\mf{K}\left(\tau^{2}\mf{F}(u;B_{\rr})^{2}\right)\stackrel{\eqref{ik}}{=}c\tau^{2}\mf{F}(u;B_{\rr})^{2}+c\tau^{2q/p}\mf{F}(u;B_{\rr})^{2q/p}\nonumber \\
&\stackrel{\eqref{ns.1}}{\le}&c\tau^{2}\mf{F}(u;B_{\rr})^{2}+cM^{2(q-p)/p}\tau^{2q/p}\left(\frac{\mf{F}(u;B_{\rr})}{\snr{(V_{p}(Du))_{\rr}}}\right)^{2(q-p)/p}\mf{F}(u;B_{\rr})^{2}\nonumber \\
&\stackrel{\eqref{ns.0}}{\le}&c\tau^{2}\mf{F}(u;B_{\rr})^{2}\left(1+M^{2(q-p)/p}\tau^{2(q-p)/p}\varepsilon_{0}^{2(q-p)/p}\right)\le c\tau^{2}\mf{F}(u;B_{\rr})^{2},
\end{eqnarray*}
for $c\equiv c(\textnormal{\texttt{data}},M)$,
\begin{eqnarray*}
\mbox{(II)}&\stackrel{\eqref{ns.10}}{\le}&c\mathds{1}_{\{q>p\}}\tau^{-nq/p}\mf{F}(u;B_{\rr})^{2q/p}\stackrel{\eqref{ns.1}}{\le}\frac{c\mathds{1}_{\{q>p\}}M^{2(q-p)/p}}{\tau^{nq/p}}\left(\frac{\mf{F}(u;B_{\rr})}{\snr{(V_{p}(Du))_{\rr}}}\right)^{2(q-p)/p}\mf{F}(u;B_{\rr})^{2}\nonumber \\
&\stackrel{\eqref{ns.0}}{\le}&c\mathds{1}_{\{q>p\}}\tau^{-nq/p}\varepsilon_{0}^{2(q-p)/p}\mf{F}(u;B_{\rr})^{2},
\end{eqnarray*}
with $c\equiv c(n,N,p,q,M)$ and
\begin{eqnarray*}
\mbox{(III)}&\stackrel{\eqref{ns.9}}{\le}&c\tau^{\frac{(m-n)p}{m(p-1)}}\left(\rr^{m}\mint_{B_{\rr}}\snr{f}^{m} \ \dx\right)^{\frac{p}{m(p-1)}}+c\tau^{\frac{2(m-n)}{m}}\snr{\bar{z}}^{2-p}\left(\rr^{m}\mint_{B_{\rr}}\snr{f}^{m} \ \dx\right)^{2/m}\nonumber \\
&\stackrel{\eqref{ns.2}}{\le}&c\varepsilon_{1}^{p/(p-1)}\tau^{\frac{(m-n)p}{m(p-1)}}\mf{F}(u;B_{\rr})^{p/(p-1)}\snr{(V_{p}(Du))_{\rr}}^{\frac{p-2}{p-1}}\nonumber \\
&&+c\varepsilon_{1}^{2}\tau^{\frac{2(m-n)}{m}}\snr{\bar{z}}^{2-p}\mf{F}(u;B_{\rr})^{2}\snr{(V_{p}(Du))_{\rr}}^{2(p-2)/p}\nonumber \\
&\stackrel{\eqref{ns.4}_{3}}{\le}&c\varepsilon_{1}^{p/(p-1)}\tau^{\frac{(m-n)p}{m(p-1)}}\left(\frac{\mf{F}(u;B_{\rr})}{\snr{(V_{p}(Du))_{\rr}}}\right)^{\frac{2-p}{p-1}}\mf{F}(u;B_{\rr})^{2}+c\varepsilon_{1}^{2}\tau^{\frac{2(m-n)}{m}}\mf{F}(u;B_{\rr})^{2}\nonumber \\
&\stackrel{\eqref{ns.0}}{\le}&c\left(\varepsilon_{1}^{p/(p-1)}\tau^{\frac{(m-n)p}{m(p-1)}}\varepsilon_{0}^{\frac{2-p}{p-1}}+\varepsilon_{1}^{2}\tau^{\frac{2(m-n)}{m}}\right)\mf{F}(u;B_{\rr})^{2},
\end{eqnarray*}
for $c\equiv c(n,p,m)$. Merging the previous four displays we obtain
\begin{eqnarray}\label{ns.13}
\ti{\mf{F}}(u,\bar{z}+\sigma\snr{\bar{z}}Dh(x_{0});B_{\tau\rr})^{2}&\le&c\mathcal{T}\mf{F}(u;B_{\rr})^{2},
\end{eqnarray}
where $c\equiv c(\textnormal{\texttt{data}},M)$ and we set
$$
\mathcal{T}:=\tau^{2}+\mathds{1}_{\{q>p\}}\tau^{-nq/p}\varepsilon_{0}^{2(q-p)/p}+\varepsilon_{1}^{p/(p-1)}\varepsilon_{0}^{(2-p)/(p-1)}\tau^{\frac{p(m-n)}{m(p-1)}}+\varepsilon_{1}^{2}\tau^{2(m-n)/m}. 
$$
Finally, let us observe that by triangular inequality it is
\begin{eqnarray}\label{ns.12}
\snr{V_{p}(Du)-V_{p}(\bar{z}+\sigma\snr{\bar{z}}Dh(x_{0})) }^{2}&\stackrel{\eqref{Vm}_{2}}{\le}& c\left(\snr{Du}^{2}+\snr{\bar{z}+\sigma\snr{\bar{z}}Dh(x_{0})}^{2}\right)^{(p-2)/2}\snr{Du-(\bar{z}+\sigma\snr{\bar{z}}Dh(x_{0}))}^{2}\nonumber \\
&\le&c\snr{V_{\snr{\bar{z}+\sigma\snr{\bar{z}}Dh(x_{0})},p}(Du-\bar{z}-\sigma\snr{\bar{z}}Dh(x_{0}))}^{2},
\end{eqnarray}
for $c\equiv c(n,N,p)$, therefore
\begin{flalign}
\mf{F}(u;B_{\tau\rr})^{2}\stackrel{\eqref{minav}}{\le}c\mf{F}(u,V_{p}(\bar{z}+\sigma\snr{\bar{z}}Dh(x_{0}));B_{\tau\rr})^{2}\stackrel{\eqref{ns.12}}{\le}c\ti{\mf{F}}(u,\bar{z}+\sigma\snr{\bar{z}}Dh(x_{0});B_{\tau\rr})^{2}\stackrel{\eqref{ns.13}}{\le}c\mathcal{T}\mf{F}(u;B_{\rr})^{2},\label{ns.14}
\end{flalign}
with $c\equiv c(\textnormal{\texttt{data}},M)$. Looking at the explicit expression of $\mathcal{T}$, we let $\beta\in (0,1)$ be any number, first fix $\tau\in (0,2^{-10})$, then reduce further the size of $\varepsilon_{0}\in (0,1)$ and finally restrict $\varepsilon_{1}$ in such a way that
\begin{flalign}\label{ns.20}
\left\{
\begin{array}{c}
\displaystyle 
\ c\max\left\{\tau^{2(1-\beta)},\tau^{\alpha/4}\right\}\le \frac{1}{2^{10}} \\[8pt]\displaystyle
\ \frac{c2^{2}\varepsilon_{0}}{\tau^{n}}+c\mathds{1}_{\{q>p\}}\tau^{-nq/p}\varepsilon_{0}^{2(q-p)/p}<\frac{\tau^{2\beta}}{2^{10}}\\[8pt]\displaystyle
\ c\max\left\{\varepsilon_{1}^{2}\tau^{\frac{2(m-n)}{m}},\varepsilon_{1}^{p/(p-1)}\tau^{\frac{p(m-n)}{m(p-1)}}\right\}<\frac{\tau^{2\beta_{0}}}{2^{10}},
\end{array}
\right.
\end{flalign}
where $\alpha\equiv \alpha(n,N,p)$ is the same exponent appearing in \eqref{y.1}$_{2}$. This way, we determine dependencies: $\tau,\varepsilon_{0},\varepsilon_{1}\equiv \tau,\varepsilon_{0},\varepsilon_{1}(\textnormal{\texttt{data}},M,\beta)$. Plugging the above restrictions in \eqref{ns.14} we get \eqref{nsns} and the proof is complete.
\end{proof}
Let us look at what happens when the complementary condition to \eqref{ns.2} is in force.
\begin{proposition}\label{p2}
In the setting of Proposition \ref{p1}, assume
\begin{eqnarray}\label{ns.21}
\varepsilon_{1}\mf{F}(u;B_{\rr}(x_{0}))\snr{(V_{p}(Du))_{B_{\rr}(x_{0})}}^{(p-2)/p}\le \left(\rr^{m}\mint_{B_{\rr}(x_{0})}\snr{f}^{m} \ \dx\right)^{1/m}
\end{eqnarray}
instead of \eqref{ns.2}. Then, for all $\tau\in (0,1)$ it is
\begin{eqnarray}\label{ns.22}
\mf{F}(u;B_{\tau\rr}(x_{0}))\le c_{0}\snr{(V_{p}(Du))_{B_{\rr}(x_{0})}}^{(2-p)/p}\left(\rr^{m}\mint_{B_{\rr}(x_{0})}\snr{f}^{m} \ \dx\right)^{1/m},
\end{eqnarray}
where $c_{0}:=2\varepsilon_{1}^{-1}\tau^{-n/2}$ and $\varepsilon_{1}\equiv \varepsilon_{1}(\textnormal{\texttt{data}},M)\in (0,1]$ is the same determined in \eqref{ns.20}.
\end{proposition}
\begin{proof}
Inequality \eqref{ns.22} is a direct consequence of \eqref{tri.1}$_{1}$ and \eqref{ns.21}.
\end{proof}
\section{The singular regime}
We start by proving that local minimizers of \eqref{exfun} are approximately $p$-harmonic within the singular scenario.
\begin{lemma}
Under assumptions \eqref{assf}-\eqref{sqc}, \eqref{assf.0} and \eqref{f} let $B_{\rr}(x_{0})\Subset \Omega$ be any ball and $u\in W^{1,p}(\Omega,\mathbb{R}^{N})$ be a local minimizer of \eqref{exfun}. Then
\begin{eqnarray}\label{pphar}
\left| \ \mint_{B_{\rr}(x_{0})}\snr{Du}^{p-2}\langle Du,D\varphi\rangle \ \dx \ \right|&\le& 4\nr{D\varphi}_{\infty}\left(\rr^{m}\mint_{B_{\rr}(x_{0})}\snr{f}^{m} \ \dx\right)^{1/m}+s\nr{D\varphi}_{\infty}\mf{J}_{p}(Du;B_{\rr})^{p-1}\nonumber \\
&&+c\nr{D\varphi}_{\infty}\left(\omega(s)^{-1}+\omega(s)^{-p}+\omega(s)^{q-p-1}\right)\mf{J}_{p}(Du;B_{\rr})^{p},
\end{eqnarray}
for all $\varphi\in C^{\infty}_{c}(B_{\rr}(x_{0}),\mathbb{R}^{N})$ and any $s\in (0,\infty)$, with $c\equiv c(n,N,\Lambda,q)$.
\end{lemma}
\begin{proof}
With the same abbreviations used in Lemma \ref{aahar}, let $\varphi\in C^{\infty}_{c}(B_{\rr},\mathbb{R}^{N})$ be any smooth map. We use \eqref{el} to control
\begin{eqnarray*}
\left| \ \mint_{B_{\rr}}\snr{Du}^{p-2}\langle Du,D\varphi\rangle \ \dx \ \right|&=&\left| \ \mint_{B_{\rr}}\langle \partial F(Du)-\partial F(0)-\snr{Du}^{p-2}Du,D\varphi\rangle-f\cdot \varphi \ \dx \ \right|\nonumber \\
&\le&\left| \ \mint_{B_{\rr}}\langle \partial F(Du)-\partial F(0)-\snr{Du}^{p-2}Du,D\varphi\rangle \ \dx \ \right|\nonumber \\
&&+\mint_{B_{\rr}}\snr{f}\snr{\varphi} \ \dx =:\mathcal{I}_{1}+\mathcal{I}_{2}.
\end{eqnarray*}
We fix $s\in (0,\infty)$, notice that
\eqn{notice}
$$
\frac{\snr{B_{\rr}\cap\{\snr{Du}>\omega(s)\}}}{\snr{B_{\rr}}}
\le \left(\frac{\mf{J}_{p}(Du;B_{\rr})}{\omega(s)}\right)^{p}$$
and then bound
\begin{eqnarray*}
\mathcal{I}_{1}&\le&\frac{1}{\snr{B_{\rr}}}\int_{B_{\rr}\cap \{\snr{Du}<\omega(s)\}}\snr{\langle \partial F(Du)-\partial F(0)-\snr{Du}^{p-2}Du,D\varphi\rangle} \ \dx\nonumber \\
&&+\frac{1}{\snr{B_{\rr}}}\int_{B_{\rr}\cap \{\snr{Du}\ge \omega(s)\}}\snr{\langle \partial F(Du)-\partial F(0)-\snr{Du}^{p-2}Du,D\varphi\rangle} \ \dx\nonumber \\
&\stackrel{\eqref{assf.0.2},\eqref{assf.0.0}}{\le}&\nr{D\varphi}_{\infty}\left(s\mf{J}_{p}(Du;B_{\rr})^{p-1}+\frac{c}{\snr{B_{\rr}}}\int_{B_{\rr}\cap\{\snr{Du}\ge \omega(s)\}}1+\snr{Du}^{p-1}+\snr{Du}^{q-1} \ \dx\right)\nonumber \\
&\stackrel{\eqref{pq}_{2}}{\le}&s\nr{D\varphi}_{\infty}\mf{J}_{p}(Du;B_{\rr})^{p-1}+c\nr{D\varphi}_{\infty}\frac{\snr{B_{\rr}\cap\{\snr{Du}>\omega(s)\}}}{\snr{B_{\rr}}}\nonumber \\
&&+c\nr{D\varphi}_{\infty}\left(\frac{\snr{B_{\rr}\cap\{\snr{Du}>\omega(s)\}}}{\snr{B_{\rr}}}\right)^{1/p}\mf{J}_{p}(Du;B_{\rr})^{p-1}\nonumber \\
&&+c\nr{D\varphi}_{\infty}\left(\frac{\snr{B_{\rr}\cap\{\snr{Du}>\omega(s)\}}}{\snr{B_{\rr}}}\right)^{(p-q+1)/p}\mf{J}_{p}(Du;B_{\rr})^{q-1}\nonumber \\
&\stackrel{\eqref{notice}}{\le}&s\nr{D\varphi}_{\infty}\mf{J}_{p}(Du;B_{\rr})^{p-1}+c\nr{D\varphi}_{\infty}\left(\omega(s)^{-1}+\omega(s)^{-p}+\omega(s)^{q-p-1}\right)\mf{J}_{p}(Du;B_{\rr})^{p},
\end{eqnarray*}
for $c\equiv c(n,N,\Lambda,q)$ and
\begin{eqnarray*}
\mathcal{I}_{2}\le 4\nr{D\varphi}_{\infty}\left(\rr^{m}\mint_{B_{\rr}}\snr{f}^{m} \ \dx\right)^{1/m}.
\end{eqnarray*}
Merging the content of the two previous display we obtain \eqref{pphar} and the proof is complete.
\end{proof}
Next, a one-scale decay estimate for the excess functional valid in the singular regime.
\begin{proposition}\label{p3}
Under hypotheses \eqref{assf}-\eqref{sqc}, \eqref{assf.0} and \eqref{f}, let $u\in W^{1,p}(\Omega,\mathbb{R}^{N})$ be a local minimizer of \eqref{exfun} satisfying \eqref{ns.1} on a ball $B_{\rr}(x_{0})\Subset \Omega$ for some positive constant $M$. Then, for any $\gamma\in (0,\alpha)$, $\chi\in (0,1]$ there are $\theta\equiv \theta(\textnormal{\texttt{data}},\chi,\gamma,M)\in (0,2^{-10})$, $\varepsilon,\varepsilon_{i}\equiv\varepsilon,\varepsilon_{i}(\textnormal{\texttt{data}},\gamma,M)$, $i\in \{2,3\}$, such that if the smallness conditions
\begin{flalign}\label{s.0}
\chi\snr{(V_{p}(Du))_{B_{\rr}(x_{0})}}\le \mf{F}(u;B_{\rr}(x_{0})),\qquad  \mf{F}(u;B_{\rr}(x_{0}))<\varepsilon_{2},\qquad \left(\rr^{m}\mint_{B_{\rr}(x_{0})}\snr{f}^{m} \ \dx\right)^{1/m}<\varepsilon_{3}
\end{flalign}
hold on $B_{\rr}(x_{0})$, then
\begin{eqnarray}\label{s.00}
\mf{F}(u;B_{\theta\rr}(x_{0}))\le \theta^{\gamma}\mf{F}(u;B_{\rr}(x_{0}))+c_{1}\mf{K}\left[\left(\rr^{m}\mint_{B_{\rr}(x_{0})}\snr{f}^{m} \ \dx\right)^{1/m}\right]^{\frac{p}{2(p-1)}},
\end{eqnarray}
with $c_{1}\equiv c_{1}(\textnormal{\texttt{data}},\chi,\gamma,M)$. Here, $\alpha\equiv \alpha(n,N,p)$ is the exponent appearing in $\eqref{y.1}_{2}$ and $\mf{K}(\cdot)$ has been defined in \eqref{ik}.
\end{proposition}
\begin{proof}
Let us premise that the same abbreviations appearing in Proposition \ref{p1} will be adopted also here. By triangular inequality and \eqref{s.0}$_{1}$ we have that
\begin{eqnarray}\label{s.1}
\mf{J}_{p}(Du;B_{\rr})^{p}\le 2\mf{F}(u;B_{\rr})^{2}+2\snr{(V_{p}(Du))_{\rr}}^{2}\le 2\left(1+\frac{1}{\chi^{2}}\right)\mf{F}(u;B_{\rr})^{2}=:c_{\chi}\mf{F}(u;B_{\rr})^{2}.
\end{eqnarray}
With this estimate at hand, \eqref{pphar} becomes
\begin{eqnarray}\label{s.2}
\left| \ \mint_{B_{\rr}}\langle \snr{Du}^{p-2}Du,D\varphi\rangle \ \dx \ \right|&\le&4\nr{D\varphi}_{\infty}\left(\rr^{m}\mint_{B_{\rr}}\snr{f}^{m} \ \dx\right)^{1/m}+sc_{\chi}^{(p-1)/p}\nr{D\varphi}_{\infty}\mf{F}(u;B_{\rr})^{2(p-1)/p}\nonumber \\
&&+cc_{\chi}\nr{D\varphi}_{\infty}\left(\omega(s)^{-1}+\omega(s)^{-p}+\omega(s)^{q-p-1}\right)\mf{F}(u;B_{\rr})^{2},
\end{eqnarray}
for all $\varphi\in C^{\infty}_{c}(B_{\rr},\mathbb{R}^{N})$, $s\in (0,\infty)$, with $c\equiv c(n,N,\Lambda,q)$. Notice that there is no loss of generality in assuming $\mf{F}(u;B_{\rr})>0$, otherwise \eqref{s.00} would trivially be true by means of \eqref{tri.1}$_{1}$. We then set
\begin{eqnarray*}
\psi:=c_{\chi}^{1/p}\mf{F}(u;B_{\rr})^{2/p}+\left(\frac{4}{\varepsilon_{4}}\right)^{1/(p-1)}\left(\rr^{m}\mint_{B_{\rr}}\snr{f}^{m} \ \dx\right)^{\frac{1}{m(p-1)}},\qquad\qquad u_{0}:=\frac{u}{\psi}
\end{eqnarray*}
and divide both sides of \eqref{s.2} by $\psi^{p-1}$ to get
\begin{eqnarray}\label{s.3}
\left| \ \mint_{B_{\rr}}\langle \snr{Du_{0}}^{p-2}Du_{0},D\varphi\rangle  \ \dx\ \right|&\le&\left(\varepsilon_{3}+s\right)\nr{D\varphi}_{\infty}\nonumber \\
&&+cc_{\chi}^{1/p}\nr{D\varphi}_{\infty}\left(\omega(s)^{-1}+\omega(s)^{-p}+\omega(s)^{q-p-1}\right)\mf{F}(u;B_{\rr})^{2/p}\nonumber \\
&\stackrel{\eqref{s.0}_{2}}{\le}&\left[\varepsilon_{3}+s+cc_{\chi}^{1/p}\left(\omega(s)^{-1}+\omega(s)^{-p}+\omega(s)^{q-p-1}\right)\varepsilon_{2}^{2/p}\right]\nr{D\varphi}_{\infty},
\end{eqnarray}
for $c\equiv c(n,N,\Lambda,q)$. Now as a direct consequence of \eqref{s.1} we obtain
\begin{eqnarray}\label{s.4}
\mf{J}_{p}(Du_{0};B_{\rr})\le c_{\chi}^{1/p}\psi^{-1}\mf{F}(u;B_{\rr})^{2/p}\le 1,
\end{eqnarray}
so fixed $\varepsilon\in (0,1)$ to be determined later on and, with $\delta\equiv \delta(n,N,p,\varepsilon)\in (0,1]$ being the small parameter given by Lemma \ref{p.phar}, we reduce the size of parameters $\varepsilon_{4},s$ first and then that of $\varepsilon_{2},\varepsilon_{3}$ in such a way that
\begin{flalign}\label{s.5}
\varepsilon_{4}+s<\frac{\delta}{4}, \qquad  cc_{\chi}^{1/p}\left(\omega(s)^{-1}+\omega(s)^{-p}+\omega(s)^{q-p-1}+c'\right)\varepsilon_{2}^{2/p}<\frac{\delta}{4},\qquad \max\{c',1\}\left(\frac{4\varepsilon_{3}}{\varepsilon_{4}}\right)^{1/(p-1)}<\frac{1}{4}
\end{flalign}
thus \eqref{s.3} is turned into \eqref{p.phar1} and together with \eqref{s.4} legalizes the application of Lemma \ref{p.phar} which renders a $p$-harmonic map $h\in W^{1,p}(B_{\rr},\mathbb{R}^{N})$ such that
\begin{flalign}\label{s.6}
\mf{J}_{p}(Dh;B_{\rr})\le 1\qquad \mbox{and}\qquad \mint_{B_{\rr}}\left| \ \frac{u_{0}-h}{\rr} \ \right|^{p}\le c\varepsilon^{p},
\end{flalign}
for $c\equiv c(n,N,p)$. In $\eqref{s.5}_{3}$, $c'$ is the constant appearing in the Lipschitz bound $\eqref{y.1}_{1}$. We point out that the choices in \eqref{s.5} fix dependencies $\varepsilon_{3},\varepsilon_{4},s\equiv \varepsilon_{3},\varepsilon_{4},s(n,N,p,\varepsilon)$ and $\varepsilon_{2}\equiv \varepsilon_{2}(n,N,p,\omega(\cdot),\varepsilon,\chi)$. Further restrictions on the size of these parameters will be imposed in a few lines. Next, for $\theta\in (0,2^{-10})$, we exploit the isomorphism properties of $V_{p}(\cdot)$ to determine $z_{2\theta\rr}\in \mathbb{R}^{N\times n}$ such that $V_{p}(z_{2\theta\rr})=(V_{p}(Dh))_{2\theta \rr}$ and estimate via $\eqref{Vm}_{3,7}$, \eqref{v.poi}, $\eqref{s.6}$ and $\eqref{y.1}_{2}$,
\begin{flalign*}
&\mint_{B_{2\theta\rr}}\left| \ V_{\snr{z_{2\theta\rr}},p}\left(\frac{u_{0}-(h)_{2\theta\rr}-\langle z_{2\theta\rr},x-x_{0}\rangle}{2\theta\rr}\right) \ \right|^{2} \ \dx\le c\mint_{B_{2\theta\rr}}\left| \ V_{\snr{z_{2\theta\rr}},p}\left(\frac{u_{0}-h}{2\theta\rr}\right) \ \right|^{2} \ \dx\nonumber \\
&\qquad \qquad\qquad \quad+c\mint_{B_{2\theta\rr}}\left| \ V_{\snr{z_{2\theta\rr}},p}\left(\frac{h-(h)_{2\theta\rr}-\langle z_{2\theta\rr},x-x_{0}\rangle}{2\theta\rr}\right) \ \right|^{2} \ \dx\nonumber \\
&\qquad \qquad\qquad  \le c\theta^{-n-p}\int_{B_{\rr}}\left| \ \frac{u_{0}-h}{\rr} \ \right|^{p} \ \dx+c\mint_{B_{2\theta\rr}}\snr{V_{\snr{z_{2\theta\rr}},p}(Dh-z_{2\theta\rr})}^{2} \ \dx\nonumber \\
&\qquad \qquad \qquad \le c\theta^{-n-p}\varepsilon^{p}+c\mint_{B_{2\theta\rr}}\snr{V_{p}(Dh)-(V_{p}(Dh))_{2\theta\rr}}^{2} \ \dx\le c\theta^{-n-p}\varepsilon^{p}+c\theta^{2\alpha},
\end{flalign*}
with $c\equiv c(n,N,p)$. Scaling back to $u$ in the previous display we obtain
\begin{eqnarray}\label{s.7}
\mint_{B_{2\theta\rr}}\left| \ V_{\psi\snr{z_{2\theta\rr}},p}\left(\frac{u-\psi((h)_{2\theta\rr}+\langle z_{2\theta\rr},x-x_{0}\rangle)}{2\theta\rr}\right) \ \right|^{2} \ \dx\le c\psi^{p}\left(\theta^{-n-p}\varepsilon^{p}+\theta^{2\alpha}\right),
\end{eqnarray}
for $c\equiv c(n,N,p)$. Notice that
\begin{flalign}
\left\{
\begin{array}{c}
\displaystyle 
\ \snr{z_{2\theta\rr}}=\snr{(V_{p}(Du))_{\rr}}^{2/p}\le \mf{J}_{p}(Dh;B_{2\theta\rr})\stackrel{\eqref{y.1}_{1}}{\le}c'\mf{J}_{p}(Dh;B_{\rr})\stackrel{\eqref{s.6}_{1}}{\le}c' \\[8pt]\displaystyle
\ \psi\stackrel{\eqref{s.0}_{2,3}}{<}c_{\chi}^{1/p}\varepsilon_{2}^{2/p}+\left(\frac{4\varepsilon_{3}}{\varepsilon_{4}}\right)^{1/(p-1)}\stackrel{\eqref{s.5}}{\le}1,
\end{array}
\right.\label{s.8}
\end{flalign}
so we can bound
\begin{eqnarray*}
\mf{F}(u;B_{\theta\rr})^{2}&\stackrel{\eqref{minav}}{\le}&4\mint_{B_{\theta\rr}}\snr{V_{p}(Du)-V_{p}(\psi z_{2\theta\rr})}^{2} \ \dx\stackrel{\eqref{Vm}_{3}}{\le}c\ti{\mf{F}}(u,\psi z_{2\theta\rr};B_{\theta\rr})^{2}\nonumber \\
&\stackrel{\eqref{s.8},\eqref{cacc.ns}}{\le}&c\mf{K}\left(\mint_{B_{2\theta\rr}}\left| \ V_{\psi\snr{z_{2\theta\rr}},p}\left(\frac{u-\psi((h)_{2\theta\rr}-\langle z_{2\theta\rr},x-x_{0}\rangle)}{2\theta\rr}\right) \ \right|^{2} \ \dx\right)\nonumber \\
&&+c\psi^{2-p}\snr{z_{2\theta\rr}}^{2-p}\left((\theta\rr)^{m}\mint_{B_{2\theta\rr}}\snr{f}^{m} \ \dx\right)^{2/m}\nonumber \\
&&+c\left((\theta\rr)^{m}\mint_{B_{2\theta\rr}}\snr{f}^{m} \ \dx\right)^{\frac{p}{m(p-1)}}+c\mathds{1}_{\{q>p\}}\ti{\mf{F}}(u,\psi z_{2\theta\rr};B_{2\theta\rr})^{2q/p}\nonumber \\
&\stackrel{\eqref{s.7}}{\le}&c\mf{K}\left(\psi^{p}\left(\theta^{-n-p}\varepsilon^{p}+\theta^{2\alpha}\right)\right)+c\psi^{2-p}\snr{z_{2\theta\rr}}^{2-p}\left((\theta\rr)^{m}\mint_{B_{2\theta\rr}}\snr{f}^{m} \ \dx\right)^{2/m}\nonumber \\
&&+c\left((\theta\rr)^{m}\mint_{B_{2\theta\rr}}\snr{f}^{m} \ \dx\right)^{\frac{p}{m(p-1)}}+c\mathds{1}_{\{q>p\}}\ti{\mf{F}}(u,\psi z_{2\theta\rr};B_{2\theta\rr})^{2q/p}\nonumber \\
&=:&\mbox{(I)}+\mbox{(II)}+\mbox{(III)}+\mbox{(IV)},
\end{eqnarray*}
with $c\equiv c(\textnormal{\texttt{data}},M)$. We continue estimating
\begin{eqnarray*}
\mbox{(I)}&\stackrel{\eqref{ik}_{2}}{\le}&c\psi^{p}\left(\theta^{-n-p}\varepsilon^{p}+\theta^{2\alpha}\right)+c\psi^{q}\left(\theta^{-n-p}\varepsilon^{p}+\theta^{2\alpha}\right)^{q/p}\nonumber \\
&\stackrel{\eqref{s.0}_{2}}{\le}&cc_{\chi}^{q/p}\left(1+\varepsilon_{2}^{2(q-p)/p}\right)\left(\theta^{-\frac{(n+p)q}{p}}\varepsilon^{p}+\theta^{2\alpha}\right)\mf{F}(u;B_{\rr})^{2}\nonumber \\
&&+\frac{c}{\varepsilon_{4}^{q/(p-1)}}\mf{K}\left[\left(\rr^{m}\mint_{B_{\rr}}\snr{f}^{m} \ \dx\right)^{1/m}\right]^{p/(p-1)},
\end{eqnarray*}
for $c\equiv c(\textnormal{\texttt{data}},M)$. Moreover, by Young inequality with conjugate exponents $\left(\frac{p}{2-p},\frac{p}{2(p-1)}\right)$ we have
\begin{eqnarray*}
\mbox{(II)}+\mbox{(III)}&\stackrel{\eqref{s.8}_{1}}{\le}&cc_{\chi}^{(2-p)/p}\mf{F}(u;B_{\rr})^{\frac{2(2-p)}{p}}\theta^{\frac{2(m-n)}{m}}\left(\rr^{m}\mint_{B_{\rr}}\snr{f}^{m} \ \dx\right)^{2/m}\nonumber \\
&&+c\left(\theta^{\frac{2(m-n)}{m}}\varepsilon_{4}^{\frac{p-2}{p-1}}+\theta^{\frac{(m-n)p}{m(p-1)}}\right)\left(\rr^{m}\mint_{B_{\rr}}\snr{f}^{m} \ \dx\right)^{\frac{p}{m(p-1)}}\nonumber \\
&\le&cc_{\chi}\varepsilon\mf{F}(u;B_{\rr})^{2}+c\left(\varepsilon^{\frac{p-2}{2(p-1)}}\theta^{\frac{p(m-n)}{m(p-1)}}+\theta^{\frac{2(m-n)}{m}}\varepsilon_{4}^{\frac{p-2}{p-1}}\right)\left(\rr^{m}\mint_{B_{\rr}}\snr{f}^{m} \ \dx\right)^{\frac{p}{m(p-1)}},
\end{eqnarray*}
with $c\equiv c(\textnormal{\texttt{data}},M)$. Finally, we control
\begin{eqnarray*}
\mbox{(IV)}&\le&c\psi^{q}\mathds{1}_{\{q>p\}}\ti{\mf{F}}(u_{0},z_{2\theta\rr};B_{2\theta\rr})^{2q/p}\nonumber \\
&\stackrel{\eqref{pq}_{1}}{\le}& c\psi^{q}\theta^{-nq/p}\mathds{1}_{\{q>p\}}\mf{J}_{p}(Du_{0}-z_{2\theta\rr};B_{\rr})^{q}\stackrel{\eqref{s.4},\eqref{s.8}_{1}}{\le}c\mathds{1}_{\{q>p\}}\psi^{q}\theta^{-nq/p}\nonumber \\
&\stackrel{\eqref{s.0}_{2}}{\le}&cc_{\chi}^{q/p}\mathds{1}_{\{q>p\}}\theta^{-nq/p}\varepsilon_{2}^{2(q-p)/p}\mf{F}(u;B_{\rr})^{2}+\frac{c\mathds{1}_{\{q>p\}}}{\varepsilon_{4}^{q/(p-1)}\theta^{nq/p}}\left(\rr^{m}\mint_{B_{\rr}}\snr{f}^{m} \ \dx\right)^{\frac{q}{m(p-1)}},
\end{eqnarray*}
for $c\equiv c(\textnormal{\texttt{data}},M)$. Setting
\begin{flalign*}
&\mathcal{T}_{1}:=\theta^{-\frac{(n+p)q}{p}}\varepsilon^{p}+\theta^{2\alpha}+\varepsilon+\theta^{-nq/p}\mathds{1}_{\{q>p\}}\varepsilon_{2}^{2(q-p)/p};\nonumber \\
&\mathcal{T}_{2}:=\varepsilon_{4}^{-q/(p-1)}\theta^{-nq/p}+\varepsilon^{\frac{p-2}{2(p-1)}}\theta^{\frac{p(m-n)}{m(p-1)}}+\theta^{\frac{2(m-n)}{m}}\varepsilon_{4}^{\frac{p-2}{p-1}}
\end{flalign*}
and merging the content of all the previous displays we end up with
\begin{eqnarray}\label{s.9}
\mf{F}(u;B_{\theta\rr})^{2}\le cc_{\chi}^{q/p}\mathcal{T}_{1}\mf{F}(u;B_{\rr})^{2}+c\mathcal{T}_{2}\mf{K}\left[\left(\rr^{m}\mint_{B_{\rr}}\snr{f}^{m} \ \dx\right)^{1/m}\right]^{p/(p-1)},
\end{eqnarray}
for $c\equiv c(\textnormal{\texttt{data}},M)$. We then reduce the size of the various parameter appearing in the definition of $\mathcal{T}_{1}$ to get
\begin{flalign}\label{s.11}
\theta^{-\frac{(n+p)q}{p}}(\varepsilon^{p}+\varepsilon)<\frac{\theta^{2\alpha}}{2^{10}},\qquad \qquad  4\theta^{-(n+2)}\varepsilon_{2}+\theta^{-nq/p}\mathds{1}_{\{q>p\}}\varepsilon_{2}^{2(q-p)/p}<\frac{\theta^{2\alpha}}{2^{10}},
\end{flalign}
thus fixing dependencies $\varepsilon,\varepsilon_{2}\equiv \varepsilon,\varepsilon_{2}(n,N,p,q,m,\theta)$, and \eqref{s.9} becomes
\begin{eqnarray}\label{s.10}
\mf{F}(u;B_{\theta\rr})^{2}\le cc_{\chi}^{q/p}\theta^{2\alpha}\mf{F}(u;B_{\rr})^{2}+c\mf{K}\left[\left(\rr^{m}\mint_{B_{\rr}}\snr{f}^{m} \ \dx\right)^{1/m}\right]^{p/(p-1)},
\end{eqnarray}
with $c\equiv c(\textnormal{\texttt{data}},\theta,M)$. Finally, we pick any $\gamma\in (0,\alpha)$, with $\alpha\equiv \alpha(n,N,p)$ being the exponent in \eqref{y.1}$_{2}$ and select $\theta\in (0,2^{-10})$ so small that
\begin{eqnarray}\label{s.12}
cc_{\chi}^{q/p}\theta^{2(\alpha-\gamma)}\le \frac{1}{2^{10}} \ \Longrightarrow \ \theta\equiv \theta(\textnormal{\texttt{data}},\chi,\gamma,M),
\end{eqnarray}
so with this choice and \eqref{s.10} we obtain \eqref{s.00} and the proof is complete.
\end{proof}
\section{Excess decay and the Regular set}\label{morsec} In this section we prove that the excess functional $\mf{F}(\cdot)$ decays on a certain subset of $\Omega$ provided the boundedness of the potential $\mathbf{I}^{f}_{1,m}(\cdot)$. Precisely, with $u\in W^{1,p}(\Omega,\mathbb{R}^{N})$ being a local minimizer of \eqref{exfun}, we set
\begin{flalign}\label{ru.0}
&\mathcal{R}_{u}:=\left\{\frac{}{} x_{0}\in \Omega\colon \exists \ M\equiv M(x_{0})\in (0,\infty),\ \bar{\rr}\equiv \bar{\rr}(\textnormal{\texttt{data}},M,f(\cdot))<\min\{d_{x_{0}},1\}, \ \bar{\varepsilon}\equiv \bar{\varepsilon}(\textnormal{\texttt{data}},M)\frac{}{}\right.\nonumber \\
&\qquad \qquad \qquad \qquad\qquad\qquad \left. \frac{}{}\mbox{such that} \ \snr{(V_{p}(Du))_{B_{\rr}(x_{0})}}<M \ \mbox{and} \ \mf{F}(u;B_{\rr}(x_{0}))<\bar{\varepsilon} \ \mbox{for some} \ \rr\in (0,\bar{\rr}]\frac{}{}\right\}.
\end{flalign}
According to the discussion in \cite[Section 5.1]{deqc}, the set $\mathcal{R}_{u}$ is well defined and open, with full $n$-dimensional Lebesgue measure. In particular, given any point $x_{0}\in \mathcal{R}_{u}$, there exists an open neighborhood $B(x_{0})\subset \mathcal{R}_{u}$ of $x_{0}$ and a radius $\rr_{x_{0}}\in (0,\bar{\rr}]$ such that
\begin{flalign}\label{70}
\snr{(V_{p}(Du))_{B_{\rr_{x_{0}}}(x)}}<M\quad \mbox{and}\quad \mf{F}(u;B_{\rr_{x_{0}}}(x))<\bar{\varepsilon}\qquad \mbox{for all} \ \ x\in B(x_{0}).
\end{flalign}
We stress that for a given point $x_{0}\in \mathcal{R}_{u}$, all the radii considered from now on will be implicitly assumed to be less than $\min\{d_{x_{0}},1\}$. Next, for $x_{0}\in \mathcal{R}_{u}$ verifying conditions \eqref{ru.0} for some $M\equiv M(x_{0})>0$, and parameters $\bar{\varepsilon},\bar{\rr}$ still to be fixed, we set $\nu:=2^{-2}$, choose $\gamma=\alpha/2$ in \eqref{s.00}, $\beta=\gamma$ in \eqref{nsns}, define $\alpha_{0}:=\gamma$ and let $\chi:=\varepsilon_{0}$ in \eqref{s.0}$_{1}$. This eventually fixes the dependency of all the parameters appearing in Propositions \ref{p1}, \ref{p2} and \ref{p3} on $(\textnormal{\texttt{data}},M)$. We then define parameters
\begin{flalign}\label{6.9}
\hat{\varepsilon}:=\frac{\varepsilon_{2}\varepsilon_{0}^{2}\mf{m}(\tau\theta)^{32npq}}{2^{40npq+10}c_{3}c_{3}'H},\qquad \qquad  H:=2^{8(n+10)}c_{3}\max\{\tau^{-n},\varepsilon_{0}^{-1}\},
\end{flalign} 
constants $c_{2}:=4(c_{0}+c_{1})$, $c_{3}:=c_{2}\max_{\delta\in \{\nu,\tau,\theta\}}(1-\delta^{\alpha_{0}})^{-1}$, $c_{3}':=\left(\frac{c_{3}2^{28nq}H}{\varepsilon_{0}\mf{m}(\tau\theta)^{16nq}}\right)^{\frac{p}{2(p-1)}}$, $\mf{m}:=\min_{\delta\in \{\nu,\tau,\theta\}}(1-\delta^{\alpha_{0}})$,
introduce the balanced composite excess functional 
\eqn{cb}
$$
(0,\rr]\ni s\mapsto \mf{C}(x_{0};s):=\mf{F}(u;B_{s}(x_{0}))+\snr{(V_{p}(Du))_{B_{s}(x_{0})}}
$$
and assume that
\begin{eqnarray}\label{6.0}
\mathbf{I}^{f}_{1,m}(x_{0},1)<\infty.
\end{eqnarray}
Notice that, up to extend $f\equiv 0$ in $\mathbb{R}^{n}\setminus \Omega$, the above position makes sense. Moreover, in \eqref{6.0} the finiteness of $\mathbf{I}^{f}_{1,m}(x_{0},\cdot)$ is assumed to hold at radius one, but of course we can suppose that it holds at any positive radius. By \eqref{6.0} and the absolute continuity of Lebesgue integral, we can find $\hat{\rr}\equiv \hat{\rr}(\textnormal{\texttt{data}},f(\cdot),M)\in (0,\min\{1,d_{x_{0}}\})$ such that
\begin{flalign}\label{6.1}
c_{4}\mf{K}\left(\mathbf{I}^{f}_{1,m}(x_{0},s)\right)^{\frac{p}{2(p-1)}}+c_{4}M^{(2-p)/p}\mathbf{I}^{f}_{1,m}(x_{0},s)< \hat{\varepsilon},\qquad c_{4}:=\left(\frac{2^{80npq}c_{3}^{2}c_{3}'H^{2}}{\varepsilon_{0}^{4}\varepsilon_{3}(\tau\theta)^{64npq}\mf{m}}\right)^{\frac{p^{2}}{4(p-1)^{2}}}
\end{flalign}
for all $s\in (0,\hat{\rr}]$, so if $\delta\in \{\nu,\tau,\theta\}$ it is
\begin{eqnarray}\label{6.2}
\max_{\delta\in \{\nu,\tau, \theta\}}\left\{ \sum_{j=0}^{\infty}\left((\delta^{j+1}s)^{m}\mint_{B_{\delta^{j+1}s}(x_{0})}\snr{f}^{m} \dx\right)^{1/m} \right\}\le \frac{2^{4n}\mathbf{I}^{f}_{1,m}(x_{0},s)}{(\tau\theta)^{2n}},
\end{eqnarray}
for all $s\in (0,\hat{\rr}]$. Recalling that $\nu>\max\{\tau,\theta\}$, by \eqref{6.2} and routine interpolation arguments we obtain that
\begin{eqnarray}\label{6.3.1}
\left(\sigma^{m}\mint_{B_{\sigma}(x_{0})}\snr{f}^{m} \dx\right)^{1/m}\le \frac{2^{8n}\mathbf{I}^{f}_{1,m}(x_{0},s)}{(\tau\theta)^{4n}}\qquad \mbox{for all} \ \ 0<\sigma\le s/4 ,
\end{eqnarray}
which, together with \eqref{6.1} yields:
\begin{flalign}\label{6.3}
&\sup_{\sigma\le s/4}\mf{K}\left[\left(\sigma^{m}\mint_{B_{\sigma}(x_{0})}\snr{f}^{m} \dx\right)^{1/m}\right]^{\frac{p}{2(p-1)}}\nonumber \\
&\qquad \qquad \qquad \qquad +M^{(2-p)/p}\sup_{\sigma\le s/4}\left(\sigma^{m}\mint_{B_{\sigma}(x_{0})}\snr{f}^{m} \dx\right)^{1/m}<\hat{\varepsilon}\left(\frac{\varepsilon_{0}^{4}\varepsilon_{3}\mf{m}(\tau\theta)^{56npq}}{2^{64npq}c_{3}^{2}c_{3}'H^{2}}\right)^{\frac{p^{2}}{4(p-1)^{2}}}
\end{flalign}
for all $s\in (0,\hat{\rr}]$ and
\begin{eqnarray}\label{6.4}
\lim_{\sigma\to 0}\left(\sigma^{m}\mint_{B_{\sigma}(x_{0})}\snr{f}^{m} \dx\right)^{1/m}=0.
\end{eqnarray}
We refer to \cite[Section 5.2]{deqc} for more details on this matter. In \eqref{ru.0}, we pick $\bar{\varepsilon}= \hat{\varepsilon}$, $\bar{\rr}= \hat{\rr}$, thus determining a ball $B_{\rr}(x_{0})\Subset \Omega$ with $\rr\in (0,\hat{\rr}]$ on which
\begin{eqnarray}\label{6.5}
\snr{(V_{p}(Du))_{B_{\rr}(x_{0})}}<M\qquad \mbox{and}\qquad \mf{F}(u;B_{\rr}(x_{0}))<\hat{\varepsilon}
\end{eqnarray}
hold true. Now we are ready to prove the main result of this section.
\begin{theorem}\label{t.ex}
Under assumptions \eqref{assf}-\eqref{sqc}, \eqref{assf.0}, \eqref{f} and \eqref{6.0}, let $u\in W^{1,p}(\Omega,\mathbb{R}^{N})$ be a local minimizer of \eqref{exfun}, $x_{0}\in \mathcal{R}_{u}$ be a point and $M\equiv M(x_{0})>0$ be the corresponding constant in \eqref{ru.0}. There are $\hat{\varepsilon}\equiv \hat{\varepsilon}(\textnormal{\texttt{data}},M)\in (0,1)$ and $\hat{\rr}\equiv \hat{\rr}(\textnormal{\texttt{data}},M,f(\cdot))<d_{x_{0}}$ as in \eqref{6.9}$_{1}$ and \eqref{6.1} respectively, such that if $\bar{\varepsilon}\equiv \hat{\varepsilon}$ and $\bar{\rr}\equiv \hat{\rr}$ in \eqref{ru.0}, then for all balls $B_{\varsigma}(x_{0})\subset B_{\rr}(x_{0})$ it holds 
\begin{flalign}\label{6.6}
\left\{
\begin{array}{c}
\displaystyle
\ \snr{(V_{p}(Du))_{B_{\varsigma}(x_{0})}}<8(1+M)\\[8pt]\displaystyle 
\ \snr{(V_{p}(Du))_{B_{\varsigma}(x_{0})}}\le c_{6}\left(\mf{C}(x_{0};\rr)+\mf{K}\left(\mathbf{I}^{f}_{1,m}(x_{0},\rr)\right)^{\frac{p}{2(p-1)}}\right),
\end{array}
\right.
\end{flalign}
and
\begin{eqnarray}\label{6.7}
\mf{F}(u;B_{\varsigma}(x_{0}))&\le& c_{5}\left(\frac{\varsigma}{\rr}\right)^{\alpha_{0}}\mf{F}(u;B_{\rr}(x_{0}))+c_{6}\sup_{\sigma\le \rr/4}\mf{K}\left[\left(\sigma^{m}\mint_{B_{\sigma}(x_{0})}\snr{f}^{m} \ \dx\right)^{1/m}\right]^{\frac{p}{2(p-1)}}\nonumber \\
&&+c_{6}\left(\mf{C}(x_{0};\rr)+\mf{K}\left(\mathbf{I}^{f}_{1,m}(x_{0},\rr)\right)^{\frac{p}{2(p-1)}}\right)^{\frac{2-p}{p}}\sup_{\sigma\le \rr/4}\left(\sigma^{m}\mint_{B_{\sigma}(x_{0})}\snr{f}^{m} \ \dx\right)^{1/m},
\end{eqnarray}
with $c_{5},c_{6}\equiv c_{5},c_{6}(\textnormal{\texttt{data}},M)$ and $\alpha_{0}\equiv \alpha_{0}(n,N,p)\in (0,1)$.
\end{theorem}
\begin{proof}
For the ease of reading, we split the proof into five steps.
\subsection*{Step 1: decay estimates at the first scale} For $j\in \N\cup \{0\}$, $\nu$ defined as above and $\tau,\theta$ from Propositions \ref{p1}, \ref{p3} respectively, we introduce the following notation: $\tau_{j}:=\tau^{j}$, $\theta_{j}:=\theta^{j}$, $\nu_{j}:=\nu^{j}$ with $\tau_{0}\equiv\theta_{0}\equiv\nu_{0}=1$, $r_{1}:=\nu_{1}\rr$ and, for $s> 0$ we set:
\begin{flalign*}
&\mf{F}(s):=\mf{F}(u;B_{s}(x_{0})),\qquad\qquad  V(s):=\snr{(V_{p}(Du))_{B_{s}(x_{0})}},\nonumber \\
&\mf{C}(s):=\mf{C}(x_{0};s),\quad\qquad\qquad \quad \mf{S}(s):=\left(s^{m}\mint_{B_{s}(x_{0})}\snr{f}^{m} \ \dx\right)^{1/m}\nonumber \\
&\mf{H}_{\mf{s}}:=\sup_{s\le \rr/4}\mf{S}(s),\qquad\qquad\quad  \ \ \mf{K}_{\mf{s}}:=\sup_{s\le \rr/4}\mf{K}\left(\mf{S}(s)\right).
\end{flalign*} We then estimate
\begin{flalign}
\left\{
\begin{array}{c}
\displaystyle
\ \mf{F}(r_{1})\stackrel{\eqref{tri.1}_{1}}{\le}2^{1+n}\mf{F}(\rr)\stackrel{\eqref{6.5}_{2}}{<}2^{n+1}\hat{\varepsilon}\stackrel{\eqref{6.9}_{1}}{<}\frac{\varepsilon_{2}(\tau\theta)^{4npq}}{2^{8npq}}<\varepsilon_{2}\\[8pt]\displaystyle 
\ V(r_{1})\stackrel{\eqref{tri.1}_{2}}{\le}2^{n}\mf{F}(\rr)+V(\rr)\stackrel{\eqref{6.5}}{<} 2^{n}\hat{\varepsilon}+M\stackrel{\eqref{6.9}_{1}}{\le}\frac{1}{2}+M.
\end{array}
\right.\label{6.8}
\end{flalign}
 With the content of display \eqref{6.8} at hand, we can start iterations.
 \subsection*{Step 2: maximal iteration chains} Let us recall from \cite[Section 12.4]{kumig} the definition of maximal iteration chains. Given any nonempty set of indices $\mathcal{J}_{0}\subset \N\cup \{0\}$, for $\kk\in \N$ the maximal iteration chain of length $\kk$ starting at $\iota$ is defined as:
\begin{eqnarray*}
\mathcal{C}^{\kk}_{\iota}:=\left\{j\in \N\cup\{0\} \colon \iota \le j\le \iota+\kk,\  \iota \in \mathcal{J}_{0}, \ \iota+\kk+1\in \mathcal{J}_{0}, \ j\not \in \mathcal{J}_{0} \ \mbox{if} \ j>\iota\right\},
\end{eqnarray*}
i.e., $\mathcal{C}^{\kk}_{\iota}=\{\iota,\iota+1,\cdots,\iota+\kk\}$ and all its elements lie outside $\mathcal{J}_{0}$ except $\iota$, which belongs to $\mathcal{J}_{0}$. Furthermore, $\mathcal{C}^{\kk}_{\iota}$ is maximal, in the sense that it cannot be properly contained in any other set of the same kind. Similarly, the infinite maximal chain starting at $\iota$ is given by
\begin{eqnarray*}
\mathcal{C}^{\infty}_{\iota}:=\left\{j\in \N\cup\{0\} \colon \iota \le j<\infty ,\  \iota \in \mathcal{J}_{0},  \ j\not \in \mathcal{J}_{0} \ \mbox{if} \ j>\iota\right\}.
\end{eqnarray*}
We then look at two different alternatives:
\begin{eqnarray}\label{6.11}
\mf{C}(r_{1})>\left( \frac{H\mf{H}_{\mf{s}}}{\varepsilon_{0}}\right)^{\frac{p}{2(p-1)}}\qquad \mbox{or}\qquad \mf{C}(r_{1})\le \left(\frac{H\mf{H}_{\mf{s}}}{\varepsilon_{0}}\right)^{\frac{p}{2(p-1)}},
\end{eqnarray}
with $\mf{H}_{\mf{s}}$ and $r_{1}$ defined at the beginning of \textbf{Step 1}, $H$ is the constant in \eqref{6.9}$_{2}$ and $\varepsilon_{0}$ is the same parameter appearing in \eqref{ns.0}.

\subsection*{Step 3: large composite excess at the first scale} In order to repeatedly apply \eqref{nsns}, \eqref{ns.22} and \eqref{s.00} while keeping under control the various parameters involved and avoiding any blow-up of the bounding constants, let us prepare the set-up for the Blocks and Chains technique introduced in \cite[Section 5.2]{deqc}. We assume that $\eqref{6.11}_{1}$ holds and, with $\nu$ as in \textbf{Step 1}, we consider the set of indices
\begin{eqnarray*}
\mathcal{J}_{0}:=\left\{j\in \N\cup\{0\}\colon \mf{C}(\nu_{j}r_{1})>\left( \frac{H\mf{H}_{\mf{s}}}{\varepsilon_{0}}\right)^{\frac{p}{2(p-1)}}\right\}.
\end{eqnarray*}
Notice that $\mathcal{J}_{0}\not=\{\emptyset\}$ by $\eqref{6.11}_{1}$. We then look at two possibilities:
\begin{itemize}
    \item[\emph{i.}] there is at least one maximal iteration chain $\mathcal{C}^{\kk}_{\iota}$ starting at $\iota\in \mathcal{J}_{0}$ for some $\kk\le \infty$;
    \item[\emph{ii.}] $\mathcal{J}_{0}\equiv \N\cup\{0\}$.
\end{itemize}
We first examine occurrence ($\emph{i.}$) at its worst: we assume that there are (countably) infinitely many finite maximal iteration chains $\{\mathcal{C}^{\kk_{d}}_{\iota_{d}}\}_{d\in \N}$ corresponding to the discrete sequences $\{\iota_{d}\}_{d\in \N},\{\kk_{d}\}_{d\in \N}\subset \N$. By maximality it is easy to see that $\mathcal{C}^{\kk_{d_{1}}}_{\iota_{d_{1}}}\cap \mathcal{C}^{\kk_{d_{2}}}_{\iota_{d_{2}}}=\{\emptyset\}$ for $d_{1}\not =d_{2}$ and 
\begin{eqnarray}\label{6.12}
 \iota_{d+1}\ge \iota_{d}+\kk_{d}+1 \ \Longrightarrow \ \{\iota_{d}\}_{d\in \N} \ \mbox{is increasing and} \ \iota_{d}\to \infty.
\end{eqnarray}
By \eqref{6.12} we can split the reference interval $(0,r_{1}]$ into the union of disjoint blocks as $(0,r_{1}]=\bigcup_{d\in \N\cup\{0\}}\texttt{B}_{d}$, where it is $\texttt{B}_{0}:=\texttt{I}_{0}\cup\texttt{I}_{1}^{1}\cup \texttt{K}_{1}$, $\texttt{B}_{d}:=\texttt{I}_{d}^{2}\cup\texttt{I}_{d+1}^{1}\cup \texttt{K}_{d+1}$ for $d\in \N$, with
\begin{flalign*}
&\texttt{I}_{0}:=(\nu_{\iota_{1}}r_{1},r_{1}],\qquad\qquad \  \texttt{K}_{d}:=(\nu_{\iota_{d}+\kk_{d}+1}r_{1},\nu_{\iota_{d}+1}r_{1}]\nonumber \\ &\texttt{I}_{d}^{1}:=(\nu_{\iota_{d}+1}r_{1},\nu_{\iota_{d}}r_{1}],\qquad \texttt{I}_{d}^{2}:=(\nu_{\iota_{d+1}}r_{1},\nu_{\iota_{d}+\kk_{d}+1}r_{1}],
\end{flalign*}
and we shall implicitly identify $\texttt{I}_{0}\equiv \texttt{I}^{2}_{0}$. By construction, the intervals described in the above display are disjoint and only $\texttt{I}^{2}_{d}$ may be empty. The very definition of maximal iteration chains for the choice of $\mathcal{J}_{0}$ made above yields that
\begin{flalign}\label{6.13}
\left\{
\begin{array}{c}
\displaystyle
\ \mf{C}(\nu_{\iota_{d}}r_{1})>\left( \frac{H\mf{H}_{\mf{s}}}{\varepsilon_{0}}\right)^{\frac{p}{2(p-1)}} \qquad \mbox{for all} \ \ d\in \N\\[8pt]\displaystyle 
\ \mf{C}(\nu_{j}r_{1})\le \left( \frac{H\mf{H}_{\mf{s}}}{\varepsilon_{0}}\right)^{\frac{p}{2(p-1)}} \qquad \mbox{for all} \ \ j\in \{\iota_{d}+1,\cdots,\iota_{d}+\kk_{d}\}, \ \ d\in \N,
\end{array}
\right.
\end{flalign}
so if $\varsigma\in \texttt{K}_{d}$ we can find $j_{\varsigma}\in \{\iota_{d}+1,\cdots, \iota_{d}+\kk_{d}\}$ such that $\nu_{j_{\varsigma}+1}r_{1}<\varsigma\le \nu_{j_{\varsigma}}r_{1}$ and
\begin{eqnarray}\label{6.14}
\mf{C}(\varsigma)\stackrel{\eqref{ls.43.1}}{\le}2^{2+n}\mf{C}(\nu_{j_{\varsigma}}r_{1})\stackrel{\eqref{6.13}_{2}}{\le}2^{n+2}\left( \frac{H\mf{H}_{\mf{s}}}{\varepsilon_{0}}\right)^{\frac{p}{2(p-1)}}.
\end{eqnarray}
On the other hand, if $\varsigma\in \texttt{I}_{0}$ or $\varsigma\in\texttt{I}_{d}^{2}$, $d\in \N$, it is possible to determine $j_{\varsigma}\in \{0,\cdots,\iota_{1}-1\}$ or $j_{\varsigma}\in \{\iota_{d}+\kk_{d}+1,\cdots,\iota_{d+1}-1\}$ verifying $\nu_{j_{\varsigma}+1}r_{1}<\varsigma\le \nu_{j_{\varsigma}}r_{1}$ and
\begin{eqnarray}\label{6.15}
\mf{C}(\varsigma)\stackrel{\eqref{ls.43.1}}{\ge}\frac{1}{2^{2+n}}\mf{C}(\nu_{j_{\varsigma}+1}r_{1})\stackrel{\eqref{6.13}_{1}}{>}\frac{1}{2^{2+n}}\left( \frac{H\mf{H}_{\mf{s}}}{\varepsilon_{0}}\right)^{\frac{p}{2(p-1)}}.
\end{eqnarray}
Next, if $\texttt{I}_{d}^{2}=\{\emptyset\}$ so $\texttt{B}_{d}=\texttt{I}^{1}_{d+1}\cup \texttt{K}_{d+1}$, it turns out that the adjacent blocks $\texttt{B}_{d-1}$-$\texttt{B}_{d}$ contain two consecutive chains. In fact, in this case it is $\iota_{d+1}=\iota_{d}+\kk_{d}+1$, therefore $\texttt{B}_{d-1}\cup \texttt{B}_{d}=\texttt{I}^{2}_{d-1}\cup\texttt{I}^{1}_{d}\cup\texttt{K}_{d}\cup \texttt{I}^{1}_{d+1}\cup\texttt{K}_{d+1}$ and if in particular $\varsigma\in \texttt{K}_{d}\cup \texttt{I}^{1}_{d+1}\cup \texttt{K}_{d+1}$, there is $j_{\varsigma}\in \{\iota_{d}+1,\cdots,\iota_{d+1}+\kk_{d+1}\}$ such that
\begin{eqnarray*}
\left\{
\begin{array}{c}
\displaystyle
\ \mf{C}(\varsigma)\stackrel{\eqref{ls.43.1}}{\le}2^{n+2}\mf{C}(\nu_{j_{\varsigma}}r_{1})\stackrel{\eqref{6.13}_{2}}{<}2^{n+2}\left( \frac{H\mf{H}_{\mf{s}}}{\varepsilon_{0}}\right)^{\frac{p}{2(p-1)}}\qquad \mbox{if} \ \ j_{\varsigma}\not =\iota_{d+1}\\[8pt]\displaystyle 
\ \mf{C}(\varsigma)\stackrel{\eqref{ls.43.1}}{\le}2^{n+2}\mf{C}(\nu_{\iota_{d+1}}r_{1})\stackrel{\eqref{ls.43.1}}{\le}2^{2n+4}\mf{C}(\nu_{\iota_{d}+\kk_{d}}r_{1})\stackrel{\eqref{6.13}_{2}}{<}2^{2n+4}\left( \frac{H\mf{H}_{\mf{s}}}{\varepsilon_{0}}\right)^{\frac{p}{2(p-1)}}\qquad \mbox{if} \ \ j_{\varsigma}=\iota_{d+1},
\end{array}
\right.
\end{eqnarray*}
so in any case we have that
\begin{eqnarray}\label{6.16}
\mf{C}(\varsigma)<2^{2n+4}\left( \frac{H\mf{H}_{\mf{s}}}{\varepsilon_{0}}\right)^{\frac{p}{2(p-1)}}\qquad \mbox{for all} \ \ \varsigma \in \texttt{K}_{d}\cup \texttt{B}_{d}.
\end{eqnarray}
We then consider two occurrences:
\begin{eqnarray}\label{6.17}
\varepsilon_{0}V(r_{1})\le \mf{F}(r_{1})\qquad \mbox{or}\qquad \varepsilon_{0}V(r_{1})>\mf{F}(r_{1})
\end{eqnarray}
assume that $\eqref{6.17}_{1}$ holds and introduce a second set of indices defined as $$\mathcal{J}_{1}:=\left\{j\in \N\cup\{0\}\colon \varepsilon_{0}V(\theta_{j}r_{1})\le \mf{F}(\theta_{j}r_{1})\right\},$$ which is nonempty given that $0\in \mathcal{J}_{1}$ by $\eqref{6.17}_{1}$.
\subsubsection*{Step 3.1: the singular regime is stable} In this case 
\eqn{6.24}
$$\mathcal{J}_{1}\equiv \N\cup\{0\},$$ 
so we can ignore the presence of blocks $\{\texttt{B}_{d}\}_{d\in \N\cup\{0\}}$ and proceed in a more standard way, cf. \cite{dumi,ts2}. By \eqref{6.8}, \eqref{6.3} and $\eqref{6.17}_{1}$ we see that Proposition \ref{p3} applies and gives
\begin{flalign}
\left\{
\begin{array}{c}
\displaystyle
\ \mf{F}(\theta_{1} r_{1})\stackrel{\eqref{s.00}}{\le}\theta^{\alpha_{0}}\mf{F}(r_{1})+c_{1}\mf{K}\left(\mf{S}(r_{1})\right)^{\frac{p}{2(p-1)}}\stackrel{\eqref{6.8}_{1},\eqref{6.3}}{<}\varepsilon_{2}\\[8pt]\displaystyle 
\ V(\theta_{1} r_{1})\stackrel{\eqref{6.24}}{\le} \frac{\mf{F}(\theta_{1} r_{1})}{\varepsilon_{0}}\stackrel{\eqref{6.18}_{1}}{\le}\frac{1}{\varepsilon_{0}}\left(\theta^{\alpha_{0}}\mf{F}(r_{1})+c_{1}\mf{K}(\mf{S}(r_{1}))^{\frac{p}{2(p-1)}}\right)\stackrel{\eqref{tri.1}_{1}}{\le}\frac{1}{\varepsilon_{0}}\left(2^{n+1}\mf{F}(\rr)+c_{1}\mf{K}_{\mf{s}}^{\frac{p}{2(p-1)}}\right)\\[8pt]\displaystyle \ V(\theta_{1}r_{1})\stackrel{\eqref{6.18}_{2},\eqref{6.8}_{1},\eqref{6.3}}{\le}1,\label{6.18}
\end{array}
\right.
\end{flalign}
where we also used that by $\eqref{pq}_{2}$ it is $\frac{p}{2(p-1)}>1$.
Let us fix $j\in \N$ and assume that
\begin{eqnarray}\label{6.19}
\mf{F}(\theta_{i}r_{1})<\varepsilon_{2}\qquad \mbox{for all} \ \ i\in \{0,\cdots,j\}.
\end{eqnarray}
As a consequence of \eqref{6.24} and \eqref{6.19}, we have
\begin{eqnarray}\label{6.20}
V(\theta_{i}r_{1})&\le&\frac{\mf{F}(\theta_{i}r_{1})}{\varepsilon_{0}}\stackrel{\eqref{6.19}}{<}\frac{\varepsilon_{2}}{\varepsilon_{0}}\stackrel{\eqref{s.11},\eqref{s.12}}{\le}1. 
\end{eqnarray}
Thanks to \eqref{6.19}-\eqref{6.20} we can apply \eqref{s.00} at the $\theta_{j}r_{1}$-scale to get
\begin{eqnarray}\label{6.21}
\mf{F}(\theta_{j+1}r_{1})&\stackrel{\eqref{s.00}}{\le}&\theta^{\alpha_{0}}\mf{F}(\theta_{j}r_{1})+c_{1}\mf{K}\left(\mf{S}(\theta_{j}r_{1})\right)^{\frac{p}{2(p-1)}}\nonumber \\
&\le&\theta^{\alpha_{0}(j+1)}\mf{F}(r_{1})+c_{1}\sum_{i=0}^{j}\theta^{\alpha_{0}(j-i)}\mf{K}\left(\mf{S}(\theta_{i}r_{1})\right)^{\frac{p}{2(p-1)}}\nonumber \\
&\le&\theta^{\alpha_{0}(j+1)}\mf{F}(r_{1})+c_{3}\mf{K}_{\mf{s}}^{\frac{p}{2(p-1)}}\stackrel{\eqref{6.3},\eqref{6.8}_{1}}{<\varepsilon_{2}}
\end{eqnarray}
and, via \eqref{6.24}, \eqref{6.21}, \eqref{s.11}, \eqref{s.12},
\begin{eqnarray}\label{6.22}
V(\theta_{j+1}r_{1})\le \frac{\mf{F}(\theta_{j+1}r_{1})}{\varepsilon_{0}}\le \frac{\varepsilon_{2}}{\varepsilon_{0}}\le 1.
\end{eqnarray}
The arbitrariety of $j\in \N$ and \eqref{6.24} allow concluding that \eqref{6.21}-\eqref{6.22} hold for all $j\in \N\cup\{0\}$; in particular it is
\begin{flalign}
V(\theta_{j+1}r_{1})\stackrel{\eqref{6.24}}{\le}\frac{\mf{F}(\theta_{j+1}r_{1})}{\varepsilon_{0}}\stackrel{\eqref{6.21}}{\le}\frac{1}{\varepsilon_{0}}\left(\theta^{\alpha_{0}(j+1)}\mf{F}(r_{1})+c_{3}\mf{K}_{\mf{s}}^{\frac{p}{2(p-1)}}\right)\stackrel{\eqref{tri.1}}{\le}\frac{1}{\varepsilon_{0}}\left(2^{n+1}\mf{F}(\rr)+c_{3}\mf{K}_{\mf{s}}^{\frac{p}{2(p-1)}}\right),\label{v.1}
\end{flalign}
for all $j\in \N\cup\{0\}$. Standard interpolative arguments and \eqref{6.8} then yield that whenever $\varsigma\in (0,r_{1}]$ there is $j_{\varsigma}\in \mathcal{J}_{1}$ such that $\theta_{j_{\varsigma}+1}r_{1}<\varsigma\le \theta_{j_{\varsigma}}r_{1}$,
\begin{eqnarray}\label{6.34}
\mf{F}(\varsigma)\le \frac{2^{4+n}}{\theta^{1+n/2}}\left(\frac{\varsigma}{\rr}\right)^{\alpha_{0}}\mf{F}(\rr)+\frac{2c_{3}}{\theta^{n/2}}\mf{K}_{\mf{s}}^{\frac{p}{2(p-1)}},\qquad \qquad V(\varsigma)\le \frac{3}{2}
\end{eqnarray}
and
\begin{eqnarray}\label{v.2}
V(\varsigma)&\le& \snr{V(\varsigma)-V(\theta_{j_{\varsigma}}r_{1})}+V(\theta_{j_{\varsigma}}r_{1})\stackrel{\eqref{ls.42.1}_{2},\eqref{v.1}}{\le}\frac{\mf{F}(\theta_{j_{\varsigma}}r_{1})}{\theta^{n/2}}+\frac{1}{\varepsilon_{0}}\left(2^{n+1}\mf{F}(\rr)+c_{3}\mf{K}_{\mf{s}}^{\frac{p}{2(p-1)}}\right)\nonumber \\
&\stackrel{\eqref{6.21}}{\le}&\frac{2^{n+4}c_{3}}{\theta^{n/2}\varepsilon_{0}}\left(\mf{F}(\rr)+V(\rr)+\mf{K}_{\mf{s}}^{\frac{p}{2(p-1)}}\right)=:\mf{V}_{1}.
\end{eqnarray}
Finally, if $\varsigma\in (r_{1},\rr]$ via \eqref{6.8} and \eqref{ls.42.1} it is
\begin{eqnarray}\label{6.35}
\mf{F}(\varsigma)\le 2^{4+n}\left(\frac{\varsigma}{\rr}\right)^{\alpha_{0}}\mf{F}(\rr)\qquad \mbox{and}\qquad V(\varsigma)\le 2^{n}\mf{F}(\rr)+V(\rr)< 1+M
\end{eqnarray}
Merging the content of the three previous displays we obtain
\begin{eqnarray}\label{final.1}
\left\{
\begin{array}{c}
\displaystyle
\ \mf{F}(\varsigma)\le \frac{2^{4+n}}{\theta^{1+n/2}}\left(\frac{\varsigma}{\rr}\right)^{\alpha_{0}}\mf{F}(\rr)+\frac{2c_{3}}{\theta^{n/2}}\mf{K}_{\mf{s}}^{\frac{p}{2(p-1)}}\\[8pt]\displaystyle 
\ V(\varsigma)\le 2+M,\qquad V(\varsigma)\le \mf{V}_{1}
\end{array}
\right.\qquad \mbox{for all} \ \ \varsigma\in (0,\rr].
\end{eqnarray}
\subsubsection*{Step 3.2: first change of scale} If \eqref{6.24} does not hold, there exists $j_{1}\in \N$ such that
\begin{eqnarray*}
j_{1}:=\min\{j\in \N\colon \varepsilon_{0}V(\theta_{j}r_{1})>\mf{F}(\theta_{j}r_{1})\},
\end{eqnarray*}
and by \eqref{6.17}$_{1}$ it is $j_{1}\ge 1$. We can rephrase the minimality character of $j_{1}$ as
\begin{eqnarray}\label{6.25}
\varepsilon_{0}V(\theta_{j}r_{1})\le \mf{F}(\theta_{j}r_{1}) \ \ \mbox{for all} \ \ j\in \{0,\cdots, j_{1}-1\}\qquad \mbox{and}\qquad \varepsilon_{0}V(\theta_{j_{1}}r_{1})>\mf{F}(\theta_{j_{1}}r_{1}),
\end{eqnarray}
therefore by $\eqref{6.25}_{1}$ and \eqref{6.3} we deduce that \eqref{6.21}-\eqref{v.1} hold for all $j\in \{0,\cdots,j_{1}-1\}$. In particular, it is
\begin{eqnarray}\label{6.26}
\mf{F}(\theta_{j_{1}}r_{1})\le \theta^{\alpha_{0}j_{1}}\mf{F}(r_{1})+c_{3}\mf{K}_{\mf{s}}^{\frac{p}{2(p-1)}},\qquad V(\theta_{j_{1}}r_{1})\le \frac{3}{2},\qquad V(\theta_{j_{1}}r_{1})\le \mf{V}_{1},
\end{eqnarray}
with $\mf{V}_{1}$ being defined in \eqref{v.2}. We set $r_{2}:=\theta_{j_{1}}r_{1}$ and introduce a new set of indices
\begin{eqnarray*}
\mathcal{J}_{2}:=\left\{j\in \N\cup \{0\}\colon \varepsilon_{0}V(\tau_{j}r_{2})>\mf{F}(\tau_{j}r_{2})\right\},
\end{eqnarray*}
which is nonempty given that $0\in \mathcal{J}_{2}$ because of $\eqref{6.25}_{2}$ and the very definition of $r_{2}$.
\subsubsection*{Step 3.3: the nonsingular regime is stable} Let us assume that
\begin{eqnarray}\label{6.27}
\mathcal{J}_{2}\equiv \N\cup\{0\} \ \Longrightarrow \ \varepsilon_{0}V(\tau_{j}r_{2})>\mf{F}(\tau_{j}r_{2}) \ \ \mbox{for all} \ \ j\in \N\cup\{0\}.
\end{eqnarray}
By induction, we want to show that
\begin{eqnarray}\label{6.28}
\left\{
\begin{array}{c}
\displaystyle
\ V(\tau_{j}r_{2})\le 2,\qquad V(\tau_{j}r_{2})\le \mf{V}_{2} \\[8pt]\displaystyle
\ \mf{F}(\tau_{j+1}r_{2})\le \tau^{(j+1)\alpha_{0}}\mf{F}(r_{2})+c_{0}\sum_{i=0}^{j}\tau^{\alpha_{0}(j-i)}V(\tau_{i}r_{2})^{(2-p)/p}\mf{S}(\tau_{i}r_{2})
\end{array}
\right.
\end{eqnarray}
for all $j\in \N\cup\{0\}$, where we set
$$
\mf{V}_{2}:=c_{3}'\left(\mf{F}(\rr)+V(\rr)+\mf{K}\left(\mathbf{I}^{f}_{1,m}(x_{0},\rr)\right)^{\frac{p}{2(p-1)}}\right)
$$
and $c_{3}'\equiv c_{3}'(\textnormal{\texttt{data}},M)$ has been defined at the beginning of Section \ref{morsec}. For $j=0$, by $\eqref{6.26}_{2,3}$, $\eqref{6.25}_{2}$ and Propositions \ref{p1}-\ref{p2} we obtain
\begin{eqnarray}\label{6.28.1}
\left\{
\begin{array}{c}
\displaystyle
\ \mf{F}(\tau_{1}r_{2})\stackrel{\eqref{nsns},\eqref{ns.22}}{\le}\tau^{\alpha_{0}}\mf{F}(r_{2})+c_{0}V(r_{2})^{(2-p)/p}\mf{S}(r_{2})\\[8pt]\displaystyle 
\ V(r_{2})\le \frac{3}{2},\qquad V(r_{2})\le \mf{V}_{1},
\end{array}
\right.
\end{eqnarray}
thus, recalling that 
\eqn{v.4}
$$\mf{V}_{1}<2^{-4}\mf{V}_{2}$$ by definition, \eqref{6.28} is proven for $j=0$. Next, let us fix $j\in \N$ and assume the validity of \eqref{6.28} for all $i\in \{0,\cdots, j\}$. In particular it holds that
\begin{eqnarray}\label{6.29}
\left\{
\begin{array}{c}
\displaystyle
\mf{F}(\tau_{i+1}r_{2})\le \tau^{\alpha_{0}(i+1)}\mf{F}(r_{2})+c_{0}\sum_{k=0}^{i}\tau^{\alpha_{0}(i-k)}V(\tau_{k}r_{2})^{(2-p)/p}\mf{S}(\tau_{k}r_{2})
\\[8pt]\displaystyle 
\ V(\tau_{i}r_{2})\le 2,\qquad V(\tau_{i}r_{2})\le \mf{V}_{2},
\end{array}
\right.
\end{eqnarray}
for all $i\in \{0,\cdots,j\}$, therefore we estimate using the discrete Fubini theorem and Young inequality with conjugate exponents $\left(\frac{p}{2-p},\frac{p}{2(p-1)}\right)$,
\begin{eqnarray}\label{6.30}
V(\tau_{j+1}r_{2})&\le&V(r_{2})+\sum_{i=0}^{j}\snr{V(\tau_{i+1}r_{2})-V(\tau_{i}r_{2})}\stackrel{\eqref{tri.1}_{3},\eqref{6.28.1}_{2}}{\le}\mf{V}_{1}+\frac{1}{\tau^{n/2}}\sum_{i=0}^{j}\mf{F}(\tau_{i}r_{2})\nonumber \\
&\stackrel{\eqref{6.29}_{1}}{\le}&\mf{V}_{1}+\frac{\mf{F}(r_{2})}{\tau^{n/2}}+\frac{\mf{F}(r_{2})}{\tau^{n/2}}\sum_{i=0}^{j-1}\tau^{\alpha_{0}(i+1)}+\frac{c_{0}}{\tau^{n/2}}\sum_{i=0}^{j-1}\sum_{k=0}^{i}\tau^{\alpha_{0}(i-k)}V(\tau_{k}r_{2})^{\frac{2-p}{p}}\mf{S}(\tau_{k}r_{2})\nonumber \\
&\stackrel{\eqref{6.26}}{\le}&\mf{V}_{1}+\frac{2}{\tau^{n/2}\mf{m}}\left(\theta^{\alpha_{0}j_{1}}\mf{F}(r_{1})+c_{3}\mf{K}_{\mf{s}}^{\frac{p}{2(p-1)}}\right)+\frac{c_{0}}{\tau^{n/2}}\sum_{i=0}^{j}\sum_{k=0}^{i}\tau^{\alpha_{0}(i-k)}V(\tau_{k}r_{2})^{\frac{2-p}{p}}\mf{S}(\tau_{k}r_{2})\nonumber \\
&\stackrel{\eqref{tri.1}_{1}}{\le}&\mf{V}_{1}+\frac{2}{\tau^{n/2}\mf{m}}\left(2^{n+1}\mf{F}(\rr)+c_{3}\mf{K}_{\mf{s}}^{\frac{p}{2(p-1)}}\right)+\frac{c_{0}}{\tau^{n/2}}\sum_{k=0}^{j}V(\tau_{k}r_{2})^{\frac{2-p}{p}}\mf{S}(\tau_{k}r_{2})\left(\sum_{i=k}^{j}\tau^{\alpha_{0}(i-k)}\right)\nonumber \\
&\stackrel{\eqref{6.2},\eqref{6.29}_{2}}{\le}&\mf{V}_{1}+\frac{2}{\tau^{n/2}\mf{m}}\left(2^{n+1}\mf{F}(\rr)+\frac{c_{3}2^{\frac{4npq}{p-1}}}{(\tau\theta)^{\frac{2npq}{p-1}}}\mf{K}\left(\mathbf{I}^{f}_{1,m}(x_{0},\rr)\right)^{\frac{p}{2(p-1)}}\right)\nonumber \\
&&+\frac{c_{0}\mf{V}_{2}^{\frac{2-p}{p}}}{\tau^{n/2}\mf{m}}\left(\mf{S}(r_{2})+\sum_{k=0}^{j-1}\mf{S}(\tau_{k+1}r_{2})\right)\nonumber \\
&\stackrel{\eqref{v.4},\eqref{6.2}}{\le}&\left(\frac{1}{2^{4}}+\frac{1}{2^{2}}\right)\mf{V}_{2}+\frac{2^{n+2}\mf{F}(\rr)}{\tau^{n/2}\mf{m}}+\frac{2^{4n+1}c_{0}}{(\tau\theta)^{3n}\mf{m}}\mf{V}_{2}^{\frac{2-p}{p}}\mathbf{I}^{f}_{1,m}(x_{0},\rr)\nonumber \\
&&+\frac{c_{3}2^{\frac{4npq}{p-1}+1}}{\mf{m}(\tau\theta)^{\frac{3npq}{p-1}}}\mf{K}\left(\mathbf{I}^{f}_{1,m}(x_{0},\rr)\right)^{\frac{p}{2(p-1)}}\nonumber \\
&\stackrel{\eqref{6.9},\eqref{6.5}_{2}}{\le}&\left(\frac{1}{2^{4}}+\frac{1}{2^{2}}+\frac{1}{2^{10}}\right)\mf{V}_{2}+\frac{c_{3}2^{\frac{8npq}{p-1}}}{\mf{m}(\tau\theta)^{\frac{4npq}{p-1}}}\mf{K}\left(\mathbf{I}^{f}_{1,m}(x_{0},\rr)\right)^{\frac{p}{2(p-1)}}\nonumber \\
&\stackrel{\eqref{6.1}}{\le}&\left(\frac{1}{2^{4}}+\frac{1}{2^{2}}+\frac{1}{2^{10}}+\frac{1}{2^{20}}\right)\mf{V}_{2}\le \mf{V}_{2}
\end{eqnarray}
and, estimating $V(\tau_{j+1}r_{2})$ in a slightly different way than \eqref{6.30} we also get
\begin{eqnarray}\label{6.30.1}
V(\tau_{j+1}r_{2})&\le& V(r_{2})+\frac{2\mf{F}(r_{2})}{\tau^{n/2}\mf{m}}+\frac{c_{0}}{\tau^{n/2}}\sum_{i=0}^{j-1}\sum_{k=0}^{i}\tau^{\alpha_{0}(i-k)}V(\tau_{k}r_{2})^{(2-p)/p}\mf{S}(\tau_{k}r_{2})\nonumber \\
&\stackrel{\eqref{6.29}_{2}}{\le}&V(r_{2})+\frac{2\mf{F}(r_{2})}{\tau^{n/2}\mf{m}}+\frac{c_{0}2^{(2-p)/p}}{\tau^{n/2}\mf{m}}\sum_{k=0}^{j}\mf{S}(\tau_{k}r_{2})\nonumber \\
&\stackrel{\eqref{6.26}}{\le}&\frac{3}{2}+\frac{2}{\tau^{n/2}\mf{m}}\left(\theta^{j_{1}\alpha_{0}}\mf{F}(r_{1})+c_{3}\mf{K}_{\mf{s}}^{\frac{p}{2(p-1)}}\right)+\frac{c_{0}2^{(2-p)/p}}{\tau^{n/2}\mf{m}}\left(\mf{S}(r_{2})+\sum_{k=0}^{j}\mf{S}(\tau_{k+1}r_{2})\right)\nonumber \\
&\stackrel{\eqref{tri.1}_{1},\eqref{6.2}}{\le}&\frac{3}{2}+\frac{2}{\tau^{n/2}\mf{m}}\left(2^{n+1}\mf{F}(\rr)+c_{3}\mf{K}_{\mf{s}}^{\frac{p}{2(p-1)}}\right)+\frac{c_{0}2^{4n+2}\mathbf{I}^{f}_{1,m}(x_{0},\rr)}{(\tau\theta)^{3n}\mf{m}}\nonumber \\
&\stackrel{\eqref{6.9}_{1},\eqref{6.3}}{\le}&\frac{3}{2}+\frac{1}{2^{10}}+\frac{1}{2^{20}}+\frac{c_{0}2^{4n+2}\mathbf{I}^{f}_{1,m}(x_{0},\rr)}{(\tau\theta)^{3n}\mf{m}}\stackrel{\eqref{6.1}}{\le}\frac{3}{2}+\frac{1}{2^{10}}+\frac{1}{2^{20}}+\frac{1}{2^{20}}\le 2.
\end{eqnarray}
We can then combine \eqref{6.27}, \eqref{6.30}-\eqref{6.30.1} and Proposition \ref{p1}-\ref{p2} to get
\begin{eqnarray}\label{6.31}
\mf{F}(\tau_{j+2}r_{2})&\le&\tau^{\alpha_{0}}\mf{F}(\tau_{j+1}r_{2})+c_{0}V(\tau_{j+1}r_{2})^{(2-p)/p}\mf{S}(\tau_{j+1}r_{2})\nonumber \\
&\stackrel{\eqref{6.28}_{2}}{\le}&\tau^{\alpha_{0}(j+2)}\mf{F}(r_{2})+c_{0}\sum_{k=0}^{j+1}\tau^{\alpha_{0}(j+1-k)}V(\tau_{k}r_{2})^{(2-p)/p}\mf{S}(\tau_{k}r_{2}).
\end{eqnarray}
Inequalities \eqref{6.30}-\eqref{6.31} prove the validity of the induction step, so by the arbitrariety of $j\in \N$ we can conclude that \eqref{6.28} holds for all $j\in \N\cup\{0\}$ and, once established this, we can refine \eqref{6.28}$_{2}$ as
\begin{eqnarray}\label{6.32}
\mf{F}(\tau_{j+1}r_{2})\le \tau^{(j+1)\alpha_{0}}\mf{F}(r_{2})+c_{3}\mf{V}_{2}^{(2-p)/p}\mf{H}_{\mf{s}}.
\end{eqnarray}
Next, if $\varsigma\in (0,r_{2}]$ there is $j_{\varsigma}\in \N\cup \{0\}$ such that $\tau_{j_{\varsigma}+1}r_{2}<\varsigma\le \tau_{j_{\varsigma}}r_{2}$ and, via \eqref{tri.1}, \eqref{ls.42.1}, \eqref{6.1}, \eqref{6.3}, \eqref{6.26}, \eqref{6.28} and \eqref{6.30}-\eqref{6.31} we have  
\begin{eqnarray*}
\left\{
\begin{array}{c}
\displaystyle
\ \mf{F}(\varsigma)\le \frac{2^{n+4}}{\tau^{1+n/2}}\left(\frac{\varsigma}{\rr}\right)^{\alpha_{0}}\mf{F}(\rr)+\frac{2c_{3}}{\tau^{n/2}}\left(\mf{K}_{\mf{s}}^{\frac{p}{2(p-1)}}+\mf{V}_{2}^{(2-p)/p}\mf{H}_{\mf{s}}\right)\\[8pt]\displaystyle 
\ V(\varsigma)\le 3,\qquad V(\varsigma)\le 2\mf{V}_{2}
\end{array}
\right.
\end{eqnarray*}
while if $\varsigma\in (r_{2},r_{1}]$ we can find $j_{\varsigma}\in \{0,\cdots,j_{1}-1\}$ verifying $\theta_{j_{\varsigma}+1}r_{1}<\varsigma\le \theta_{j_{\varsigma}}r_{1}$, so as done before we can confirm the validity of estimates \eqref{6.34}-\eqref{v.2}, and when $\varsigma\in (r_{1},\rr]$ the bounds in \eqref{6.35} trivially hold true. All in all, we can conclude with
\begin{eqnarray}\label{final.2}
\left\{
\begin{array}{c}
\displaystyle
\ \mf{F}(\varsigma)\le \frac{2^{n+4}}{(\tau\theta)^{1+n/2}}\left(\frac{\varsigma}{\rr}\right)^{\alpha_{0}}\mf{F}(\rr)+\frac{2c_{3}}{(\tau\theta)^{n/2}}\left(\mf{K}_{\mf{s}}^{\frac{p}{2(p-1)}}+\mf{V}_{2}^{(2-p)/p}\mf{H}_{\mf{s}}\right)\\[8pt]\displaystyle 
\ V(\varsigma)\le 3(1+M), \qquad V(\varsigma)\le 2\mf{V}_{2}
\end{array}
\right.
\end{eqnarray}
for all $\varsigma\in (0,\rr]$, where we accounted also for the case in which we directly started from \eqref{6.17}$_{2}$ in case of stable nonsingular regime - just set $j_{1}=0$, replace $r_{2}$ with $r_{1}$ above and recall \eqref{6.8}$_{2}$.
\subsubsection*{Step 3.4: second change of scale and block $\texttt{B}_{0}$} We now examine the case in which $\mathcal{J}_{2}\not \equiv \N\cup \{0\}$, i.e., there exists $j_{2}\in \N$ such that
\begin{eqnarray*}
j_{2}:=\min\left\{j\in \N\colon \varepsilon_{0}V(\tau_{j}r_{2})\le \mf{F}(\tau_{j}r_{2})\right\},
\end{eqnarray*}
and $\eqref{6.25}_{2}$ assures that $j_{2}\ge 1$.  The minimality of $j_{2}$ renders that
\begin{eqnarray}\label{6.33}
\varepsilon_{0}V(\tau_{j}r_{2})>\mf{F}(\tau_{j}r_{2}) \ \ \mbox{for all} \ \ j\in \{0,\cdots,j_{2}-1\}\qquad \mbox{and}\qquad \varepsilon_{0}V(\tau_{j_{2}}r_{2})\le \mf{F}(\tau_{j_{2}}r_{2}).
\end{eqnarray}
Set $r_{3}:=\tau_{j_{2}}r_{2}$ and notice that we can repeat the same procedure described in \emph{Step 3.3} a finite number of times for getting
\begin{eqnarray}\label{6.37}
\mf{F}(\tau_{j+1}r_{2})\le \tau^{\alpha_{0}(j+1)}\mf{F}(r_{2})+c_{3}\mf{V}_{2}^{(2-p)/p}\mf{H}_{\mf{s}},\qquad V(\tau_{j}r_{2})\le 2(M+1),\qquad V(\tau_{j}r_{2})\le \mf{V}_{2},
\end{eqnarray}
for all $j\in \{0,\cdots,j_{2}-1\}$\footnote{In comparison to \eqref{6.28}, here we included also the case in which we directly start from \eqref{6.17}$_{2}$. In fact, by \eqref{6.8}$_{2}$ the bound on averages increases by $2M$.}. Next we prove that
\begin{eqnarray}\label{6.36}
r_{3} \ \mbox{cannot belong to} \ \texttt{I}_{0} \ \mbox{or to} \ \texttt{I}^{2}_{d} \ \mbox{for all} \ d\in \N.
\end{eqnarray}
By contradiction, assume that \eqref{6.36} does not hold. We would then have
\begin{eqnarray}\label{6.38}
V(\tau_{j_{2}-1}r_{2})&\stackrel{\eqref{tri.1}_{3}}{\le}&V(r_{3})+\frac{1}{\tau^{n/2}}\mf{F}(\tau_{j_{2}-1}r_{2})\nonumber \\
&\stackrel{\eqref{6.33}_{1}}{\le}&V(r_{3})+\frac{\varepsilon_{0}}{\tau^{n/2}}V(\tau_{j_{2}-1}r_{2}) \ \stackrel{\eqref{ns.20}_{2}}{\Longrightarrow} \ 2V(r_{3})\ge V(\tau_{j_{2}-1}r_{2}),
\end{eqnarray}
so recalling that \eqref{6.33}$_{1}$-\eqref{6.37}$_{2}$ legalize the application of Propositions \ref{p1}-\ref{p2}, we obtain
\begin{eqnarray*}
\mf{F}(r_{3})&\stackrel{\eqref{nsns},\eqref{ns.22}}{\le}&\tau^{\alpha_{0}}\mf{F}(\tau_{j_{2}-1}r_{2})+c_{0}V(\tau_{j_{2}-1}r_{2})^{(2-p)/p}\mf{S}(\tau_{j_{2}-1}r_{2})\nonumber \\
&\stackrel{\eqref{6.33}_{1},\eqref{6.38}}{\le}&2\tau^{\alpha_{0}}\varepsilon_{0}V(r_{3})+c_{0}2^{(2-p)/p}V(r_{3})^{(2-p)/p}\mf{H}_{\mf{s}}\nonumber\\
&\stackrel{\eqref{6.15}}{\le}&2\tau^{\alpha_{0}}\varepsilon_{0}V(r_{3})+\frac{2^{2n+6}\varepsilon_{0}c_{0}\mf{C}(r_{3})}{H}\nonumber \\
&\stackrel{\eqref{tri.1}_{1}}{\le}&\varepsilon_{0}V(r_{3})\left(2\tau^{\alpha_{0}}+\frac{2^{2n+6}c_{0}}{H}\right)+\frac{2^{2n+8}\varepsilon_{0}c_{0}\mf{F}(\tau_{j_{2}-1}r_{2})}{\tau^{n/2}H}\nonumber \\
&\stackrel{\eqref{6.33}_{1},\eqref{6.38}}{\le}&\varepsilon_{0}V(r_{3})\left(2\tau^{\alpha_{0}}+\frac{2^{2n+6}c_{0}}{H}+\frac{2^{2n+10}\varepsilon_{0}c_{0}}{\tau^{n/2}H}\right)\stackrel{\eqref{ns.20}_{1},\eqref{6.9}_{2}}{<}\varepsilon_{0}V(r_{3}),
\end{eqnarray*}
thus contradicting $\eqref{6.33}_{2}$. Therefore \eqref{6.36} is true and in particular it holds that
\begin{eqnarray}\label{6.39}
\texttt{I}_{0}\subseteq (r_{3},r_{2}]\cup(r_{2},r_{1}].
\end{eqnarray}
This means that if $\varsigma\in \texttt{I}_{0}$ we can find $j_{\varsigma}\in \{0,\cdots,j_{1}-1\}$ or $j_{\varsigma}\in \{0,\cdots,j_{2}-1\}$ such that either $\theta_{j_{\varsigma}+1}r_{1}<\varsigma\le \theta_{j_{\varsigma}}r_{1}$ or $\tau_{j_{\varsigma}+1}r_{2}<\varsigma\le \tau_{j_{\varsigma}}r_{2}$ but in any case the estimates in \eqref{final.2} are valid. Next, we observe that
\begin{eqnarray}\label{6.40}
\mf{F}(\nu_{\iota_{1}}r_{1})&\stackrel{\eqref{tri.1}_{1}}{\le}&2^{1+n}\mf{F}(\nu_{\iota_{1}-1}r_{1})\nonumber \\
&\stackrel{\eqref{final.2}_{1}}{\le}&\frac{2^{2n+5}}{(\tau\theta)^{1+n/2}}\left(\frac{\nu_{\iota_{1}-1}r_{1}}{\rr}\right)^{\alpha_{0}}\mf{F}(\rr)+\frac{2^{n+2}c_{3}}{(\tau\theta)^{n/2}}\left(\mf{K}_{\mf{s}}^{\frac{p}{2(p-1)}}+\mf{V}_{2}^{(2-p)/p}\mf{H}_{\mf{s}}\right)\nonumber \\
&\le&\frac{2^{2n+7}}{(\tau\theta)^{1+n/2}}\left(\frac{\nu_{\iota_{1}}r_{1}}{\rr}\right)^{\alpha_{0}}\mf{F}(\rr)+\frac{2^{n+2}c_{3}}{(\tau\theta)^{n/2}}\left(\mf{K}_{\mf{s}}^{\frac{p}{2(p-1)}}+\mf{V}_{2}^{(2-p)/p}\mf{H}_{\mf{s}}\right)
\end{eqnarray}
and, concerning the average,
\begin{eqnarray}\label{6.41}
V(\nu_{\iota_{1}}r_{1})&\stackrel{\eqref{tri.1}_{2}}{\le}&2^{n}\mf{F}(\nu_{\iota_{1}-1}r_{1})+V(\nu_{\iota_{1}-1}r_{1})\nonumber \\
&\stackrel{\eqref{final.2}}{\le}&\frac{2^{2n+4}}{(\tau\theta)^{1+n/2}}\left(\frac{\nu_{\iota_{1}-1}r_{1}}{\rr}\right)^{\alpha_{0}}\mf{F}(\rr)+\frac{2^{n+1}c_{3}}{(\tau\theta)^{n/2}}\left(\mf{K}_{\mf{s}}^{\frac{p}{2(p-1)}}+\mf{V}_{2}^{(2-p)/p}\mf{H}_{\mf{s}}\right)+3(1+M)\nonumber \\
&\stackrel{\eqref{6.5}_{1}}{\le}&3(M+1)+\left(\frac{2^{2n+4}}{(\tau\theta)^{1+n/2}}+\frac{c_{3}'}{4}\right)\mf{F}(\rr)+\left(\frac{c_{3}2^{n+1}}{(\tau\theta)^{n/2}}+2^{\frac{2-p}{p-1}}\left(\frac{2^{2n+2}c_{3}}{(\tau\theta)^{n/2}}\right)^{\frac{p}{2(p-1)}}\right)\mf{K}_{\mf{s}}^{\frac{p}{2(p-1)}}\nonumber \\
&&+\frac{c'_{3}}{4}\mf{K}\left(\mathbf{I}^{f}_{1,m}(x_{0},\rr)\right)^{\frac{p}{2(p-1)}}+\frac{2^{n+1}c_{3}(c_{3}')^{(2-p)/p}}{(\tau\theta)^{n/2}}M^{(2-p)/p}\mf{H}_{\mf{s}}\nonumber \\
&\stackrel{\eqref{6.1},\eqref{6.3}}{\le}&\frac{1}{2}+3(M+1)+\mf{F}(\rr)\left(\frac{2^{2n+4}}{(\tau\theta)^{1+n/2}}+\frac{c_{3}'}{4}\right)\stackrel{\eqref{6.5}_{2},\eqref{6.9}}{\le}4(M+1),
\end{eqnarray}
and, using also the definition of $c_{3}'$ we get
\begin{eqnarray}\label{6.41.1}
V(\nu_{\iota_{1}}r_{1})&\stackrel{\eqref{tri.1}_{2},\eqref{final.2}_{2}}{\le}&2^{n}\mf{F}(\nu_{\iota_{1}-1}r_{1})+2\mf{V}_{2}\nonumber \\
&\stackrel{\eqref{final.2}_{1}}{\le}&\frac{2^{2n+4}}{(\tau\theta)^{1+n/2}}\mf{F}(\rr)+\frac{2^{n+1}c_{3}}{(\tau\theta)^{n/2}}\left(\mf{K}_{\mf{s}}^{\frac{p}{2(p-1)}}+\mf{V}_{2}^{(2-p)/p}\mf{H}_{\mf{s}}\right)+2\mf{V}_{2}\nonumber \\
&\stackrel{\eqref{6.3.1}}{\le}&\frac{9}{4}\mf{V}_{2}+\frac{2^{2n+4}}{(\tau\theta)^{1+n/2}}\mf{F}(\rr)+\left(\frac{2^{16nq}}{(\tau\theta)^{8n}}\right)^{\frac{p}{2(p-1)}}\mf{K}\left(\mathbf{I}^{f}_{1,m}(x_{0},\rr)\right)^{\frac{p}{2(p-1)}}\le 3\mf{V}_{2}.
\end{eqnarray}
Once \eqref{6.40}-\eqref{6.41.1} are available, given any $\varsigma\in \texttt{I}^{1}_{1}$ by \eqref{tri.1} we obtain
\begin{eqnarray}\label{6.42}
\left\{
\begin{array}{c}
\displaystyle
\ \mf{F}(\varsigma)\le \frac{2^{3n+10}}{(\tau\theta)^{1+n/2}}\left(\frac{\varsigma}{\rr}\right)^{\alpha_{0}}\mf{F}(\rr)+\frac{2^{2n+3}c_{3}}{(\tau\theta)^{n/2}}\left(\mf{K}_{\mf{s}}^{\frac{p}{2(p-1)}}+\mf{V}_{2}^{(2-p)/p}\mf{H}_{\mf{s}}\right)\\[8pt]\displaystyle 
\ V(\varsigma)\le 5(1+M),\qquad V(\varsigma)\le 4\mf{V}_{2}.
\end{array}
\right.
\end{eqnarray}
Finally, if $\varsigma\in \texttt{K}_{1}$, by \eqref{6.14}, \eqref{6.3} and \eqref{6.1} we can conclude that
\begin{eqnarray}\label{6.43}
\left\{
\begin{array}{c}
\displaystyle
\ \mf{F}(\varsigma)\le 2^{n+2}\left(\frac{H\mf{H}_{\mf{s}}}{\varepsilon_{0}}\right)^{\frac{p}{2(p-1)}} \\[8pt]\displaystyle 
\ V(\varsigma)\le 1,\qquad V(\varsigma)\le \mf{V}_{2}.
\end{array}
\right.
\end{eqnarray}
Combining estimates \eqref{final.2}, \eqref{6.42} and \eqref{6.43} we can conclude that for all $\varsigma \in \texttt{B}_{0}$ we obtain
\begin{eqnarray}\label{final.2.0}
\left\{
\begin{array}{c}
\displaystyle
\ \mf{F}(\varsigma)\le \frac{2^{4n+12}}{(\tau\theta)^{1+n/2}}\left(\frac{\varsigma}{\rr}\right)^{\alpha_{0}}\mf{F}(\rr)+\frac{2^{2n+4}c_{3}}{(\tau\theta)^{n/2}}\left(\frac{H}{\varepsilon_{0}}\right)^{\frac{p}{2(p-1)}}\left(\mf{K}_{\mf{s}}^{\frac{p}{2(p-1)}}+\mf{V}_{2}^{(2-p)/p}\mf{H}_{\mf{s}}\right)\\[8pt]\displaystyle
\ V(\varsigma)\le 5(1+M),\qquad V(\varsigma)\le 4\mf{V}_{2}.
\end{array}
\right.
\end{eqnarray}
\subsubsection*{Step 3.5: the general block $\texttt{B}_{d}$} Here we prove that in the regularity perspective, each block $\texttt{B}_{d}$ acts independently for all $d\in \N$ (the case $d=0$ is contained in \emph{Step 3.4}). Recalling the definition of $\texttt{B}_{d}$ given in \textbf{Step 3}, we immediately notice that if $\texttt{I}^{2}_{d}=\{\emptyset\}$, we can conclude with \eqref{6.16} and $\eqref{6.43}_{2}$. Next, define quantities
\begin{flalign*}
&\mf{V}_{1,d}:=c_{3,d}'\left(\mathbf{I}^{f}_{1,m}(x_{0},\rr)\right)^{\frac{p}{2(p-1)}},\qquad\quad \qquad \  c_{3,d}':=\left(\frac{2^{10n}H}{(\tau\theta)^{4n}\varepsilon_{0}}\right)^{\frac{p}{2(p-1)}}\\
&\mf{V}_{2,d}:=c_{3,d}''\left(\mf{K}\left(\mathbf{I}^{f}_{1,m}(x_{0},\rr)\right)\right)^{\frac{p}{2(p-1)}},\qquad \quad c_{3,d}'':=\left(\frac{2^{20nq}c_{3}H}{\mf{m}(\tau\theta)^{16nq}\varepsilon_{0}^{2}}\right)^{\frac{p}{2(p-1)}},
\end{flalign*}
 assume $\texttt{I}_{d}^{2}\not =\left\{\emptyset\right\}$ and observe that
\begin{eqnarray*}
\mf{C}(\nu_{\iota_{d}+\kk_{d}+1}r_{1})&\stackrel{\eqref{ls.43.1}}{\le}&2^{n+2}\mf{C}(\nu_{\iota_{d}+\kk_{d}}r_{1})\stackrel{\eqref{6.13}_{2}}{\le}2^{n+2}\left(\frac{H\mf{H}_{\mf{s}}}{\varepsilon_{0}}\right)^{\frac{p}{2(p-1)}},
\end{eqnarray*}
which, by means of \eqref{6.3}, \eqref{6.3.1} and \eqref{6.9}$_{1}$ yields:
\begin{flalign}
\left\{
\begin{array}{c}
\displaystyle
\ \mf{F}(\nu_{\iota_{d}+\kk_{d}+1}r_{1})\le 2^{n+2}\left(\frac{H\mf{H}_{\mf{s}}}{\varepsilon_{0}}\right)^{\frac{p}{2(p-1)}}<\frac{\varepsilon_{2}(\tau\theta)^{4npq}}{2^{8npq}}<\varepsilon_{2}\\[8pt] \displaystyle \ V(\nu_{\iota_{d}+\kk_{d}+1}r_{1})\le 1,\qquad V(\nu_{\iota_{d}+\kk_{d}+1}r_{1})\le \mf{V}_{1,d}.
\end{array}
\right.\label{6.44}
\end{flalign}
 Therefore, up to make the following substitutions:
\begin{flalign}\label{6.50}
r_{1}\leadsto r_{d}:=\nu_{\iota_{d}+\kk_{d}+1}r_{1},\qquad r_{2}\leadsto r_{2;d}:=\theta_{j_{1}}r_{d},\qquad r_{3}\leadsto r_{3;d}:=\tau_{j_{2}}r_{2;d},
\end{flalign}
we can replicate the whole procedure developed in \emph{Step 3.1}-\emph{Step 3.4}\footnote{Of course parameters $j_{1}$, $j_{2}$ appearing in the previous display do not necessarily coincide with those in \emph{Step 3.2} and in \emph{Step 3.4} respectively.} to obtain 
\begin{eqnarray*}
\left\{
\begin{array}{c}
\displaystyle
\ \mf{F}(\theta_{j+1}r_{d})\le \left(\frac{2^{4n}Hc_{3}}{\varepsilon_{0}}\right)^{\frac{p}{2(p-1)}}\mf{K}_{\mf{s}}^{\frac{p}{2(p-1)}}\\[8pt] \displaystyle \ V(\theta_{j}r_{d})\le 1,\qquad V(\theta_{j}r_{d})\le \mf{V}_{2;d},
\end{array}
\right.
\end{eqnarray*}
for all $j\in \N\cup\{0\}$ in case of stability of the singular regime or 
\begin{eqnarray}\label{0.1}
\left\{
\begin{array}{c}
\displaystyle
\ \mf{F}(\tau_{j+1}r_{2;d})\le \left(\frac{2^{4n}Hc_{3}}{\varepsilon_{0}}\right)^{\frac{p}{2(p-1)}}\mf{K}_{\mf{s}}^{\frac{p}{2(p-1)}}+2^{4}c_{3}\mf{V}_{2;d}^{(2-p)/p}\mf{H}_{\mf{s}}\\[8pt] \displaystyle \ V(\tau_{j}r_{2;d})\le 4,\qquad V(\tau_{j}r_{2;d})\le 4\mf{V}_{2;d},
\end{array}
\right.
\end{eqnarray}
for any $j\in \N\cup \{0\}$ if the nonsingular regime is stable. In any case, for all $\varsigma\in (0,r_{d}]$ it is
\begin{eqnarray}\label{0.0}
\left\{
\begin{array}{c}
\displaystyle
\ \mf{F}(\varsigma)\le \left(\frac{2^{6n}Hc_{3}}{\varepsilon_{0}(\tau\theta)^{n/2}}\right)^{\frac{p}{2(p-1)}}\mf{K}_{\mf{s}}^{\frac{p}{2(p-1)}}+\frac{2^{4}c_{3}}{(\tau\theta)^{n/2}}\mf{V}_{2;d}^{(2-p)/p}\mf{H}_{\mf{s}}\\[8pt] \displaystyle \ V(\varsigma)\le 5,\qquad V(\varsigma)\le 5\mf{V}_{2;d}.
\end{array}
\right.
\end{eqnarray}
On the other hand, keeping in mind that by \eqref{6.36} the nonsingular regime cannot end in $\texttt{I}^{2}_{d}$, we obtain that $\texttt{I}_{d}^{2}\subseteq (r_{3;d},r_{2;d}]\cup (r_{2;d},r_{d}]$, thus whenever $\varsigma\in \texttt{I}^{2}_{d}$ as in \emph{Step 3.4} we assure the validity of \eqref{0.0}.
Since by \eqref{tri.1}$_{1}$, \eqref{0.0}, \eqref{6.1}, \eqref{6.3} and \eqref{6.9}$_{1}$ it holds:
\begin{eqnarray}\label{6.46}
\left\{
\begin{array}{c}
\displaystyle
\ \mf{F}(\nu_{\iota_{d+1}}r_{1})\le  \left(\frac{2^{8n}Hc_{3}}{\varepsilon_{0}(\tau\theta)^{n/2}}\right)^{\frac{p}{2(p-1)}}\mf{K}_{\mf{s}}^{\frac{p}{2(p-1)}}+\frac{2^{n+6}c_{3}}{(\tau\theta)^{n/2}}\mf{V}_{2;d}^{(2-p)/p}\mf{H}_{\mf{s}}\\[8pt]\displaystyle 
\ V(\nu_{\iota_{d+1}}r_{1})\le 6,\qquad V(\nu_{\iota_{d+1}}r_{1})\le 6\mf{V}_{2;d}
\end{array}
\right.
\end{eqnarray}
for $\varsigma\in \texttt{I}^{1}_{d+1}$ we obtain thanks to \eqref{6.46}:
\begin{eqnarray}\label{6.51}
\left\{
\begin{array}{c}
\displaystyle
\ \mf{F}(\varsigma)\le  \left(\frac{2^{10n}Hc_{3}}{\varepsilon_{0}(\tau\theta)^{n/2}}\right)^{\frac{p}{2(p-1)}}\mf{K}_{\mf{s}}^{\frac{p}{2(p-1)}}+\frac{2^{2n+8}c_{3}}{(\tau\theta)^{n/2}}\mf{V}_{2;d}^{(2-p)/p}\mf{H}_{\mf{s}} \\[8pt]\displaystyle 
\ V(\varsigma)\le 7,\qquad V(\varsigma)\le 7\mf{V}_{2;d}
\end{array}
\right.
\end{eqnarray}
Finally, if $\varsigma\in \texttt{K}_{d+1}$ we directly have \eqref{6.14} and \eqref{6.43}$_{2}$. To summarize, we have just proven the validity of estimates \eqref{6.51}
for all $\varsigma\in \texttt{B}_{d}$, $d\in \N$.
\subsubsection*{Step 3.6: a finite number of finite iteration chains} Assume now that there is only a finite number, say $e_{*}\in \N$, of finite iteration chains, $\{\mathcal{C}^{\kk_{d}}_{\iota_{d}}\}_{d\in \{1,\cdots,e_{*}\}}$. Such chains determine blocks $\{\texttt{B}_{d}\}_{d\in \{0,\cdots,e_{*}-1\}}$, on which estimates \eqref{final.2.0} and \eqref{6.51} apply and, being only $e_{*}$ chains, by definition it follows that $\left\{j\in \N\colon j\ge \iota_{e_{*}}+\kk_{e_{*}}+1\right\}\subset \mathcal{J}_{0}$. Then, for all $\varsigma \in (0,r_{1}]\setminus \bigcup_{d\in \{0,\cdots,e_{*}-1\}}\texttt{B}_{d}\equiv(0,\nu_{\iota_{e_{*}}+\kk_{e_{*}}+1}r_{1}]$ there is $j\ge \iota_{e_{*}}+\kk_{e_{*}}+1$ such that $\nu_{j_{\varsigma}+1}r_{1}<\varsigma\le \nu_{j_{\varsigma}}r_{1}$ and
\begin{eqnarray}\label{6.47}
\mf{C}(\varsigma)\stackrel{\eqref{6.15}}{\ge}\frac{1}{2^{2+n}}\left(\frac{H\mf{H}_{\mf{s}}}{\varepsilon_{0}}\right)^{\frac{p}{2(p-1)}}\qquad \mbox{for all} \ \ \varsigma\in (0,\nu_{\iota_{e_{*}}+\kk_{e_{*}}+1}r_{1}].
\end{eqnarray}
Furthermore, we also have that
\begin{eqnarray*}
\mf{C}(\nu_{\iota_{e_{*}}+\kk_{e_{*}}+1}r_{1})\stackrel{\eqref{ls.43.1}}{\le}2^{2+n}\mf{C}(\nu_{\iota_{e_{*}}+\kk_{e_{*}}}r_{1})\stackrel{\eqref{6.13}_{2}}{\le}2^{2+n}\left(\frac{H\mf{H}_{\mf{s}}}{\varepsilon_{0}}\right)^{\frac{p}{2(p-1)}}\stackrel{\eqref{6.3},\eqref{6.9}_{1}}{<}\frac{\varepsilon_{2}(\tau\theta)^{4npq}}{2^{8npq}},
\end{eqnarray*}
which means that \eqref{6.44} holds and, via \eqref{6.47} we also see that the same argument leading to \eqref{6.36} works in this case as well and renders that the nonsingular regime is stable over the whole $(0,\nu_{\iota_{e_{*}}+\kk_{e_{*}}+1}r_{1}]$. With these last informations at hand we gain that \eqref{0.1} (with $\nu_{\iota_{e_{*}}+\kk_{e_{*}}+1}r_{1}$ instead of $r_{2;d}$) holds, and as a consequence \eqref{0.0} is satisfied for all $\varsigma\in (0,\nu_{\iota_{e_{*}}+\kk_{e_{*}}+1}r_{1}]$. To summarize, we have just proven that \eqref{final.2.0} or \eqref{6.51} hold for all $\varsigma\in (0,r_{1}]$.
\subsubsection*{Step 3.7: an infinite iteration chain} We describe the presence of an infinite iteration chains by introducing a number $e_{*}\in \N$ - assume $e_{*}\ge 2$ for the moment - and corresponding sets of integers $\{\iota_{1},\cdots,\iota_{e_{*}}\}\subset \N$, $\{\kk_{1},\cdots,\kk_{e_{*}-1}\}\subset \N$ and $\kk_{e_{*}}=\infty$, determining $e_{*}-1$ finite iteration chains $\{\mathcal{C}^{\kk_{d}}_{\iota_{d}}\}_{d\in \{1,\cdots,e_{*}-1\}}$ and one infinite iteration chains $\mathcal{C}^{\infty}_{\iota_{e_{*}}}$ that must be unique by maximality. On each of blocks $\{\texttt{B}_{d}\}_{d\in \{0,\cdots,e_{*}-2\}}$ determined by chains $\{\mathcal{C}^{\kk_{d}}_{\iota_{d}}\}_{d\in \{1,\cdots,e_{*}-1\}}$ estimates \eqref{final.2.0} or \eqref{6.51} hold true. Concerning the last chain $\mathcal{C}^{\infty}_{\iota_{e_{*}}}$, it generates the last block $\texttt{B}_{e_{*}-1}=\texttt{I}^{2}_{e_{*}-1}\cup \texttt{I}^{1}_{e_{*}}\cup \texttt{K}_{e_{*}}$ with $\texttt{K}_{e_{*}}=(0,\nu_{\iota_{e_{*}}+1}r_{1}]$. On intervals $\texttt{I}^{2}_{e_{*}-1}$-$\texttt{I}^{1}_{e_{*}}$ \eqref{6.46}-\eqref{6.51} are verified, while on $\texttt{K}_{e_{*}}$ we can simply conclude by means of \eqref{6.43}. On the other hand, if $e_{*}=1$, there is only one block $\texttt{B}_{0}=\texttt{I}_{0}\cup \texttt{I}^{1}_{1}\cup \texttt{K}_{1}\equiv (0,r_{1}]$, on which \eqref{final.2} and \eqref{6.42}-\eqref{6.43} holds, therefore we can conclude with \eqref{final.2.0} also in this case.
\subsubsection*{Step 3.8: occurrence \emph{(}ii.\emph{)}} Since $\mathcal{J}_{0}\equiv \N\cup \{0\}$, inequality \eqref{6.15} is satisfied by all $\varsigma\in (0,r_{1}]$, so the validity of \eqref{6.36} is now extended to the full interval $(0,r_{1}]$ and this guarantees the stability of the nonsingular regime. Therefore we can proceed as done in \emph{Step 3.1}-\emph{Step 3.3} to get \eqref{final.2}.
\subsection*{Step 4: small composite excess at the first scale.} This time, the set $\mathcal{J}_{0}\subseteq \N\cup\{0\}$ is defined as
\begin{eqnarray*}
\mathcal{J}_{0}:=\left\{j\in \N\cup\{0\}\colon \mf{C}(\nu_{j}r_{1})\le \left(\frac{H\mf{H}_{\mf{s}}}{\varepsilon_{0}}\right)^{\frac{p}{2(p-1)}}\right\}\stackrel{\eqref{6.11}_{2}}{\not =}\left\{\emptyset\right\}. 
\end{eqnarray*}
We immediately notice that if $\mathcal{J}_{0}\equiv \N\cup\{0\}$, then \eqref{6.14} holds for all $\varsigma\in (0,r_{1}]$ so this, $\eqref{6.43}$ and \eqref{6.35} give the result. We then look at the case in which there exist infinitely many finite iteration chains $\{\mathcal{C}^{\kk_{d}}_{\iota_{d}}\}_{d\in \N}$ with $\{\iota_{d}\}_{d\in \N}$, $\{\kk_{d}\}_{d\in \N}$ as in \eqref{6.12}, determining intervals
\begin{flalign*}
&\texttt{I}_{0}:=(\nu_{\iota_{1}+1}r_{1},r_{1}],\qquad\qquad \qquad  \qquad \texttt{K}_{d}^{1}:=(\nu_{\iota_{d}+\kk_{d}}r_{1},\nu_{\iota_{d}+1}r_{1}]\nonumber \\
&\texttt{K}_{d}^{2}:=(\nu_{\iota_{d}+\kk_{d}+1}r_{1},\nu_{\iota_{d}+\kk_{d}}r_{1}]\qquad\qquad  \texttt{I}_{d}^{2}:=(\nu_{\iota_{d+1}+1}r_{1},\nu_{\iota_{d}+\kk_{d}+1}r_{1}]
\end{flalign*}
and blocks $\texttt{B}_{0}:=\texttt{I}_{0}\cup \texttt{K}^{1}_{1}\cup\texttt{K}^{2}_{1}$, $\texttt{B}_{d}:=\texttt{I}^{2}_{d}\cup \texttt{K}^{1}_{d+1}\cup \texttt{K}^{2}_{d+1}$ such that $(0,r_{1}]\equiv \bigcup_{d\in \N\cup\{0\}}\texttt{B}_{d}$ by \eqref{6.12}. As done in \textbf{Step 3}, we readily observe that
\begin{flalign}\label{j.0}
\{0,\cdots,\iota_{1}\}, \ \{\iota_{d}+\kk_{d}+1,\cdots,\iota_{d+1}\}\subset \mathcal{J}_{0}\qquad\mbox{and}\qquad  \{\iota_{d}+1,\cdots,\iota_{d}+\kk_{d}\}\subset \mathcal{C}^{\kk_{d}}_{\iota_{d}},
\end{flalign}
therefore we have
\begin{eqnarray}\label{6.48}
\left\{
\begin{array}{c}
\displaystyle
\ \varsigma\in \texttt{I}_{0}\ \ \mbox{or} \ \ \varsigma\in \texttt{I}^{2}_{d} \ \Longrightarrow \ \mf{C}(\varsigma)\le 2^{2+n}\left(\frac{H\mf{H}_{\mf{s}}}{\varepsilon_{0}}\right)^{\frac{p}{2(p-1)}} \\[8pt]\displaystyle 
\ \varsigma\in \texttt{K}_{d}^{1} \ \Longrightarrow \ \mf{C}(\varsigma)> \frac{1}{2^{2+n}}\left(\frac{H\mf{H}_{\mf{s}}}{\varepsilon_{0}}\right)^{\frac{p}{2(p-1)}}.
\end{array}
\right.
\end{eqnarray}
Notice that $\texttt{I}^{2}_{d}$ cannot be empty otherwise $\iota_{d+1}=\kk_{d}+\iota_{d}$ and this is not possible by means of \eqref{j.0}, while if $\texttt{K}^{1}_{d+1}=\left\{\emptyset\right\}$ (i.e. if $\kk_{d+1}=1$), we can exploit \eqref{6.48}$_{1}$, \eqref{ls.43.1} and that $\iota_{d}\in \mathcal{J}_{0}$ for all $d\in \N$ to derive
\begin{eqnarray}\label{6.52}
\left.
\begin{array}{c}
\displaystyle
\ \varsigma\in \texttt{B}_{0} \ \ \mbox{with} \ \ \texttt{K}_{1}^{1}=\left\{\emptyset\right\}\\[8pt]\displaystyle 
\ \varsigma\in \texttt{B}_{d}  \ \ \mbox{with} \ \ \texttt{K}_{d+1}^{1}=\left\{\emptyset\right\}, \ \ d\in \N
\end{array}
\right. \quad \Longrightarrow \quad \mf{C}(\varsigma)\le 2^{2n+4}\left(\frac{H\mf{H}_{\mf{s}}}{\varepsilon_{0}}\right)^{\frac{p}{2(p-1)}};
\end{eqnarray}
in other words, whenever $\texttt{K}^{1}_{d+1}=\left\{\emptyset\right\}$ there is nothing to prove on the related block $\texttt{B}_{d}$ and \eqref{6.52} and \eqref{6.43}$_{2}$ immediately follow. Now, given a general block $\texttt{B}_{d}$ with $\texttt{K}^{1}_{d+1}\not =\left\{\emptyset\right\}$ and $d\in\N\cup\{0\}$, if $\varsigma\in \texttt{I}_{0}$ or $\varsigma\in \texttt{I}^{2}_{d}$ estimate \eqref{6.48} holds. Next, observe that
\begin{eqnarray}\label{6.53}
\mf{C}(\nu_{\iota_{d+1}+1}r_{1})&\stackrel{\eqref{ls.43.1}}{\le}&2^{2+n}\mf{C}(\nu_{\iota_{d+1}}r_{1})\stackrel{\eqref{j.0}}{\le}2^{2+n}\left(\frac{H\mf{H}_{\mf{s}}}{\varepsilon_{0}}\right)^{\frac{p}{2(p-1)}}\nonumber \\
&&\stackrel{\eqref{6.3},\eqref{6.9}_{1}}{\Longrightarrow} \mf{F}(\nu_{\iota_{d+1}+1}r_{1})<\frac{\varepsilon_{2}(\tau\theta)^{4npq}}{2^{8npq}},\qquad V(\nu_{\iota_{d+1}+1}r_{1})\le 1,\qquad V(\nu_{\iota_{d+1}+1}r_{1})\le \mf{V}_{1,d},
\end{eqnarray}
therefore, setting this time $r_{d}:=\nu_{\iota_{d+1}+1}r_{1}$ we can plug in the substitutions in \eqref{6.50} and apply the content of \emph{Step 3.5} (with $\texttt{I}^{2}_{d}$ replaced by $\texttt{K}_{d+1}^{1}$) to get \eqref{0.0} and, recalling that $\texttt{K}_{d+1}^{2}$ differs from $\texttt{K}^{1}_{d+1}$ only by one scale, we recover also \eqref{6.51}. Merging \eqref{6.48}, \eqref{6.52}, \eqref{0.0} and \eqref{6.51} we can conclude that \eqref{6.51} holds for all $\varsigma\in \texttt{B}_{d}$, $d\in \N\cup \{0\}$.
\subsubsection*{Step 4.1: a finite number of finite iteration chains} Let us assume now that there is a finite number, say $e_{*}\in \N$ of finite iteration chains $\{\mathcal{C}^{\kk_{d}}_{\iota_{d}}\}_{d\in \{1,\cdots,e_{*}\}}$ and corresponding blocks $\{\texttt{B}_{d}\}_{d\in \{0,\cdots,e_{*}-1\}}$. On every block $\texttt{B}_{d}$, $d\in \N\cup\{0\}$, estimates \eqref{6.51} apply. Notice that $(0,r_{1}]\setminus \bigcup_{d=0}^{e_{*}-1}\texttt{B}_{d}=(0,\nu_{\iota_{e_{*}}+\kk_{e_{*}}+1}r_{1}]$ and, since the last finite iteration chain is $\mathcal{C}^{\kk_{e_{*}}}_{\iota_{e_{*}}}$, it follows that $\left\{j\in \N\colon j\ge \iota_{e_{*}}+\kk_{e_{*}}+1\right\}\subset \mathcal{J}_{0}$, therefore for all $\varsigma\in (0,\nu_{\iota_{e_{*}}+\kk_{e_{*}}+1}r_{1}]$ the bound in \eqref{6.48} is verified, so we confirm again the validity of \eqref{6.51}.

\subsubsection*{Step 4.2: an infinite iteration chain} In this case, for $e_{*}\in \N$ (assume for the moment that $e_{*}\ge 2$) we can find finite set of integers $\{\iota_{1},\cdots,\iota_{e_{*}}\}\subset \N$, $\{\kk_{1},\cdots,\kk_{e_{*}-1}\}\subset \N$ and $\kk_{e_{*}}=\infty$, thus determining $e_{*}-1$ finite iteration chains $\{\mathcal{C}^{\kk_{d}}_{\iota_{d}}\}_{d\in \{1,\cdots,e_{*}-1\}}$ and one infinite iteration chain $\mathcal{C}_{\iota_{e_{*}}}^{\infty}$, that is unique by maximality. Chains $\{\mathcal{C}^{\kk_{d}}_{\iota_{d}}\}_{d\in \{1,\cdots,e_{*}-1\}}$ determine blocks $\{\texttt{B}_{d}\}_{d\in \{1,\cdots,e_{*}-2\}}$ on which the content of \textbf{Step 4} applies and \eqref{6.51} holds true, while the presence of $\mathcal{C}_{\iota_{e_{*}}}^{\infty}$ results into $\texttt{B}_{e_{*}-1}=\texttt{I}^{2}_{e_{*}-1}\cup(0,\nu_{\iota_{e_{*}}+1}r_{1}]$. If $\varsigma\in \texttt{I}^{2}_{e_{*}-1}$ we directly have \eqref{6.48} which in particular implies the validity of \eqref{6.53} with $d=e_{*}-1$, so we can reproduce the content of \emph{Step 3.5} with $r_{d}=\nu_{\iota_{e_{*}}+1}r_{1}$ and eventually arrive at \eqref{6.51}. Finally if $e_{*}=1$, there is only the infinite iteration chain, thus $(0,r_{1}]=\texttt{I}_{0}\cup (0,\nu_{\iota_{e_{*}}+1}r_{1}]$. On $\texttt{I}_{0}$ the bound in \eqref{6.48} is in force, this in turn yields \eqref{6.53} so, proceeding as in \emph{Step 3.5} we obtain \eqref{6.51}.
\subsection*{Step 5: conclusions} Collecting estimates \eqref{final.1}, \eqref{final.2}, \eqref{final.2.0},  and \eqref{6.51}, and setting
\begin{flalign*}
c_{5}:=\frac{2^{4n+12}}{(\tau\theta)^{1+n/2}},\qquad \qquad \quad c_{6}:=\left(\frac{2^{40nq}H^{2}c_{3}^{2}c_{3}'}{\varepsilon_{0}^{4}\mf{m}(\tau\theta)^{32nq}}\right)^{\frac{p}{2(p-1)}}
\end{flalign*}
we obtain \eqref{6.6}-\eqref{6.7} and the proof is complete.
\end{proof}
For later use, let us record a couple of consequences of Theorem \ref{t.ex}, that come along the lines of \cite[Proposition 5.1 and Corollary 5.1]{deqc}.
\begin{corollary}\label{cor.ex}
Assume \eqref{assf}-\eqref{sqc}, \eqref{f}, 
let $u\in W^{1,p}(\Omega,\mathbb{R}^{N})$ be a local minimizer of \eqref{exfun}, $x_{0}\in \mathcal{R}_{u}$ be a point, $M\equiv M(x_{0})>0$ be the constant in \eqref{ru.0}, $\hat{\varepsilon}\equiv \hat{\varepsilon}(\textnormal{\texttt{data}},M)$ and $\hat{\rr}\equiv \hat{\rr}(\textnormal{\texttt{data}},M,f(\cdot))$ be as in $\eqref{6.9}_{1}$ and $\eqref{6.1}$ respectively.
\begin{itemize}
    \item If
    \begin{eqnarray}\label{6.54}
\mathbf{I}^{f}_{1,m}(x,\sigma)\to 0\quad \mbox{locally uniformly in} \ \ x\in \Omega,
\end{eqnarray}
then 
if $\bar{\varepsilon}\equiv \hat{\varepsilon}$ and $\bar{\rr}\equiv \hat{\rr}$ in \eqref{ru.0}, there is an open neighborhood $B(x_{0})\subset \mathcal{R}_{u}$ and a positive radius $\rr_{x_{0}}\equiv \rr_{x_{0}}(\textnormal{\texttt{data}},M,f(\cdot))\in (0,\hat{\rr}]$ such that
\begin{eqnarray}\label{6.55}
\left\{
\begin{array}{c}
\displaystyle
\ \snr{(V_{p}(Du))_{B_{\varsigma}(x)}}<8(1+M)\\[8pt]\displaystyle 
\ \snr{(V_{p}(Du))_{B_{\sigma}(x)}}\le c_{8}\left(\mf{C}(x;\varsigma)+\mf{K}\left(\mathbf{I}^{f}_{1,m}(x,\varsigma)\right)^{\frac{p}{2(p-1)}}\right)
\end{array}
\right.
\end{eqnarray}
and
\begin{eqnarray}\label{6.56}
\mf{F}(u;B_{\sigma}(x))&\le& c_{7}\left(\frac{\varsigma}{\varsigma}\right)^{\alpha_{0}}\mf{F}(u;B_{\varsigma}(x))+c_{8}\sup_{s\le \varsigma/4}\mf{K}\left[\left(s^{m}\mint_{B_{s}(x)}\snr{f}^{m} \ \dx\right)^{1/m}\right]^{\frac{p}{2(p-1)}}\nonumber \\
&&+c_{8}\left(\mf{C}(x;\varsigma)+\mf{K}\left(\mathbf{I}^{f}_{1,m}(x,\varsigma)\right)^{\frac{p}{2(p-1)}}\right)^{(2-p)/p}\sup_{s\le \varsigma/4}\left(s^{m}\mint_{B_{s}(x)}\snr{f}^{m} \ \dx\right)^{1/m},
\end{eqnarray}
hold for all $x\in B(x_{0})$, $0<\sigma\le\varsigma\le \rr_{x_{0}}$, where $c_{7}:=c_{5}(2^{12}c_{5})^{1+\frac{n}{2\alpha_{0}}}$ and $c_{8}:=c_{6}(2^{16}c_{5})^{2+\frac{n}{2\alpha_{0}}}$, $c_{7},c_{8}\equiv c_{7},c_{8}(\textnormal{\texttt{data}},M)$. 
\item If \eqref{6.0} is in force instead of \eqref{6.54}, then the following "restricted" versions of \eqref{6.55}-\eqref{6.56} hold:
\begin{eqnarray}\label{6.55.1}
\left\{
\begin{array}{c}
\displaystyle
\ \snr{(V_{p}(Du))_{B_{\varsigma}(x_{0})}}<8(1+M)\\[8pt]\displaystyle 
\ \snr{(V_{p}(Du))_{B_{\sigma}(x_{0})}}\le c_{8}\left(\mf{C}(x_{0};\varsigma)+\mf{K}\left(\mathbf{I}^{f}_{1,m}(x_{0},\varsigma)\right)^{\frac{p}{2(p-1)}}\right)
\end{array}
\right.
\end{eqnarray}
and
\begin{eqnarray}\label{6.56.1}
\mf{F}(u;B_{\sigma}(x_{0}))&\le& c_{7}\left(\frac{\sigma}{\varsigma}\right)^{\alpha_{0}}\mf{F}(u;B_{\varsigma}(x_{0}))+c_{8}\sup_{s\le \varsigma/4}\mf{K}\left[\left(s^{m}\mint_{B_{s}(x_{0})}\snr{f}^{m} \ \dx\right)^{1/m}\right]^{\frac{p}{2(p-1)}}\nonumber \\
&&+c_{8}\left(\mf{C}(x_{0};\varsigma)+\mf{K}\left(\mathbf{I}^{f}_{1,m}(x_{0},\varsigma)\right)^{\frac{p}{2(p-1)}}\right)^{(2-p)/p}\sup_{s\le \varsigma/4}\left(s^{m}\mint_{B_{s}(x_{0})}\snr{f}^{m} \ \dx\right)^{1/m},
\end{eqnarray}
for all balls $B_{\sigma}(x_{0})\subseteq B_{\varsigma}(x_{0})\subseteq B_{\rr_{x_{0}}}(x_{0})$.
\item With \eqref{6.0} in force, if 
in \eqref{ru.0} it is $\bar{\varepsilon}\equiv \hat{\varepsilon}$, $\bar{\rr}\equiv \hat{\rr}$, then
\begin{eqnarray}\label{6.57}
\sup_{\sigma\le \rr_{x_{0}}}\mf{F}(u;B_{\rr}(x_{0}))\le c_{9}\hat{\varepsilon},
\end{eqnarray}
with $c_{9}:=2^{8}(c_{7}+c_{8})$, $c_{9}\equiv c_{9}(\textnormal{\texttt{data}},M)$.
\end{itemize}
\end{corollary}
We conclude this section with an almost everywhere $VMO$ result. To do so, we need some preliminaries. Assume \eqref{6.54},
let $x_{0}\in \mathcal{R}_{u}$ be any point, $M\equiv M(x_{0})$ be the positive constant in \eqref{ru.0}. With $\bar{\varepsilon}$, $\bar{\rr}$ still to be determined, we introduce constants:
\begin{flalign}\label{hhh}
H_{1}:=\max\left\{2^{24n}Hc_{9},\frac{2^{24n}H}{\varepsilon_{1}(\tau\theta)^{6n}}\right\},\qquad\qquad  H_{2}:=\left(\frac{2^{36nq}c_{8}H}{\varepsilon_{1}(\tau\theta)^{20nq}}\right)^{\frac{p}{2(p-1)}},\end{flalign}
and fix
\begin{eqnarray}\label{6.58}
\varepsilon_{*}:=\frac{\hat{\varepsilon}}{2^{20}c_{9}},
\end{eqnarray}
where $\hat{\varepsilon}\equiv \hat{\varepsilon}(\textnormal{\texttt{data}},M)$, $H\equiv H(\textnormal{\texttt{data}},M)$ are defined in \eqref{6.9}. Notice that \eqref{6.54} implies that
\begin{flalign}\label{f.p.m}
\mathbf{I}^{f}_{1,m}(\cdot,1)\in L^{\infty}_{\loc}(\Omega)\qquad \mbox{and}\qquad \left(s^{m}\mint_{B_{s}(x)}\snr{f}^{m} \ \dx\right)^{1/m}\to 0\quad \mbox{locally uniformly in} \ \ x\in \Omega.
\end{flalign} 
By means of \eqref{6.54}, we determine a threshold radius $\rr_{*}\equiv \rr_{*}(\textnormal{\texttt{data}},M,f(\cdot))\in (0,\hat{\rr}]$ such that
\begin{flalign}\label{6.59}
c_{10}\mf{K}\left(\mathbf{I}^{f}_{1,m}(x,s)\right)^{\frac{p}{2(p-1)}}+c_{10}M^{(2-p)/p}\mathbf{I}^{f}_{1,m}(x,s)< \varepsilon_{*},\qquad\qquad  c_{10}:=\left(\frac{2^{24npq}c_{9}H_{2}}{\varepsilon_{0}(\tau\theta)^{20npq}\mf{m}}\right)^{\frac{p}{2(p-1)}}
\end{flalign}
for all $s\le \rr_{*}$, $x\in B_{d_{x_{0}}}(x_{0})$, which implies via \eqref{6.3.1} that
\begin{flalign}\label{6.59.1}
\left\{
\begin{array}{c}
\displaystyle
\ c_{8}\mf{K}\left[\sup_{s\le \rr_{*}/4}\left(s^{m}\mint_{B_{s}(x)}\snr{f}^{m} \ \dx\right)^{1/m}\right]^{\frac{p}{2(p-1)}}<\frac{1}{2^{10}}\\[8pt]\displaystyle
\ c_{8}\left(2+M\right)^{\frac{2-p}{p}}\left[\sup_{s\le \rr_{*}/4}\left(s^{m}\mint_{B_{s}(x)}\snr{f}^{m} \ \dx\right)^{1/m}\right]<\frac{1}{2^{10}},
\end{array}
\right.
\end{flalign}
for all $x\in B_{d_{x_{0}}}(x_{0})$, and recalling the definition of $c_{9}$ given in Corollary \ref{cor.ex}, we have $c_{10}>c_{4}$ and by \eqref{hhh} that it is $H_{2}>H$, so the choice made in \eqref{6.59} immediately implies the validity of \eqref{6.1} on $B_{d_{x_{0}}}(x_{0})$. Now we are ready to prove:
\begin{proposition}\label{vmo.t}
Under assumptions \eqref{assf}-\eqref{sqc}, \eqref{f} and \eqref{6.54}, let $u\in W^{1,p}(\Omega,\mathbb{R}^{N})$ be a local minimizer of \eqref{exfun}. There exists an open set $\Omega_{u}\subset \Omega$ of full $n$-dimensional Lebesgue measure such that $Du\in VMO_{\loc}(\Omega_{u},\mathbb{R}^{N\times n})$ which can be characterized as
\begin{flalign*}
\Omega_{u}:=\left\{x_{0}\in \Omega\colon \exists M\equiv M(x_{0})>0\colon \snr{(V_{p}(Du))_{B_{\rr}(x_{0})}}<M \ \ \mbox{and}\ \ \mf{F}(u;B_{\rr}(x_{0}))<\varepsilon_{*} \ \  \mbox{for some} \ \ \rr\in (0,\rr_{*}]\right\},
\end{flalign*}
with $\varepsilon_{*}\equiv \varepsilon_{*}(\textnormal{\texttt{data}},M)$ as in \eqref{6.59} and $\rr_{*}\equiv \rr_{*}(\textnormal{\texttt{data}},M,f(\cdot))$ defined by \eqref{6.59}-\eqref{6.59.1}. In particular, for all $x_{0}\in \Omega_{u}$ there is an open neighborhood $B(x_{0})\subset \Omega_{u}$ such that
\begin{eqnarray}\label{6.63}
\lim_{\rr\to 0}\mf{F}(u;B_{\rr}(x))=0\qquad \mbox{uniformly for all} \ \ x\in B(x_{0}).
\end{eqnarray}

\end{proposition}

\begin{proof}
In the light of the discussion at the beginning of Section \ref{morsec}, the ideal candidate for $\Omega_{u}$ is set $\mathcal{R}_{u}$ as in \eqref{ru.0} with $\bar{\varepsilon}\equiv \varepsilon_{*}$ and $\bar{\rr}\equiv \rr_{*}$: in fact, we already know that it is an open set of full $n$-dimensional Lebesgue measure so we only need to prove the $VMO$-result. We take $x_{0}\in \mathcal{R}_{u}$ with $\bar{\varepsilon}\equiv \varepsilon_{*}$, $\bar{\rr}\equiv \rr_{*}$ in \eqref{ru.0} and observe that \eqref{6.54} allows applying the first part of Corollary \ref{cor.ex}, so there exists an open neighborhood $B(x_{0})\subset \mathcal{R}_{u}$ and a positive radius $\rr_{x_{0}}\equiv \rr_{x_{0}}(\textnormal{\texttt{data}},M,f(\cdot))$ such that \eqref{6.55}-\eqref{6.56} are verified for all $x\in B(x_{0})$ and any $0<\sigma\le \varsigma\le \rr_{x_{0}}$. Of course, we can always assume that $B(x_{0})\subset B_{d_{x_{0}}}(x_{0})$. Fixed an arbitrary $r\in (0,1)$, by \eqref{f.p.m}$_{2}$ we can find a radius $\rr''\equiv \rr''(\textnormal{\texttt{data}},M,f(\cdot))\in (0,\rr_{x_{0}}]$ satisfying
\begin{flalign}\label{6.60}
&c_{8}\sup_{s\le \rr''}\mf{K}\left[\left(\mint_{B_{s}(x)}\snr{f}^{m} \ \dx\right)^{1/m}\right]^{\frac{p}{2(p-1)}}+c_{8}\left(2+M\right)^{(2-p)/p}\sup_{s\le \rr''}\left(\mint_{B_{s}(x)}\snr{f}^{m} \ \dx\right)^{1/m}\le \frac{r}{2^{4}}.
\end{flalign}
Moreover, via \eqref{6.56} with $\sigma\equiv \rr''$ and $\varsigma\equiv \rr_{x_{0}}$, \eqref{6.58}, \eqref{6.59}, \eqref{6.59.1} and \eqref{70} with $\bar{\varepsilon}\equiv \varepsilon_{*}$, $\bar{\rr}\equiv \rr_{*}$ we obtain
\begin{eqnarray}\label{6.61}
\mf{F}(u;B_{\rr''}(x))&\le&c_{7}\left(\frac{\rr''}{\rr_{x_{0}}}\right)^{\alpha_{0}}\mf{F}(u;B_{\rr_{x_{0}}}(x))+c_{8}\sup_{s\le \rr_{x_{0}}/4}\mf{K}\left[\left(s^{m}\mint_{B_{s}(x)}\snr{f}^{m} \ \dx\right)^{1/m}\right]^{\frac{p}{2(p-1)}}\nonumber \\
&&+c_{8}\left(\mf{C}(x;\rr_{x_{0}})+\mf{K}\left(\mathbf{I}^{f}_{1,m}(x,\rr_{x_{0}})\right)^{\frac{p}{2(p-1)}}\right)^{(2-p)/p}\sup_{s\le \rr_{x_{0}}/4}\left(s^{m}\mint_{B_{s}(x)}\snr{f}^{m} \ \dx\right)^{1/m}\nonumber \\
&\le&c_{7}\varepsilon_{*}+\frac{1}{2^{10}}\le 1
\end{eqnarray}
Finally we pick $\sigma_{r}\equiv \sigma_{r}(\textnormal{\texttt{data}},M,f(\cdot))\in (0,\rr'']$ small enough that 
\eqn{6.62}
$$c_{9}(\sigma_{r}/\rr'')^{\alpha_{0}}\le r/2.$$ Plugging \eqref{6.60}-\eqref{6.62} in \eqref{6.56} with $\sigma\equiv \sigma_{r}$ and $\varsigma\equiv \rr''$ we obtain that 
$$
\sigma\le \sigma_{r} \ \Longrightarrow \ \mf{F}(u;B_{\sigma}(x))\le r\qquad \mbox{for all} \ \ x\in B(x_{0}).
$$
The arbitrariety of $r$, \eqref{6.55} and a standard covering argument eventually lead to \eqref{6.63} and the proof is complete.
\end{proof}
\begin{remark}\label{rem}
\emph{Let us list some relevant observations.
\begin{itemize}
    \item Replacing \eqref{6.54} with \eqref{6.0} in Proposition \ref{vmo.t}, we obtain that whenever $x_{0}\in \Omega$ verifies the conditions in \eqref{ru.0} with $\bar{\varepsilon}\equiv \varepsilon_{*}$ and $\bar{\rr}\equiv \rr_{*}$ it holds that
    \eqn{6.64}
    $$
    \lim_{\rr\to 0}\mf{F}(u;B_{\rr}(x_{0}))=0.
    $$
    \item Corollary \ref{cor.ex} and Proposition \ref{vmo.t} remain valid if $M$ is replaced by $8(1+M)$, without affecting the magnitude of the bounding constants appearing in the various estimates as they are all derived in correspondence of larger values than $M$.
    \item Corollary \ref{cor.ex} guarantees in particular that once \eqref{6.5} - with $\hat{\varepsilon}$, $\hat{\rr}$ as in \eqref{6.9}$_{1}$ and \eqref{6.1} respectively - is verified for a certain $\rr\in (0,\hat{\rr}]$, then the Morrey type decay estimates \eqref{6.56} and \eqref{6.56.1} for the excess functional $\mf{F}(\cdot)$ hold at all scales smaller than $\rr$. This will allow us to work on all scales smaller that $\rr$.  
\end{itemize}}
\end{remark}
\section{Borderline gradient continuity}\label{7sec}
This final section is devoted to the proof of the partial gradient continuity for minima of \eqref{exfun}. Let $x_{0}\in \mathcal{R}_{u}$ be any point, $M\equiv M(x_{0})>0$, $\bar{\varepsilon}$, $\bar{\rr}$ be the parameters appearing in \eqref{ru.0}, still to be fixed as functions of $(\textnormal{\texttt{data}},M)$ and $(\textnormal{\texttt{data}},M,f(\cdot))$ respectively. We assume the validity of \eqref{6.0} at $x_{0}$, define the smallness threshold
\begin{eqnarray}\label{7.1}
\varepsilon':=\frac{\varepsilon_{*}}{2^{8}c_{9}\max\{H_{1},H_{2}\}} \ \Longrightarrow \ \varepsilon'\equiv \varepsilon'(\textnormal{\texttt{data}},M)
\end{eqnarray}
and determine the radius $\rr'\equiv \rr'(\textnormal{\texttt{data}},M,f(\cdot))\in (0,\rr_{*}]$ so small that
\begin{flalign}\label{7.2}
c_{11}\mf{K}\left(\mathbf{I}^{f}_{1,m}(x_{0},s)\right)^{\frac{p}{2(p-1)}}+c_{11}M^{(2-p)/p}\mathbf{I}^{f}_{1,m}(x_{0},s)< \varepsilon',\qquad\quad  c_{11}:=\left(\frac{c_{10}2^{32npq}\max\{H_{1},H_{2}\}}{(\tau\theta)^{16npq}}\right)^{\frac{p}{2(p-1)}}.
\end{flalign}
for all $s\in (0,\rr']$, which yields
\begin{flalign}\label{7.3}
&\sup_{\sigma\le s/4}\mf{K}\left[\left(\sigma^{m}\mint_{B_{\sigma}(x_{0})}\snr{f}^{m} \dx\right)^{1/m}\right]^{\frac{p}{2(p-1)}}\nonumber \\
&\qquad \qquad \qquad +M^{(2-p)/p}\sup_{\sigma\le s/4}\left(\sigma^{m}\mint_{B_{\sigma}(x_{0})}\snr{f}^{m} \dx\right)^{1/m}<\varepsilon'\left(\frac{(\tau\theta)^{12npq}}{2^{24npq}c_{10}\max\{H_{1},H_{2}\}}\right)^{\frac{p}{2(p-1)}},
\end{flalign}
cf. Section \ref{morsec}. Of course there is no loss of generality in assuming that $0<\rr'\le \rr_{*}\le \hat{\rr}$, so setting $\bar{\varepsilon}\equiv \varepsilon'$ and $\bar{\rr}\equiv \rr'$ in \eqref{ru.0} we can find $\rr\in (0,\rr']$ such that
\begin{eqnarray}\label{7.4}
\mf{F}(u;B_{\rr}(x_{0}))<\varepsilon'\qquad \mbox{and}\qquad \snr{(V_{p}(Du))_{B_{\rr}(x_{0})}}<M.
\end{eqnarray}
Thanks to the choices made in \eqref{7.1}-\eqref{7.2} we see that \eqref{6.55.1}-\eqref{6.56.1} and \eqref{6.64} are available; later on, we shall strengthen \eqref{6.0} by assuming \eqref{6.54} thus \eqref{6.55}-\eqref{6.56} and Proposition \ref{vmo.t} will be at hand. Now, with $H\equiv H(\textnormal{\texttt{data}},M)$ as in $\eqref{6.9}_{2}$ and $H_{1},H_{2}\equiv H_{1},H_{2}(\textnormal{\texttt{data}},M)$ being defined in \eqref{hhh}, we slightly modify the definition of the composite excess functional given in \eqref{cb} and consider its "unbalanced" version:
$$
(0,\rr]\ni s\mapsto \mf{C}_{H}(x_{0};s):=H\mf{F}(u;B_{\rr}(x_{0}))+\snr{(V_{p}(Du))_{B_{\rr}(x_{0})}}
$$
and, for $s\in (0,\rr]$, introduce the nonhomogeneous excess functional:
$$
 \mf{N}(x_{0};s):=H_{1}\mf{F}(u;B_{s}(x_{0}))+c_{12}H_{2}\left(\mf{K}\left(\textbf{I}^{f}_{1,m}(x_{0},s)\right)^{\frac{p}{2(p-1)}}+\snr{(V_{p}(Du))_{B_{s}(x_{0})}}^{(2-p)/p}\mathbf{I}^{f}_{1,m}(x_{0},s)\right),
$$
where $\mf{K}(\cdot)$ is defined in \eqref{ik} and $c_{12}:=\left(2^{16npq}(\tau\theta)^{-8npq}\right)^{\frac{p}{2(p-1)}}$. Notice that by \eqref{6.64}, $\eqref{6.55.1}_{1}$ and \eqref{6.0} we have
\begin{eqnarray}\label{7.5}
\lim_{\rr\to 0}\mf{F}(u;B_{\rr}(x_{0}))=0 \ \Longrightarrow \ \lim_{\rr\to 0}\mf{N}(x_{0};\rr)=0.
\end{eqnarray}
For the ease of exposition, we shall adopt some abbreviations. With $\tau\equiv \tau(\textnormal{\texttt{data}},M)$ being the parameter determined in Propositions \ref{p1}-\ref{p2}, $j\in \N\cup \{-1,0\}$ and $\sigma\in (0,\rr]$ set $\sigma_{j}:=\tau^{j+1}\sigma$, $\sigma_{-1}:= \sigma$ and $B_{j}:=B_{\sigma_{j}}(x_{0})$. From now on we will mostly employ the shorthands described in \textbf{Step 1} of the proof of Theorem \ref{t.ex} and, with Remark \ref{rem} in mind, unless otherwise specified we shall work within the setting designed at the beginning of Section \ref{7sec}.
\subsection{An inductive lemma} The key tool for proving our sharp partial continuity result is an inductive technical lemma that is the subquadratic counterpart of \cite[Lemma 6.1]{deqc}, both inspired by \cite[Lemma 6.1]{kumi}.
\begin{lemma}\label{i.lem}
Let $x_{0}\in \mathcal{R}_{u}$ be a point with $M\equiv M(x_{0})$ being the positive constant in \eqref{ru.0}, $\gamma$ be a positive number and assume that $\bar{\varepsilon}\equiv \varepsilon'$, $\bar{\rr}\equiv \rr'$ in \eqref{ru.0} with $\varepsilon'\equiv \varepsilon'(\textnormal{\texttt{data}},M)$, $\rr'\equiv \rr'(\textnormal{\texttt{data}},M,f(\cdot))$ defined in \eqref{7.1}-\eqref{7.2}; that
\begin{eqnarray}\label{7.6}
\mf{N}(x_{0};\sigma)\le 2\gamma\qquad \mbox{for some} \ \ \sigma\in (0,\rr]
\end{eqnarray}
and that, for integers $k\ge i\ge 0$ inequalities
\begin{eqnarray}\label{7.7}
\mf{C}_{H}(\sigma_{j})\le \gamma, \qquad   \mf{C}_{H}(\sigma_{j+1})\ge \frac{\gamma}{16} \qquad  \mbox{for all} \ \ j\in \{i,\cdots,k\},\qquad\qquad \mf{C}_{H}(\sigma_{i})\le \frac{\gamma}{4}
\end{eqnarray}
are verified. Then the following holds:
\begin{flalign}\label{7.8}
\mf{C}_{H}(\sigma_{k+1})\le \gamma,\qquad \qquad\qquad \sum_{j=i}^{k+1}\mf{F}(\sigma_{j})\le \frac{\gamma}{2H}
\end{flalign}
and
\begin{flalign}\label{7.8.1}
\sum_{j=i}^{k+1}\mf{F}(\sigma_{j})\le \frac{4\mf{F}(\sigma_{i})}{3}+\frac{2^{3}\gamma^{(2-p)/p}}{3\varepsilon_{1}\tau^{n/2}}\sum_{j=i}^{k}\mf{S}(\sigma_{j}),
\end{flalign}
where $H,\varepsilon_{1}\equiv H,\varepsilon_{1}(\textnormal{\texttt{data}},M)$ defined in $\eqref{6.9}_{2}$ and in Propositions \ref{p1}-\ref{p2} respectively.
\end{lemma}
\begin{proof}
Our preliminary observation is that $x_{0}\in \mathcal{R}_{u}$ with $\bar{\varepsilon}\equiv \varepsilon'$ and $\bar{\rr}\equiv \rr'$ guarantees the validity of \eqref{7.4} and of \eqref{6.55.1}-\eqref{6.56.1}. A straightforward computation shows that
\begin{eqnarray}\label{7.9}
\snr{V(\sigma_{j})-V(\sigma_{j+1})}\stackrel{\eqref{tri.1}_{3}}{\le} \frac{\mf{F}(\sigma_{j})}{\tau^{n/2}}\le \frac{\mf{C}_{H}(\sigma_{j})}{\tau^{n/2}H}\stackrel{\eqref{6.9}_{2},\eqref{7.7}_{1}}{\le}\frac{\gamma}{2^{6}}.
\end{eqnarray}
Next, let us prove that under \eqref{7.6}-\eqref{7.7} the singular regime cannot be in force, i.e.:
\begin{flalign}\label{7.10}
\varepsilon_{0}V(\sigma_{j})\le \mf{F}(\sigma_{j})\quad \mbox{cannot hold for all} \ \ j\in \{i,\cdots,k\}.
\end{flalign}
By contradiction, we assume that
\begin{flalign}\label{7.11}
\mbox{there is} \ \ j\in \{i,\cdots,k\} \ \ \mbox{such that} \ \ \varepsilon_{0}V(\sigma_{j})\le \mf{F}(\sigma_{j}) \ \ \mbox{holds true}
\end{flalign}
and estimate via Young inequality with conjugate exponents $\left(\frac{p}{2-p},\frac{p}{2(p-1)}\right)$,
\begin{eqnarray}\label{7.12}
H\mf{F}(\sigma_{j+1})&\stackrel{\eqref{6.56.1}}{\le}&c_{7}H\tau^{\alpha_{0}(j+2)}\mf{F}(\sigma)+c_{8}H\sup_{s\le \sigma/4}\mf{K}\left(\mf{S}(s)\right)^{\frac{p}{2(p-1)}}\nonumber \\
&&+c_{8}H\left(\mf{C}(\sigma)+\mf{K}\left(\mathbf{I}^{f}_{1,m}(x_{0},\sigma)\right)^{\frac{p}{2(p-1)}}\right)^{(2-p)/p}\sup_{s\le \sigma/4}\mf{S}(s)\nonumber \\
&\le&H(c_{7}+2c_{8})\mf{F}(\sigma)+c_{8}H\sup_{s\le \sigma/4}\mf{K}\left(\mf{S}(s)\right)^{\frac{p}{2(p-1)}}\nonumber \\
&&+c_{8}H\mf{K}(\mathbf{I}^{f}_{1,m}(x_{0},\sigma))^{\frac{2-p}{2(p-1)}}\sup_{s\le \sigma/4}\mf{S}(s)+c_{8}HV(\sigma)^{(2-p)/p}\sup_{s\le \sigma/4}\mf{S}(\sigma)\nonumber \\
&\stackrel{\eqref{6.3.1}}{\le}&\mf{N}(x_{0};\sigma)\left(\frac{H(c_{7}+c_{8})}{H_{1}}+\frac{Hc_{8}}{c_{12}H_{2}}\left(\frac{2^{8nq}}{(\tau\theta)^{4nq}}\right)^{\frac{p}{2(p-1)}}+\frac{2Hc_{8}}{c_{12}H_{2}}\left(\frac{2^{8n}}{(\tau\theta)^{4n}}\right)\right)\nonumber \\
&\stackrel{\eqref{7.6}}{\le}&\gamma\left(\frac{2H(c_{7}+c_{8})}{H_{1}}+\frac{2Hc_{8}}{c_{12}H_{2}}\left(\frac{2^{8nq}}{(\tau\theta)^{4nq}}\right)^{\frac{p}{2(p-1)}}+\frac{2^{2}Hc_{8}}{c_{12}H_{2}}\left(\frac{2^{8n}}{(\tau\theta)^{4n}}\right)\right)\stackrel{\eqref{hhh}}{\le}\frac{\gamma}{2^{6}}.
\end{eqnarray}
Furthermore, we have
\begin{eqnarray}\label{7.13}
V(\sigma_{j})\stackrel{\eqref{7.10}}{\le}\frac{\mf{F}(\sigma_{j})}{\varepsilon_{0}}\le \frac{\mf{C}_{H}(\sigma_{j})}{\varepsilon_{0}H}\stackrel{\eqref{7.7}_{1}}{\le}\frac{\gamma}{\varepsilon_{0}H}\stackrel{\eqref{6.9}_{2}}{\le}\frac{\gamma}{2^{6}}
\end{eqnarray}
and so
\begin{eqnarray*}
\mf{C}_{H}(\sigma_{j+1})\le\snr{V(\sigma_{j+1})-V(\sigma_{j})}+V(\sigma_{j})+H\mf{F}(\sigma_{j+1})\stackrel{\eqref{7.9},\eqref{7.12},\eqref{7.13}}{\le}\frac{3\gamma}{2^{6}}<\frac{\gamma}{16},
\end{eqnarray*}
in contradiction with $\eqref{7.7}_{2}$ and \eqref{7.10} is verified. Next, we prove the validity of
\begin{eqnarray}\label{7.14}
\mf{F}(\sigma_{j+1})\le \frac{\mf{F}(\sigma_{j})}{4}+\frac{2\gamma^{(2-p)/p}\mf{S}(\sigma_{j})}{\varepsilon_{1}\tau^{n/2}}.
\end{eqnarray}
In the light of \eqref{7.10}, we have to consider only two possibilities: either \eqref{ns.2} holds and, given $\eqref{7.7}_{1}$ and the bound imposed on the size of $\gamma$, via \eqref{nsns} and \eqref{ns.20} we directly have \eqref{7.14}; or \eqref{ns.21} is satisfied and 
\begin{eqnarray*}
\mf{F}(\sigma_{j+1})&\stackrel{\eqref{tri.1}_{1}}{\le}&\frac{2}{\tau^{n/2}}\mf{F}(\sigma_{j})\stackrel{\eqref{ns.21}}{\le}\frac{2V(\sigma_{j})^{(2-p)/p}\mf{S}(\sigma_{j})}{\varepsilon_{1}\tau^{n/2}}\stackrel{\eqref{pq}_{1},\eqref{7.7}_{1}}{\le}\frac{2\gamma^{(2-p)/p}\mf{S}(\sigma_{j})}{\varepsilon_{1}\tau^{n/2}},
\end{eqnarray*}
and \eqref{7.14} follows in any case. Before proceeding further, notice that
\begin{eqnarray}\label{7.15}
\left(\sum_{j=i}^{k}\mf{S}(\sigma_{j})\right)^{\frac{p}{2(p-1)}}\stackrel{\eqref{6.2}}{\le}\frac{2^{\frac{4np}{p-1}}\left(\mathbf{I}^{f}_{1,m}(x_{0},\sigma)\right)^{\frac{p}{2(p-1)}}}{(\tau\theta)^{\frac{2np}{p-1}}}\le \frac{\mf{N}(x_{0};\sigma)}{H_{2}}\stackrel{\eqref{7.6}}{\le}\frac{2\gamma}{H_{2}}.
\end{eqnarray}
Summing \eqref{7.14} for $j\in \{i,\cdots,k\}$ we obtain
\begin{eqnarray*}
\sum_{j=i+1}^{k+1}\mf{F}(\sigma_{j})\le \frac{1}{4}\sum_{j=i}^{k}\mf{F}(\sigma_{j})+\frac{2\gamma^{(2-p)/p}}{\varepsilon_{1}\tau^{n/2}}\sum_{j=i}^{k}\mf{S}(\sigma_{j})
\end{eqnarray*}
Adding on both sides of the previous inequality $\mf{F}(\sigma_{i})$ and reabsorbing terms, we get \eqref{7.8.1}. We continue estimating in \eqref{7.8.1}:
\begin{eqnarray*}
\sum_{j=i}^{k+1}\mf{F}(\sigma_{j})\stackrel{\eqref{7.15}}{\le} \frac{4\mf{C}_{H}(\sigma_{i})}{3H}+\frac{2^{5}\gamma}{3\varepsilon_{1}\tau^{n/2}H_{2}^{\frac{2(p-1)}{p}}}\stackrel{\eqref{7.7}_{3}}{\le}\left(\frac{1}{3H}+\frac{2^{5}}{3\varepsilon_{1}\tau^{n/2}H_{2}^{\frac{2(p-1)}{p}}}\right)\gamma\stackrel{\eqref{hhh}}{\le}\frac{5\gamma}{12H},
\end{eqnarray*}
which implies \eqref{7.8}$_{2}$. Finally, we estimate
\begin{eqnarray*}
V(\sigma_{k+1})&\le& \snr{V(\sigma_{k+1})-V(\sigma_{i})}+V(\sigma_{i})\stackrel{\eqref{7.7}_{3}}{\le}\frac{\gamma}{4}+\sum_{j=i}^{k}\snr{V(\sigma_{j+1})-V(\sigma_{j})}\nonumber \\
&\le&\frac{\gamma}{4}+\frac{1}{\tau^{n/2}}\sum_{j=i}^{k}\mf{F}(\sigma_{j})\stackrel{\eqref{7.8}_{2}}{\le}\frac{\gamma}{4}+\frac{\gamma}{2H\tau^{n/2}}\stackrel{\eqref{6.9}_{2}}{\le}\frac{\gamma}{2}
\end{eqnarray*}
and, combining the content of the above display with $\eqref{7.8}_{2}$ we obtain $\eqref{7.8}_{1}$ and the proof is complete.
\end{proof}
\subsection{Oscillation estimates for large gradients}\label{7.2s} For some $\sigma\in (0,\rr ]$ we consider the case in which
\begin{eqnarray}\label{7.17}
\frac{\gamma}{8}:=V(\sigma_{0})>\frac{\mf{N}(x_{0};\sigma)}{16} \ \Longrightarrow \ \mf{N}(x_{0};\sigma)\le 2\gamma
\end{eqnarray}
where we used that by \eqref{6.0}, \eqref{7.1}, \eqref{7.2} and \eqref{7.4}, estimates \eqref{6.55.1}-\eqref{6.56.1} of Corollary \ref{cor.ex} are available, recall also Remark \ref{rem}. We proceed with two technical lemmas, eventually leading to quantitative oscillation estimates for nonzero gradients.  
\begin{lemma}
Assume \eqref{7.17}. Then
\begin{flalign}\label{7.18}
\sum_{j=0}^{\infty}\mf{F}(\sigma_{j})\le \frac{\mf{N}(x_{0};\sigma)}{H}\qquad \mbox{and}\qquad \frac{\gamma}{16}\le V(\sigma_{j})\le \gamma \ \ \mbox{for all} \ \ j\in \N\cup\{0\},
\end{flalign}
for any $\sigma\in (0,\rr]$ with $H\equiv H(\textnormal{\texttt{data}},M)$ defined in $\eqref{6.9}_{2}$.
\end{lemma}

\begin{proof}
Let us first prove that
\begin{eqnarray}\label{7.19}
V(\sigma_{j})\ge \frac{V(\sigma_{0})}{2}\qquad \mbox{for all} \ \ j\in \N\cup\{0\}.
\end{eqnarray}
Notice that
\begin{eqnarray*}
V(\sigma_{1})&\ge& V(\sigma_{0})-\snr{V(\sigma_{1})-V(\sigma_{0})}\stackrel{\eqref{tri.1}_{3}}{\ge} V(\sigma_{0})-\frac{\mf{F}(\sigma_{0})}{\tau^{n/2}}\nonumber \\
&\stackrel{\eqref{tri.1}_{1}}{\ge}& V(\sigma_{0})-\frac{2\mf{F}(\sigma_{-1})}{\tau^{n}}\ge V(\sigma_{0})-\frac{2\mf{N}(x_{0};\sigma)}{\tau^{n}H_{1}}\nonumber \\
&\stackrel{\eqref{hhh}}{\ge}& V(\sigma_{0})-\frac{\mf{N}(x_{0};\sigma)}{2^{7}}\stackrel{\eqref{7.17}}{\ge}\frac{V(\sigma_{0})}{2}.
\end{eqnarray*}
By contradiction, we assume that there is a finite exit time index $J\ge 2$ such that
\begin{eqnarray}\label{7.20}
V(\sigma_{J})< \frac{V(\sigma_{0})}{2}\qquad \mbox{and}\qquad V(\sigma_{j})\ge \frac{V(\sigma_{0})}{2}  \quad \mbox{for all} \ \ j\in \{0,\cdots,J-1\}.
\end{eqnarray}
Let us preliminary observe that
\begin{flalign}\label{7.21}
V(\sigma_{j})\ge \frac{V(\sigma_{0})}{2} \ \ \mbox{for all} \ \ j\in \{0,\cdots,J-1\} \ \Longrightarrow \ \mf{C}_{H}(\sigma_{j})\le \gamma \ \ \mbox{for all} \ \ j\in \{0,\cdots,J-1\}.
\end{flalign}
To show the validity of implication \eqref{7.21}, we proceed by induction. By direct calculation, we see that
\begin{flalign}
\mf{C}_{H}(\sigma_{0})\stackrel{\eqref{tri.1}_{1}}{\le} V(\sigma_{0})+\frac{2H\mf{F}(\sigma_{-1})}{\tau^{n/2}}\le V(\sigma_{0})+\frac{2H\mf{N}(x_{0};\sigma)}{\tau^{n/2}H_{1}}\stackrel{\eqref{7.17}}{\le}\gamma\left(\frac{1}{8}+\frac{2^{2}H}{\tau^{n/2}H_{1}}\right)\stackrel{\eqref{hhh}}{\le}\frac{\gamma}{4}.\label{7.21.1}
\end{flalign}
We then fix an arbitrary $k\in \{0,\cdots,J-2\}$, assume that $\mf{C}_{H}(\sigma_{j})\le \gamma$ holds for all $j\in \{0,\cdots,k\}$ and notice that $\eqref{7.20}_{2}$ and \eqref{7.17} yield that $\mf{C}_{H}(\sigma_{j+1})\ge \gamma/16$ for all $j\in \{0,\cdots,k\}$, therefore keeping \eqref{7.17} in mind, we deduce that the assumptions of Lemma \ref{i.lem} are verified with $i=0$ and $k$ being the number used here so $\mf{C}_{H}(\sigma_{k+1})\le \gamma$. Implication \eqref{7.21} then follows by the arbitrariety of $k\in \{0,\cdots,J-2\}$. By \eqref{7.20}-\eqref{7.21} now we know that $\mf{C}_{H}(\sigma_{j})\le \gamma$ for all $j\in \{0,\cdots,J-1\}$ and $\mf{C}_{H}(\sigma_{j+1})\ge \gamma/16$ for all $j\in \{0,\cdots,J-2\}$, thus via \eqref{7.18} we can apply again Lemma \ref{i.lem} with $i=0$ and $k=J-2$ to get
\begin{eqnarray}\label{7.22}
\sum_{j=0}^{J-1}\mf{F}(\sigma_{j})\le \frac{\gamma}{2H}\stackrel{\eqref{7.17}}{\le}\frac{4V(\sigma_{0})}{H},
\end{eqnarray}
so we can bound
\begin{eqnarray}\label{7.23}
\snr{V(\sigma_{J})-V(\sigma_{0})}\le \sum_{j=0}^{J-1}\snr{V(\sigma_{j+1})-V(\sigma_{j})}\stackrel{\eqref{tri.1}_{1}}{\le}\frac{1}{\tau^{n/2}}\sum_{j=0}^{J-1}\mf{F}(\sigma_{j})\stackrel{\eqref{7.22}}{\le}\frac{4V(\sigma_{0})}{\tau^{n/2}H}\stackrel{\eqref{6.9}_{2}}{\le}\frac{V(\sigma_{0})}{4}
\end{eqnarray}
for concluding:
\begin{eqnarray*}
V(\sigma_{J})\ge V(\sigma_{0})-\snr{V(\sigma_{0})-V(\sigma_{J})}\stackrel{\eqref{7.23}}{\ge}\frac{3V(\sigma_{0})}{4},
\end{eqnarray*}
in contradiction with \eqref{7.20}$_{1}$. This and the arbitrariety of $J\ge 2$ yield validity of \eqref{7.19}, which in turn implies the left-hand side of inequality \eqref{7.18}$_{2}$ and, applying \eqref{7.21} for all $j\in \N\cup\{0\}$ we derive the full chain of inequalities in $\eqref{7.18}_{2}$. We only need to verify $\eqref{7.18}_{1}$. Using \eqref{7.17}, \eqref{7.18}$_{2}$ and \eqref{7.21.1} we apply Lemma \ref{i.lem} with $i=0$ and for every integer $k$ to have
\begin{eqnarray*}
\sum_{j=0}^{\infty}\mf{F}(\sigma_{j})&\stackrel{\eqref{7.8.1}}{\le}&\frac{4\mf{F}(\sigma_{0})}{3}+\frac{2^{3}\gamma^{(2-p)/p}}{3\varepsilon_{1}\tau^{n/2}}\sum_{j=0}^{\infty}\mf{S}(\sigma_{j})\nonumber \\
&\stackrel{\eqref{tri.1}_{1},\eqref{6.2}}{\le}&\frac{8\mf{F}(\sigma_{-1})}{3\tau^{n/2}}+\frac{2^{3+4n}\gamma^{(2-p)/p}\mathbf{I}^{f}_{1,m}(x_{0},\sigma)}{3\varepsilon_{1}(\tau\theta)^{4n}}\nonumber \\
&\stackrel{\eqref{7.17}}{\le}&\frac{8\mf{N}(x_{0};\sigma)}{3\tau^{n/2}H_{1}}+\frac{2^{10n}V(\sigma_{0})^{(2-p)/p}\mathbf{I}^{f}_{1,m}(x_{0},\sigma)}{3\varepsilon_{1}(\tau\theta)^{4n}}\nonumber \\
&\stackrel{\eqref{tri.1}_{2}}{\le}&\frac{8\mf{N}(x_{0};\sigma)}{3\tau^{n/2}H_{1}}+\frac{2^{10n}V(\sigma_{-1})^{(2-p)/p}\mathbf{I}^{f}_{1,m}(x_{0},\sigma)}{3\varepsilon_{1}(\tau\theta)^{4n}}+\frac{2^{12n}\mf{F}(\sigma_{-1})^{(2-p)/p}\mathbf{I}^{f}_{1,m}(x_{0},\sigma)}{3\varepsilon_{1}(\tau\theta)^{6n}}\nonumber \\
&\le&\mf{N}(x_{0};\sigma)\left(\frac{8}{3\tau^{n/2}H_{1}}+\frac{2^{10n}}{3\varepsilon_{1}(\tau\theta)^{4n}c_{12}H_{2}}+\frac{2^{12n}}{3\varepsilon_{1}(\tau\theta)^{6n}H_{1}}+\frac{2^{12n}}{3\varepsilon_{1}c_{12}H_{2}(\tau\theta)^{6n}}\right)\nonumber \\
&\stackrel{\eqref{hhh}}{\le}&\frac{\mf{N}(x_{0};\sigma)}{H},
\end{eqnarray*}
where we also used Young inequality with conjugate exponents $\left(\frac{p}{2(p-1)},\frac{p}{2-p}\right)$ and the proof is complete.
\end{proof}

\begin{lemma}\label{7.3l}
Whenever $\sigma\in (0,\rr]$ is such that \eqref{7.17} holds, the limits in \eqref{limv} exist and inequalities
\begin{eqnarray}\label{7.25}
\left\{
\begin{array}{c}
\displaystyle
\ \snr{V_{p}(Du(x_{0}))-(V_{p}(Du))_{B_{\sigma}(x_{0})}}\le c\mf{N}(x_{0};\sigma)\\[8pt]\displaystyle 
\ \snr{Du(x_{0})-(Du)_{B_{\sigma}(x_{0})}}\le c\mf{N}(x_{0};\sigma)^{2/p}+c\snr{(Du)_{B_{\sigma}(x_{0})}}^{(2-p)/2}\mf{N}(x_{0};\sigma)
\end{array}
\right.
\end{eqnarray}
hold true for a constant $c\equiv c(\textnormal{\texttt{data}},M)$.
\end{lemma}
\begin{proof}
We start by showing that $\{(V_{p}(Du))_{B_{j}}\}_{j\in \N\cup\{0\}}$ is a Cauchy sequence. In fact, fixed integers $0\le i\le k-1$ we bound
\begin{eqnarray}\label{7.26}
\snr{(V_{p}(Du))_{B_{k}}-(V_{p}(Du))_{B_{i}}}&\le& \sum_{j=i}^{k-1}\snr{(V_{p}(Du))_{B_{j+1}}-(V_{p}(Du))_{B_{j}}}\nonumber \\
&\stackrel{\eqref{tri.1}_{3}}{\le}& \frac{1}{\tau^{n/2}}\sum_{j=i}^{k-1}\mf{F}(\sigma_{j})\le \frac{1}{\tau^{n/2}}\sum_{j=i}^{\infty}\mf{F}(\sigma_{j})\stackrel{\eqref{7.18}}{\le}c\mf{N}(x_{0};\sigma),
\end{eqnarray}
and
\begin{eqnarray}\label{7.27}
\snr{(V_{p}(Du))_{B_{0}}-(V_{p}(Du))_{B_{-1}}}\stackrel{\eqref{tri.1}_{3}}{\le} \frac{\mf{F}(\sigma)}{\tau^{n/2}}\le c\mf{N}(x_{0};\sigma)
\end{eqnarray}
with $c\equiv c(\textnormal{\texttt{data}},M)$, therefore there exists $\ell_{V}\in \mathbb{R}^{N\times n}$ such that 
\eqn{7.28}
$$\lim_{j\to \infty}(V_{p}(Du))_{B_{j}}=\ell_{V}.$$ Sending $k\to \infty$ in \eqref{7.26} we obtain
\begin{eqnarray*}
\snr{\ell_{V}-(V_{p}(Du))_{B_{i}}}\le c\mf{N}(x_{0};\sigma)\qquad \mbox{for all} \ \ i\in \N\cup\{0\}.
\end{eqnarray*}
Now, given any $s\in (0,\sigma]$ - and since we are interested in $s\to 0$ we can assume $s\le \sigma_{0}$ - there is $j_{s}\in \N\cup\{0\}$ such that $\sigma_{j_{s}+1}<s\le \sigma_{j_{s}}$ and
\begin{eqnarray}\label{7.28.1}
\lim_{s\to 0}\snr{\ell_{V}-(V_{p}(Du))_{B_{s}(x_{0})}}&\le& \lim_{j_{s}\to \infty}\snr{\ell_{V}-(V_{p}(Du))_{B_{j_{s}}}}+\lim_{j_{s}\to \infty}\snr{(V_{p}(Du))_{B_{s}(x_{0})}-(V_{p}(Du))_{B_{j_{s}}}}\nonumber \\
&\stackrel{\eqref{tri.1}_{3}}{\le}&\lim_{j_{s}\to \infty}\snr{\ell_{V}-(V_{p}(Du))_{B_{j_{s}}}}+\frac{1}{\tau^{n/2}}\lim_{j_{s}\to \infty}\mf{F}(\sigma_{j_{s}})\stackrel{\eqref{7.28},\eqref{6.64}}{=}0,
\end{eqnarray}
and the first limit in \eqref{limv} equals $\ell_{V}$, which defines the precise representative of $V_{p}(Du)$ at $x_{0}$, i.e.: $\ell_{V}=(V_{p}(Du))(x_{0})$. Next, notice that whenever $B\Subset \Omega$ is a ball, by \eqref{minav} and \cite[(2.6)]{gm} it is
\begin{eqnarray}\label{7.30}
\mf{F}(u;B)\approx \left(\mint_{B}\snr{V_{p}(Du)-V_{p}((Du)_{B})}^{2} \ \dx\right)^{1/2},
\end{eqnarray}
with constants implicit in "$\approx$" depending only on $p$,  so for any given $j\in \N\cup\{0\}$ it is
\begin{eqnarray}\label{xx.0}
\snr{(Du)_{B_{j}}}&\le& \mf{J}_{2}(V_{p}(Du);B_{j})^{2/p}\le c\mf{F}(\sigma_{j})^{2/p}+cV(\sigma_{j})^{2/p}
\end{eqnarray}
with $c\equiv c(p)$, while for $j=-1$ via H\"older and Young inequalities with conjugate exponents $\left(\frac{2-p}{2},\frac{2}{p}\right)$ we have
\begin{eqnarray}\label{xx.1}
\snr{(Du)_{B_{-1}}}&\le&\mf{J}_{2}(V_{p}(Du);B_{0})^{2/p}+\left(\mint_{B_{0}}\snr{Du-(Du)_{B_{-1}}}^{p} \ \dx\right)^{1/p}\nonumber \\
&\stackrel{\eqref{tri.1}_{1},\eqref{Vm.1}}{\le}&\frac{c\mf{F}(\sigma_{-1})^{2/p}}{\tau^{n/p}}+cV(\sigma_{0})^{2/p}+c\left(\mint_{B_{0}}\snr{V_{p}(Du)-V_{p}((Du)_{B_{-1}})}^{2} \ \dx\right)^{1/p}\nonumber \\
&&+c\snr{(Du)_{B_{-1}}}^{(2-p)/2}\left(\mint_{B_{0}}\snr{V_{p}(Du)-V_{p}((Du)_{B_{-1}})}^{p} \ \dx\right)^{1/p}\nonumber\\
&\stackrel{\eqref{7.30},\eqref{tri.1}_{1}}{\le}&\frac{1}{2}\snr{(Du)_{B_{-1}}}+c\mf{F}(\sigma_{-1})^{2/p}+cV(\sigma_{0})^{2/p},
\end{eqnarray}
for $c\equiv c(\textnormal{\texttt{data}},M)$ and using this time Young inequality with conjugate exponents $\left(\frac{2}{2-p},\frac{2}{p}\right)$ we get
\begin{eqnarray}\label{7.29.0.0}
V(\sigma_{0})&\le&\mf{J}_{p}(Du;B_{0})^{p/2}\le c\left(\mint_{B_{0}}\snr{Du-(Du)_{B_{-1}}}^{p} \ \dx\right)^{1/2}+c\snr{(Du)_{B_{-1}}}^{p/2}\nonumber \\
&\stackrel{\eqref{Vm.1}}{\le}&c\snr{(Du)_{B_{-1}}}^{p/2}+c\left(\mint_{B_{0}}\snr{V_{p}(Du)-V_{p}((Du)_{B_{-1}})}^{2} \ \dx\right)^{1/2}\nonumber \\
&&+c\snr{(Du)_{B_{-1}}}^{p(2-p)/4}\left(\mint_{B_{0}}\snr{V_{p}(Du)-V_{p}((Du)_{B_{-1}})}^{p} \ \dx\right)^{1/2}\nonumber \\
&\stackrel{\eqref{7.30},\eqref{tri.1}_{1}}{\le}&c\snr{(Du)_{B_{-1}}}^{p/2}+\frac{c}{\tau^{n/2}}\mf{F}(\sigma_{-1})+\frac{c}{\tau^{\frac{np}{4}}}\snr{(Du)_{B_{-1}}}^{p(2-p)/4}\mf{F}(\sigma_{-1})^{p/2}\nonumber \\
&\le& c\snr{(Du)_{B_{-1}}}^{p/2}+c\mf{F}(\sigma_{-1}),
\end{eqnarray}
for $c\equiv c(\textnormal{\texttt{data}},M)$, therefore in any case it is
\begin{eqnarray}\label{7.29}
\snr{(Du)_{B_{j}}}\stackrel{\eqref{7.18}}{\le} c\mf{N}(x_{0};\sigma)^{2/p}+c\gamma^{2/p}\qquad \mbox{for all} \ \ j\in \N\cup\{0,-1\},
\end{eqnarray}
with $c\equiv c(\textnormal{\texttt{data}},M)$. Moreover, given any ball $B\Subset \Omega$, by triangular and H\"older inequalities we bound
\begin{eqnarray}\label{7.30.1}
\snr{V_{p}((Du)_{B})-(V_{p}(Du))_{B}}\le \left(\mint_{B}\snr{V_{p}(Du)-V_{p}((Du)_{B})}^{2} \ \dx\right)^{1/2}\stackrel{\eqref{7.30}}{\le}c(p)\mf{F}(u;B)
\end{eqnarray}
and then estimate for integers $0\le i\le k-1$:
\begin{eqnarray*}
\snr{(Du)_{B_{i}}-(Du)_{B_{k}}}&\le& \sum_{j=i}^{k-1}\snr{(Du)_{B_{j+1}}-(Du)_{B_{i}}}\nonumber \\
&\stackrel{\eqref{Vm}_{2}}{\le}&c\sum_{j=i}^{k-1}\snr{V_{p}((Du)_{B_{j+1}})-V_{p}((Du)_{B_{j}})}(\snr{(Du)_{B_{j+1}}}+\snr{(Du)_{B_{j}}})^{(2-p)/2}\nonumber \\
&\stackrel{\eqref{pq}_{1},\eqref{7.29}}{\le}&c\left(\mf{N}(x_{0};\sigma)^{(2-p)/p}+\gamma^{(2-p)/p}\right)\sum_{j=i}^{k-1}\snr{V_{p}((Du)_{B_{j+1}})-V_{p}((Du)_{B_{j}})}\nonumber \\
&\le&c\left(\mf{N}(x_{0};\sigma)^{(2-p)/p}+\gamma^{(2-p)/p}\right)\sum_{j=i}^{k}\snr{V_{p}((Du)_{B_{j}})-(V_{p}(Du))_{B_{j}}}\nonumber \\
&&+c\left(\mf{N}(x_{0};\sigma)^{(2-p)/p}+\gamma^{(2-p)/p}\right)\sum_{j=i}^{k-1}\snr{(V_{p}(Du))_{B_{j+1}}-(V_{p}(Du))_{B_{j}}}\nonumber \\
&\stackrel{\eqref{tri.1}_{3},\eqref{7.30.1}}{\le}&\frac{c}{\tau^{n/2}}\left(\mf{N}(x_{0};\sigma)^{(2-p)/p}+\gamma^{(2-p)/p}\right)\sum_{j=i}^{k}\mf{F}(\sigma_{j})\nonumber \\
&\stackrel{\eqref{7.18}_{1}}{\le}&c\mf{N}(x_{0},\sigma)^{2/p}+c\gamma^{(2-p)/p}\mf{N}(x_{0};\sigma)\nonumber \\
&\stackrel{\eqref{7.17},\eqref{7.29.0.0}}{\le}&c\mf{N}(x_{0};\sigma)^{2/p}+c\snr{(Du)_{B_{-1}}}^{(2-p)/2}\mf{N}(x_{0};\sigma)
\end{eqnarray*}
for $c\equiv c(\textnormal{\texttt{data}},M)$. Recalling \eqref{6.55}$_{1}$, we get that $\{(Du)_{B_{j}}\}_{j\in \N\cup\{0\}}$ is a Cauchy sequence and there exists $\ell\in \mathbb{R}^{N\times n}$ such that $\lim_{j\to \infty}(Du)_{B_{j}}=\ell$. A standard interpolative argument analogous to that leading to \eqref{7.28.1} allows concluding that $\ell$ defines the precise representative of $Du$ at $x_{0}$, i.e. $Du(x_{0})=\ell$ and this assures the validity of the second limit in \eqref{limv}. Combining this last information with \eqref{7.30.1} and \eqref{7.28.1} we get that $\ell_{V}=(V_{p}(Du))(x_{0})=V_{p}(Du(x_{0}))$ so via \eqref{7.28} we eventually recover the first limit in \eqref{limv}. Finally, merging \eqref{7.27}-\eqref{7.28} and recalling that $V_{p}(Du(x_{0}))=(V_{p}(Du))(x_{0})$ we obtain \eqref{7.25}. The proof is complete.
\end{proof}

\subsection{Oscillation estimates for small gradients}\label{7.3s}
In this section we look at what happens when the complementary condition to \eqref{7.17} holds, i.e. when for $\sigma\in (0,\rr]$ it is
\begin{eqnarray}\label{7.33}
\frac{\gamma}{8}=:\frac{\mf{N}(x_{0};\sigma)}{16}\ge V(\sigma_{0}) \ \Longrightarrow \ \mf{N}(x_{0};\sigma)= 2\gamma.
\end{eqnarray}
Let us first observe that to avoid trivialities, we can suppose $\gamma>0$, and that there is no loss of generality in assuming that \eqref{7.33} actually holds for all $s\in (0,\sigma]$. In fact, if for some $s\in (0,\sigma]$ the opposite inequality to \eqref{7.33}, i.e. \eqref{7.17} holds, then we can directly conclude with \eqref{limv} and \eqref{7.25}. The validity of \eqref{7.33} for all $s\in (0,\sigma]$, \eqref{6.55.1}$_{1}$ and \eqref{6.64} yield that $\lim_{s\to 0}(V_{p}(Du))_{B_{s}(x_{0})}=0$, therefore keeping in mind that also $\lim_{s\to 0}V(s)=0$ and that
\eqn{j.2}
$$\snr{(Du)_{B_{s}(x_{0})}}\le c(p)\left( \mf{F}(u;B_{s}(x_{0}))^{2/p}+V(s)^{2/p}\right)$$ 
we can conclude that $\lim_{s\to 0}(Du)_{B_{s}(x_{0})}=0$ and the existence of the two limits in \eqref{limv} is proven. Next, we show the validity of \eqref{7.25} also in the case in which \eqref{7.33} is in force. Let us prove by induction that
\begin{eqnarray}\label{7.34}
\mf{C}_{H}(\sigma_{j})\le \gamma \qquad \mbox{for all} \ \ j\in \N\cup\{0\}.
\end{eqnarray}
A direct computation renders:
\begin{flalign}
\mf{C}_{H}(\sigma_{0})\stackrel{\eqref{7.33}}{\le}\frac{\mf{N}(x_{0};\sigma)}{16}+\frac{2H\mf{F}(\sigma_{-1})}{\tau^{n/2}}\le \mf{N}(x_{0};\sigma)\left(\frac{1}{16}+\frac{2H}{\tau^{n/2}H_{1}}\right)\stackrel{\eqref{hhh},\eqref{7.33}}{\le}\frac{\gamma}{4}.\label{7.35}
\end{flalign}
Then, we assume by contradiction that $\{j\in \N\cup\{0\}\colon \mf{C}_{H}(\sigma_{j+1})>\gamma\}\not =\{\emptyset\}$, define $\mf{l}:=\min\{j\in \N\cup\{0\}\colon \mf{C}_{H}(\sigma_{j+1})>\gamma\}$, i.e. the smallest integer minus one for which \eqref{7.34} fails, introduce the set $\mathcal{I}_{\mf{l}}:=\{j\in \N\cup\{0\}\colon \mf{C}_{H}(\sigma_{j})\le \gamma/4, \ j<\mf{l}+1\}$ and set $\chi:=\max \mathcal{I}_{\mf{l}}$. Notice that by \eqref{7.35} it is $\mathcal{I}_{\mf{l}}\not =\{\emptyset \}$, by definition $\mf{C}_{H}(\sigma_{\chi})\le \gamma/4$ and for $j\in \{\chi,\cdots,\mf{l}\}$ we have $\gamma\ge \mf{C}_{H}(\sigma_{j+1})\ge\gamma/4>\gamma/16$, therefore, recalling also \eqref{7.33} we can apply Lemma \ref{i.lem} with $i\equiv \chi$ and $k\equiv \mf{l}$ to conclude that $\mf{C}_{H}(\sigma_{\mf{l}+1})\le \gamma$ in contradiction with the definition of $\mf{l}$. This means that $\{j\in \N\cup\{0\}\colon \mf{C}_{H}(\sigma_{j+1})>\gamma\}=\{\emptyset\}$ and \eqref{7.34} holds true. Next, we take any $s\in (0,\sigma_{0}]$, determine $j_{s}\in \N\cup \{0\}$ such that $\sigma_{j_{s}+1}< s\le \sigma_{j_{s}}$ and estimate
\begin{eqnarray*}
V(s)\stackrel{\eqref{ls.42.1}_{2}}{\le}V(\sigma_{j_{s}})+\frac{\mf{F}(\sigma_{j_{s}})}{\tau^{n/2}}\stackrel{\eqref{6.9}_{2}}{\le}\mf{C}_{H}(\sigma_{j_{s}})\le \gamma\le\mf{N}(x_{0};\sigma).
\end{eqnarray*}
Moreover, if $s\in (\sigma_{0},\sigma_{-1}]$ we directly obtain
\begin{flalign*}
V(s)\le V(\sigma_{0})+\snr{V(\sigma_{0})-V(s)}\stackrel{\eqref{ls.42.1}_{1}}{\le} V(\sigma_{0})+\frac{2\mf{F}(\sigma_{-1})}{\tau^{n}}\stackrel{\eqref{7.34}}{\le}\gamma+\frac{2\mf{N}(x_{0};\sigma)}{\tau^{n}H_{1}}\stackrel{\eqref{hhh}}{\le}2\gamma=\mf{N}(x_{0};\sigma),
\end{flalign*}
so in any case it is $\sup_{s\le \sigma}V(s)\le \mf{N}(x_{0};\sigma)$, which in turn implies that $\snr{V(\sigma_{-1})-V(s)}\le 2\mf{N}(x_{0};\sigma)$ and \eqref{7.25}$_{1}$ can now be derived by sending $s\to 0$ in the previous inequality and recalling \eqref{limv}. Concerning $\eqref{7.25}_{2}$, we use \eqref{xx.0}-\eqref{xx.1} and \eqref{7.34} to deduce that $\snr{(Du)_{B_{j}}}\le c\gamma^{2/p}$ for all $j\in \N\cup\{-1,0\}$. This, the same interpolation argument exploited before and standard manipulations eventually render that $\snr{(Du)_{B_{-1}}-(Du)_{B_{s}(x_{0})}}\le c\gamma^{2/p}$, which, together with \eqref{7.33} and \eqref{limv} yield $\eqref{7.25}_{2}$ by sending $s\to 0$. In conclusion, we have just proven the following lemma.
\begin{lemma}\label{7.4l}
Assume that \eqref{7.33} holds for some $\sigma\in (0,\rr]$. Then the limits in \eqref{limv} exist and the bounds in \eqref{7.25} are verified.
\end{lemma}
\subsection{Sharp partial gradient continuity and proof of Theorems \ref{t.v.0}-\ref{t.v}}
Let us complete the proof of Theorem \ref{t.v}, started in Sections \ref{7.2s}-\ref{7.3s}.
\begin{proof}[Proof of Theorem \ref{t.v}]
Let $x_{0}\in \mathcal{R}_{u}$ be a point satisfying \eqref{1.3}, $M\equiv M(x_{0})>0$, $\bar{\varepsilon}\in (0,1)$, $\bar{\rr}\in (0,\min\{1,d_{x_{0}}\})$ be the corresponding parameters in \eqref{ru.0} with $\bar{\varepsilon},\bar{\rr}$ to be determined. We define $\ti{\varepsilon}:=2^{-10}\varepsilon'$ and suitably reduce the threshold radius to determine $\ti{\rr}\in (0,\rr']$ in such a way that inequality \eqref{7.2} holds with $\ti{\varepsilon}2^{-2}$ replacing $\varepsilon'$ for all $s\in (0,\ti{\rr}]$. Setting $\bar{\varepsilon}\equiv \ti{\varepsilon}/2$ and $\bar{\rr}\equiv \ti{\rr}$ in \eqref{ru.0} we see that both \eqref{1.4} and the assumptions in force in Sections \ref{7.2s}-\ref{7.3s} are satisfied, therefore the existence of the limits in \eqref{limv} follows from Lemmas \ref{7.3l}-\ref{7.4l}, while the (almost) pointwise oscillation estimates in \eqref{1.7} are exactly those appearing in \eqref{7.25}. We are only left with the proof of the assertion on the Lebesgue points of $V_{p}(Du)$ and of $Du$. Let us first assume that $x_{0}$ verifies both \eqref{1.3} and \eqref{1.4} with the just fixed parameters $\ti{\varepsilon}\equiv \ti{\varepsilon}(\textnormal{\texttt{data}},M)$ and $\ti{\rr}\equiv \ti{\rr}(\textnormal{\texttt{data}},M,f(\cdot))$. This choice assures that \eqref{limv}, \eqref{6.55.1} and \eqref{6.64} are available, and this in particular assures that $x_{0}$ is a Lebesgue point of $V_{p}(Du)$. Moreover, with $\sigma\in (0,\rr]$, recalling $\eqref{limv}_{2}$, we bound by means of \eqref{Vm.1}, \eqref{7.30}, \eqref{j.2}, \eqref{6.55.1} and \eqref{6.64},
\begin{eqnarray*}
\left(\mint_{B_{\sigma}(x_{0})}\snr{Du-(Du)_{B_{\sigma}(x_{0})}}^{p} \ \dx\right)^{1/p}&\le&c\left(\mint_{B_{\sigma}(x_{0})}\snr{V_{p}(Du)-V_{p}((Du)_{B_{\sigma}(x_{0})})}^{p}\snr{(Du)_{B_{\sigma}(x_{0})}}^{p(2-p)/2} \ \dx\right)^{1/p}\nonumber \\
&&+c\left(\mint_{B_{\sigma}(x_{0})}\snr{V_{p}(Du)-V_{p}((Du)_{B_{\sigma}(x_{0})})}^{2} \ \dx\right)^{1/p}\nonumber \\
&\le&c\mf{F}(u;B_{\sigma}(x_{0}))^{2/p}+c\snr{(Du)_{B_{\sigma}(x_{0})}}^{(2-p)/2}\mf{F}(u;B_{\sigma}(x_{0}))\nonumber \\
&\le&c\mf{F}(u;B_{\sigma}(x_{0}))^{2/p}+c(1+M)^{(2-p)/p}\mf{F}(u;B_{\sigma}(x_{0}))\to 0
\end{eqnarray*}
with $c\equiv c(n,N,p)$ and $x_{0}$ is a Lebesgue point of $Du$ as well. On the other hand, if $x_{0}$ is a Lebesgue point of $V_{p}(Du)$ we know that $\mf{F}(u;B_{\sigma}(x_{0}))\to 0$ and that $\eqref{limv}_{1}$ exists, therefore, recalling that \eqref{1.3} is in force, we can fix $\rr$ so small that $\eqref{1.4}_{2}$ holds and set $M:=2\limsup_{\sigma\to 0}\snr{(V_{p}(Du))_{B_{\sigma}(x_{0})}}+1$ to verify also $\eqref{1.4}_{1}$. Finally, if $x_{0}$ is a Lebesgue point of $Du$, then
\eqn{7.36}
$$
\left(\mint_{B_{\sigma}(x_{0})}\snr{Du-(Du)_{B_{\sigma}(x_{0})}}^{p} \ \dx\right)^{1/p}\to 0,\qquad \qquad \limsup_{\sigma\to 0}\snr{(Du)_{B_{\sigma}(x_{0})}}<\infty
$$
and $\eqref{limv}_{2}$ exists. Since
\begin{flalign}\label{7.38}
\snr{(V_{p}(Du))_{B_{\sigma}(x_{0})}}\le \mf{J}_{p}(Du;B_{\sigma}(x_{0}))^{p/2}\le c \left(\mint_{B_{\sigma}(x_{0})}\snr{Du-(Du)_{B_{\sigma}(x_{0})}}^{p} \ \dx\right)^{1/2}+c\snr{(Du)_{B_{\sigma}(x_{0})}}^{p/2},
\end{flalign}
with $c\equiv c(p)$ and, via triangular inequality,
\begin{eqnarray}\label{7.37}
\mf{F}(u;B_{\sigma}(x_{0}))&\stackrel{\eqref{7.30}}{\le}&c\left(\mint_{B_{\sigma}(x_{0})}\snr{V_{p}(Du)-V_{p}((Du)_{B_{\sigma}(x_{0})})}^{2} \ \dx\right)^{1/2}\nonumber \\
&\stackrel{\eqref{Vm}_{2}}{\le}&c\left(\mint_{B_{\sigma}(x_{0})}(\snr{Du}+\snr{(Du)_{B_{\sigma}(x_{0})}})^{(p-2)/2}\snr{Du-(Du)_{B_{\sigma}(x_{0})}}^{2} \ \dx\right)^{1/2}\nonumber \\
&\le&c\left(\mint_{B_{\sigma}(x_{0})}\snr{Du-(Du)_{B_{\sigma}(x_{0})}}^{p} \ \dx\right)^{1/2}\stackrel{\eqref{7.36}}{\to}0,
\end{eqnarray}
for $c\equiv c(n,N,p)$, keeping \eqref{1.3} and \eqref{7.37} in mind we can choose $\rr$ so small that $\eqref{1.4}_{2}$ holds true, and setting $M:=c+2c\limsup_{\sigma\to 0}\snr{(Du)_{B_{\sigma}(x_{0})}}^{p/2}$ where $c\equiv c(p)$ is the constant appearing in \eqref{7.38}, we obtain also $\eqref{1.4}_{1}$ and the proof is complete.
\end{proof}
Next, we prove Theorem \ref{t.v.0}.
\begin{proof}[Proof of Theorem \ref{t.v.0}]
Since our results are local in nature, we can assume that \eqref{1.2} holds globally in $\Omega$ - notice that being \eqref{1.2} in force, we can always assume the validity of \eqref{f.p.m}. Let $\mathcal{R}_{u}$ be the set defined in \eqref{ru.0} with $\bar{\varepsilon}\equiv \ti{\varepsilon}$, $\bar{\rr}\equiv \ti{\rr}$ and $\ti{\varepsilon}$, $\ti{\rr}$ defined in the proof of Theorem \ref{t.v}. The discussion at the beginning of Section \ref{morsec}, see also \cite[Section 5.1]{deqc}, yields that $\mathcal{R}_{u}$ is an open set of full $n$-dimensional Lebesgue measure and $\snr{\Omega\setminus \mathcal{R}_{u}}=0$ therefore given any $x_{0}\in \mathcal{R}_{u}$ with the specifics described before there is an open neighborhood $B(x_{0})$ of $x_{0}$ and a positive radius $\rr_{x_{0}}\in (0,\ti{\rr}]$ such that $\snr{(V_{p}(Du))_{B_{\rr_{x_{0}}}(x)}}<M$ and $\mf{F}(u;B_{\rr_{x_{0}}}(x))<\ti{\varepsilon}$. Given \eqref{1.2} and our choice of $\bar{\rr}$, $\bar{\varepsilon}$ we see that $\eqref{1.4}_{2}$ holds on $B(x_{0})$, Corollary \ref{cor.ex}, Theorem \ref{t.v} and Proposition \ref{vmo.t} apply, the limits in \eqref{limv} exist and define the precise representative of $V_{p}(Du)$ and of $Du$ at all $x\in B(x_{0})$. With these informations at hand, we aim to prove that the limits in \eqref{limv} are uniform in the sense that the continuous maps $B(x_{0})\ni x\mapsto (V_{p}(Du))_{B_{\sigma}(x)}$, $B(x_{0})\ni x\mapsto (Du)_{B_{\sigma}(x)}$ with $\sigma\in (0,\rr_{x_{0}}]$ uniformly converge to $V_{p}(Du(x))$ and to $Du(x)$ respectively as $\sigma\to 0$ thus yielding that $V_{p}(Du)$ and $Du$ are continuous on $B(x_{0})$. This is a consequence of the two inequalities in \eqref{1.7} as their right-hand side uniformly converges to zero by means of \eqref{1.2}, \eqref{6.63}, \eqref{6.55} and \eqref{j.2}. The proof is complete. 
\end{proof}

\subsection{Optimal function space criteria and proof of Theorem \ref{t.v.1}} This final section is devoted to the proof of Theorem \ref{t.v.1}. Once noticed that
$$
\left\{
\begin{array}{c}
\displaystyle
\ f\in L(n,1) \ \Longrightarrow \ \mathbf{I}^{f}_{1,m}(x_{0},s)\to_{s\to 0} \ \ \mbox{uniformly in} \ \ x_{0}\in \Omega \\[8pt]\displaystyle 
\ f\in L^{d}, \ \ d>n \ \Longrightarrow \ \mathbf{I}^{f}_{1,m}(x_{0},s)\le \frac{ds^{1-n/d}\nr{f}_{L^{d}}}{(d-n)\omega_{n}^{1/d}},
\end{array}
\right.
$$
cf. \cite[Section 2.3]{kumi1} and \cite[Section 6.5]{deqc} respectively, keeping in mind \eqref{6.55}, the proof goes exactly as in \cite[Proof of Theorem 2]{deqc}.

\end{document}